\documentclass[10pt,a4paper,reqno]{amsart}

\usepackage{amsmath}
\usepackage{amsfonts}
\usepackage{amssymb}
\usepackage{amsthm}
\usepackage{mathrsfs}
\usepackage[shortlabels]{enumitem}
\usepackage{graphicx}
\usepackage{color}
\usepackage{dsfont}
\usepackage[latin1]{inputenc}
\usepackage{MnSymbol}
\usepackage{enumitem}
\usepackage{extpfeil}
\usepackage{mathtools}
\usepackage{url}
\usepackage[margin=1.25in]{geometry}
\usepackage{fancyhdr}
\usepackage[all,cmtip]{xy}
\usepackage{hyperref}
\usepackage{thm-restate}
\usepackage{cleveref}

\usepackage{pgf,tikz}
\usetikzlibrary{arrows}
\usepackage{tikz-layers}
\usetikzlibrary{positioning}
\usetikzlibrary{patterns.meta}
\usetikzlibrary{patterns.meta}
\tikzstyle{stuff_fill}=[rectangle,fill=white,minimum size=1em]
\usepackage{pgfplots}
\pgfplotsset{compat = newest}

\numberwithin{equation}{section}

\newcommand{\R}{{\mathds R}}
\newcommand{\N}{{\mathds N}}
\newcommand{\Z}{{\mathds Z}}

\def\gr{{\rm gr}}

\newcommand{\eps}{\varepsilon}
\def \hatphi {{\hat{\phi}}}
\def \mult {{\cdot}}

\makeatletter
\newcommand{\AlignRight}[1]{\ifmeasuring@#1\else\omit\hfill$\displaystyle#1$\fi\ignorespaces}
\makeatother
\def\acts{\curvearrowright}

\theoremstyle{plain}
\newtheorem{theorem}{Theorem}[section]

\newtheorem{conjecture}[theorem]{Conjecture}
\newtheorem{corollary}[theorem]{Corollary}

\newtheorem{lemma}[theorem]{Lemma}
\newtheorem{proposition}[theorem]{Proposition}
\newtheorem{question}[theorem]{Question}
\newtheorem*{question*}{Question}

\newtheorem{problem}[theorem]{Problem}
\theoremstyle{definition}
\newtheorem{remark}[theorem]{Remark}
\newtheorem*{acknowledgements*}{Acknowledgements}

\newtheorem{definition}[theorem]{Definition}
\newtheorem*{notation*}{Notation}
\newtheorem*{convention*}{Convention}

\title{$L^p$ measure equivalence of nilpotent groups}
\author{Thiebout Delabie}
\email{thiebout.delabie@gmail.com }
\author{Claudio Llosa Isenrich}
\address{Faculty of Mathematics, Karlsruhe Institute of Technology, Englerstr. 2, 76131 Karlsruhe, Germany}
\email{claudio.llosa@kit.edu}

\author{Romain Tessera}
\address{Universit\'e Paris Cit\'e, Sorbonne Universit\'e, Institut de Math\'ematiques de Jussieu-Paris Rive Gauche, 75013 Paris, France}
\email{romatessera@gmail.com}


\begin{document}

\begin{abstract}
We classify compactly generated locally compact groups of polynomial growth up to $L^p$ measure equivalence (ME) for all $p\leq 1$.

To achieve this, we combine rigidity results (previously proved for discrete groups by Bowen and Austin) with new constructions of explicit orbit equivalences between simply connected nilpotent Lie groups. In particular, we prove that for every pair of simply connected nilpotent Lie groups there is an $L^p$ orbit equivalence for some $p>0$, where we can choose $p>1$ if and only if the groups have isomorphic asymptotic cones. We also prove analogous results for lattices in simply connected nilpotent Lie groups. This yields a strong converse of Austin's Theorem that two nilpotent groups which are $L^1$ ME have isomorphic Carnot graded groups. 

We also address the much harder problem of extending this classification to $L^p$ ME for $p>1$: we obtain the first rigidity results, providing examples of nilpotent groups with isomorphic Carnot graded groups (hence $L^1$ OE) which are not $L^p$ ME for some finite (explicit) $p$. For this we introduce a new technique, which consists of combining induction of cohomology and scaling limits via the use of a theorem of Cantrell.

Finally, in the appendix, we  extend theorems of Bowen, Austin and Cantrell on $L^1$ ME to locally compact groups.
\end{abstract}

\setcounter{tocdepth}{1}
\maketitle
 
\tableofcontents
\section{Introduction}

In this work we conduct a systematic study of $L^p$ measure equivalence, and the slightly stronger relation of $L^p$ orbit equivalence, between compactly generated locally compact groups with polynomial growth. Being $L^p$ measure (resp. orbit) equivalent becomes more restrictive as 
$p$ grows, and defines an equivalence relation only for $p=0$, or for $p\geq 1$ (see \cite{DKLMT-22}). The case $p=0$ corresponds to measure (resp. orbit) equivalence, which by Ornstein-Weiss's theorem \cite{OrnWei-80} is reduced to a single equivalence class for countable amenable groups (see \cite{KoiKyeRau-21} for a generalisation to amenable unimodular locally compact groups). At the other end of the spectrum, being quasi-isometric is equivalent to uniform measure equivalence, a notion which lies between $L^\infty$ measure equivalence and $L^{\infty}$ orbit equivalence: this was first observed by Shalom, based on a result of Gromov \cite{Sha-04,Sau-06}. In the same vein, $L^{\infty}$ orbit equivalence between two amenable finitely generated groups is equivalent to the existence of a bijective quasi-isometry between them.  Versions of these statements in the context of  amenable unimodular compactly generated locally compact groups were later proved in \cite{BR-18,KoiKyeRau-21}. 

It emerges from this discussion that studying $L^p$ measure equivalence among amenable groups extends and in a sense quantifies their study up to quasi-isometry. We therefore start this introduction by recalling some known facts about the quasi-isometry classification of groups with polynomial growth. 

The reader only interested in our main results can jump directly to Section \ref{secIntro:flexib}.

\begin{notation*}
    To shorten notation, we will write ME for measure equivalence and OE for orbit equivalence throughout the paper.
\end{notation*}

\subsection{The quasi-isometric classification}
A celebrated theorem of Gromov states that a finitely generated group with at most polynomial growth is virtually nilpotent \cite{Gro-81}. Conversely, a result independently established by Bass and Guivarc'h says that a nilpotent finitely generated group has volume growth exactly polynomial of degree a certain integer $d$ which can be characterized algebraically \cite{Bass-72, Guivarch73}. Recall that any torsion-free  finitely generated  nilpotent group canonically embeds as a uniform lattice in a simply connected nilpotent Lie group, called its Malcev completion \cite{Malcev-51}. Since  finitely generated nilpotent groups are virtually torsion-free,  this fact implies that, as far as the large scale geometry is concerned, the study of finitely generated groups with polynomial growth reduces to that of simply connected nilpotent Lie groups. 
This was generalised to the locally compact setting by Losert \cite{Losert-20}: he proves that any compactly generated locally compact group with (at most) polynomial growth $G$ admits a continuous proper morphism with cocompact image to a group of the form $N\rtimes Q$, where $N$ is a simply connected nilpotent Lie group, and $Q$ is a compact group. The group $N$ is uniquely determined by $G$, and for this reason we will call it the Losert-Malcev completion of $G$.

To summarize, Losert's result implies that any compactly generated locally compact group with polynomial growth is quasi-isometric (hence $L^{\infty}$ ME) to a simply connected nilpotent Lie group. A key advantage of this reduction is that the algebraic structure of a simply connected nilpotent Lie group is completely encoded in its Lie algebra, which is very often  easier to deal with.  Note also that there are uncountably many isomorphism classes of simply connected nilpotent groups, even in nilpotency class 2 \cite[Remark 2.14]{Rag-72}, implying that not every simply connected nilpotent group contains a lattice. 

We briefly pause here to observe that this allows us to reduce our study of $L^p$ ME to the subclass of simply connected nilpotent Lie groups. 
Indeed, we recall that by \cite[Section A.1.3]{BadFurSau-13} and \cite[Proposition 2.26]{DKLMT-22}, if $G$ and $G'$ are $L^{\infty}$ ME (for instance if they are quasi-isometric), and if $G$ and $H$ are $L^p$ ME for some $p>0$, then $G'$ and $H$ are also $L^p$ ME (a similar statement holds for OE).

The following conjecture would imply that the Losert-Malcev completion exactly captures the large-scale geometry (see \cite[Conjecture 19.114]{Cor-18} and \cite{FarMos-00}).
\begin{conjecture}\label{conj:QI}
  Two simply connected nilpotent Lie groups are quasi-isometric if and only if they are isomorphic.
\end{conjecture}
In other words, two compactly generated locally compact groups with polynomial growth are quasi-isometric if and only if their Losert-Malcev completions are isomorphic.

The first evidence towards this conjecture is due to Pansu. In order to state it precisely we need to recall the notions of a Carnot graded Lie algebra (resp.\ a Carnot graded nilpotent Lie group).

\subsection{From Pansu to Austin}

The \emph{lower central series} of a group $G$ (resp. a Lie algebra $\mathfrak{g}$) is the sequence of characteristic subgroups (resp. sub Lie algebras) inductively defined by $\gamma_1(G)=G$ and $\gamma_{i+1}(G)=[G,\gamma_i(G)]$ for $i\geq 1$ (resp. $\gamma_1(\mathfrak{g})=\mathfrak{g}$ and $\gamma_{i+1}(\mathfrak{g})=[\mathfrak{g},\gamma_i(\mathfrak{g})]$ for $i\geq 1$). The group $G$ (resp. Lie algebra $\mathfrak{g})$ is called \emph{$k$-nilpotent} if $\gamma_k(G)\neq \left\{1\right\}$ and $\gamma_{k+1}(G)=\left\{1\right\}$ (resp. $\gamma_k(\mathfrak{g})\neq \left\{0\right\}$ and $\gamma_{k+1}(\mathfrak{g})=\left\{0\right\}$). 

 The lower central series gives rise to a filtration of $\mathfrak{g}$ in the sense that $[\gamma_{i} \mathfrak{g},\gamma_{j} \mathfrak{g}]\subset \gamma_{i+j} \mathfrak{g}$. 
 A Lie algebra is called \emph{Carnot gradable} if this filtration comes from a grading, i.e.\ a decomposition $\mathfrak{g}=\bigoplus_i m_i$ satisfying $\gamma_{j} \mathfrak{g}=\bigoplus_{i\geqslant j} m_i$ and $[m_i,m_j]\subset m_{i+j}$; such a grading is called a Carnot grading.
It is always possible to associate a Carnot graded Lie algebra $\mathsf{gr}(\mathfrak{g})$ to any nilpotent Lie algebra $\mathfrak{g}$ by letting 
$\mathsf{gr}(\mathfrak{g})=\bigoplus_{i\geqslant 1} m_i$ for $m_i=\gamma_{i} \mathfrak{g}/\gamma_{i+1} \mathfrak{g}$ and defining the Lie bracket in the obvious way to make $m_i$ a grading.
We denote $\mathsf{gr}(G)$ the simply connected nilpotent Lie group whose Lie algebra is $\mathsf{gr}(\mathfrak{g})$ and we denote $[\cdot,\cdot]_{\gr}$ the Lie bracket on $\gr(\mathfrak{g})$. The pair $(\mathsf{gr}(G), m_1)$ is then called a Carnot-graded group\footnote{Some authors say stratified group.}. 
Observe that $\mathsf{gr}(G)$ has the same dimension, step, and growth exponent as $G$. 

In the sequel, we call the Carnot gradable simply connected Lie group associated to the Losert-Malcev completion of a compactly generated locally compact group with polynomial growth its \emph{associated Carnot group}.

We can now state Pansu's Theorem.

\begin{theorem}[{\cite{PanCBN,Breu-14}} and {\cite{PansuCCqi}}] 
Let $G$ be a simply connected nilpotent Lie group, equipped with a left-invariant word metric $d$ associated to some compact generating subset. 
Then $(G,d/n)$ converges in the Gromov-Hausdorff topology to $\mathsf{gr}(G)$ equipped with a left-invariant sub-Finsler metric $d_c$ as $n\to \infty$.
Moreover, if two simply connected nilpotent Lie groups $G$ and $G'$ have bilipschitz asymptotic cones (e.g.\ if they are quasi-isometric), then $\mathsf{gr}(G)$ and $\mathsf{gr}(G')$ are isomorphic.
\label{thm:Pansu}
\end{theorem}
This reduces the conjecture to the task of distinguishing up to quasi-isometry between groups with the same associated Carnot.
We will address the related issue in $L^p$ ME in Sections \ref{secintro:pYves} and \ref{secintro:obstruction}. For now, let us come back to quantitative measure equivalence, with the following result, which is due to Austin for finitely generated groups \cite{Aus-16} (see \Cref{sec:AustinBowenLC} for the general case). 
 
\begin{theorem}[Austin]\label{thm:Austin}
Let $G$ and $H$ be two compactly generated locally compact groups with polynomial growth. If they are $L^1$ ME, then they have bilipschitz asymptotic cones. 
\end{theorem}

Combined with Pansu's theorem, Austin's result implies that two locally compact compactly generated groups of polynomial growth which are $L^1$ ME have isomorphic associated Carnot.

Here, we prove that the converse holds in a strong sense.

\subsection{A converse of Austin's theorem}\label{secIntro:flexib}
Our main result is the following theorem.
\begin{theorem}\label{mainthm:Carnot} Two simply connected nilpotent Lie groups, or two virtually nilpotent finitely generated groups, with isomorphic associated Carnot are $L^p$ OE for some $p>1$.
 \end{theorem}
In particular, any two compactly generated locally compact groups with polynomial growth and isomorphic associated Carnot are $L^p$-ME for some $p>1$. 
Theorem \ref{mainthm:Carnot} as well as the following results are proved by constructing OE couplings between explicit ergodic pmp actions of the groups (see Section \ref{secintro:Folnertilings} for more details). In particular, this construction yields an explicit value for $p$, the sharpness of which will be discussed shortly (see Section \ref{secintro:pYves}).

This yields a complete classification of compactly generated locally compact groups with polynomial growth up to $L^1$ ME.
\begin{corollary}\label{cor:converse-Austin}
Let $\Gamma$ and $\Lambda$ be either two  simply connected nilpotent Lie groups, or two virtually nilpotent finitely generated groups. The following are equivalent.
\begin{enumerate}
    \item $\Gamma$ and $\Lambda$ are $L^1$ OE;
      \item $\Gamma$ and $\Lambda$ are $L^p$ OE for some $p>1$;
        \item $\Gamma$ and $\Lambda$ are $L^1$ ME;
         \item $\Gamma$ and $\Lambda$ are $L^p$ ME for some $p>1$;
          \item $\Gamma$ and $\Lambda$ have bilipschitz equivalent asymptotic cones;
          \item $\Gamma$ and $\Lambda$ have isomorphic associated Carnot-graded groups.
\end{enumerate}
\end{corollary}

We now focus on $L^p$ ME for $p<1$. In what follows, $L^{<p}$ ME means that there exists a single ME coupling which is $L^q$-integrable for all $q<p$.
\begin{theorem}\label{mainthm:growth}
Let $G$ and $G'$ be two simply connected nilpotent Lie groups, or two finitely generated virtually nilpotent groups of degree of growth respectively $m$ and $n$, with $m\leq n$. Then $G$ and $G'$  are $L^{<m/n}$ OE.
\end{theorem}
This result vastly generalizes \cite[Theorems 1.9 and 1.11]{DKLMT-22}, which treats the cases of $\Z^d$ and the discrete Heisenberg group. Focusing on the case $m=n$, one gets the following characterization. 
 
\begin{corollary}
Let $G$ and $G'$ be either two finitely generated virtually nilpotent groups or two simply connected nilpotent Lie groups. The following are equivalent.
\begin{enumerate}
\item $G$ and $G'$ are $L^p$ OE for all $p<1$;
    \item $G$ and $G'$are $L^{<1}$ OE;
    \item $G$ and $G'$ are $L^p$ ME for all $p<1$;
    \item $G$ and $G'$are $L^{<1}$ ME;
      \item $G$ and $G'$ have same degree of growth.
 \end{enumerate}   
\end{corollary}
This implies in particular that any compactly generated locally compact group with polynomial growth of degree $d$ is $L^{<1}$ ME to $\Z^d$ (or $\R^d$).
\subsection{$L^p$ ME classification of groups of polynomial growth for $p\leq 1$}
The problem of classifying groups up to $L^p$ measure equivalence is a natural refinement of the classification up to measure equivalence. Its origins can be traced back to the work of Bader, Sauer and Furman \cite{BadFurSau-13} on $L^1$ ME rigidity of real hyperbolic lattices. As already pointed out, the classification of compactly generated locally compact Polish groups with polynomial growth up to $L^p$ ME reduces to that of simply connected nilpotent Lie groups. 

The conclusion of Theorem \ref{mainthm:growth} turns out to be sharp in the sense that if $n<m$, then the groups are not $L^{n/m}$ ME. 
Indeed, it is proved in \cite{DKLMT-22} that if $G$ and $H$ are discrete, then they are not $L^p$ ME for any $p>n/m$; this generalises to locally compact groups (see Theorem \ref{thm:BowenAsym} in our appendix). Subsequently, Correia \cite{Correia-24} found a clever trick to prove that $G$ and $H$ are in fact not $L^{n/m}$ ME; once again \cite{Correia-24} only treats finitely generated groups, but there is a generalisation to locally compact groups, which will be part a forthcoming work of Correia and Paucar \cite{CorPau-25}. As a direct consequence of this discussion and the results of the previous sections, we thus obtain the following complete classification for $p\leq 1$ (and some partial information for $p\geq 1$). 

\begin{theorem}\label{thm:classification-Lp}
    Let $G$ and $H$ be compactly generated locally compact groups with polynomial growth of degrees $n$ and $m$ such that $n\leq m$. Then the interval $I$ of values of $p$ such that $G$ and $H$ are $L^p$-ME is as follows:
    \begin{itemize}
        \item if $G$ and $H$ have non-isomorphic Carnot, then $I=[0,n/m)$;
        \item if $G$ and $H$ have isomorphic Carnot, then $I=[0,p_0)$ or $[0,p_0]$ for some (possibly infinite) $p_0>1$.
    \end{itemize}
  
\end{theorem}
We will discuss what we know about $p_0$ in more detail in Sections \ref{secintro:pYves} (lower bound) and \ref{secintro:Folnertilings} (upper bound). For now we only mention that determining the precise value of $p_0$ is an interesting problem. It is currently out of reach with our methods, except when the groups have isomorphic Losert-Malcev completions, in which case $I=[0,\infty]$. A proof that $I$ is infinite and closed precisely when the groups have isomorphic Losert-Malcev completions would settle \Cref{conj:QI}.
More generally, one can ask if $I$ is closed or half-open for groups with the same Carnot when $p_0<\infty$; we currently don't have a single example for which we can answer this question.

\subsection{SBE and $L^p$ ME: two ways of interpolating between quasi-isometries and bilipschitz asymptotic cones}\label{secintro:pYves}

Finding quasi-isometry invariants that distinguish simply connected nilpotent Lie groups with isomorphic associated Carnot graded groups turns out to be a challenging problem.
The main known invariants are the real cohomology algebra by the aforementioned works of Shalom \cite{Sha-04} and Sauer \cite{Sau-06} and their Dehn functions \cite{LIPT-23,GMLIP-23}. 
These results show in particular that non-quasi-isometric nilpotent groups might have bilipschitz asymptotic cones. It is natural to try to interpolate between these two situations. 
Recall that quasi-isometries obviously induce  bilipschitz maps between the asymptotic cones: this comes from the fact that the additive constant vanishes in the rescaling procedure. 
Cornulier observed that the same holds if we allow this additive error to grow sublinearly with the distance to the origin, defining the notion of sublinear bilipschitz equivalence \cite{Cor-19}. 
\begin{definition}
    A map $f: (X,d_X) \to (Y,d_Y)$ between metric spaces is called a \emph{($v$-)sublinear bilipschitz equivalence} (short: ($v$-)SBE), if there are a non-decreasing, \emph{sublinear} map $v:\mathbb{R}_{>0}\to \mathbb{R}_{>0}$ (i.e. $\lim_{t\to \infty} v(t)/t=0$), base points $x_0\in X$, $y_0\in Y$, and a constant $M\geq 1$ such that the following hold:
    \begin{itemize}
        \item[(i)] For all $r\geq 0$ and all $x,~x'\in B(x_0,r)$
        \[
            \frac{1}{M}\cdot d_X(x,x')-v(r) \leq d_Y(f(x),f(x'))\leq M \cdot d_X(x,x') + v(r).
        \]
        \item[(ii)] For all $y\in B(y_0,r)$ there exists $x\in X$ such that $d(f(x),y)\leq v(r)$.
    \end{itemize}
\end{definition}
In \cite{Cor-19}, Cornulier shows the existence of an explicit number $e_G\in [0,1)$, defined in terms of the Lie algebra of a simply connected nilpotent $G$, such that $G$ and $\gr(G)$ are $r^{e_G}$-SBE. Conversely, the authors proved in joint works with Pallier and Garc\'ia-Mej\'ia \cite{LIPT-23, GMLIP-23} that there are large classes of simply connected nilpotent groups where the set of $t$ such that there is an $r^t$-SBE between $G$ and $\gr(G)$ is bounded below by some positive number, providing quantitative obstructions to the existence of a quasi-isometry. 

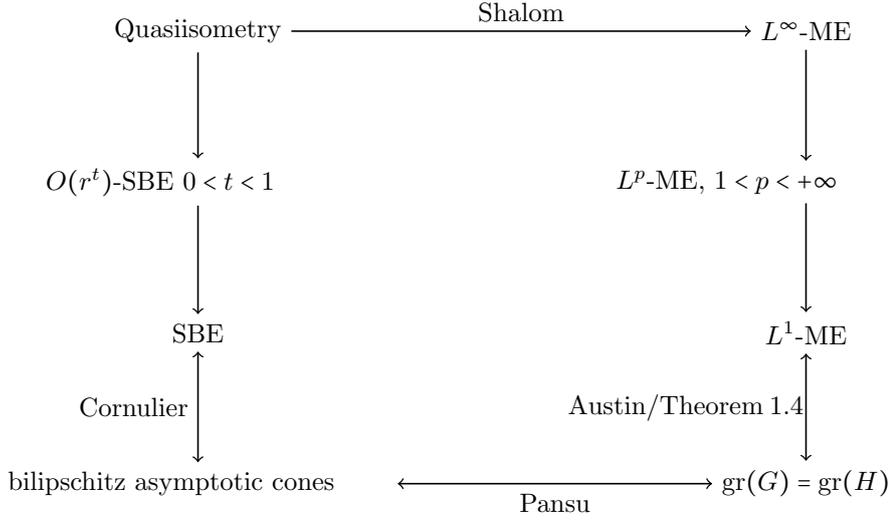
\begin{figure}
\begin{center}
\begin {tikzpicture}[-latex ,auto ,node distance =2 cm and 8cm ,on grid ,
semithick , state/.style ={ rectangle ,
minimum width =1 cm}]

\node[state] (QI) {Quasiisometry 
};
\node[state] (UME) [right=of QI] {$L^\infty$-ME 
};
\node[state] (E) [below=of QI] {\begin{minipage}{4cm} $O(r^t)$-SBE  
$0<t<1$ 
\end{minipage}
};
\node[state] (LME) [below=of UME] {\begin{minipage}{5cm} $L^p$-ME, 
$1 < p < +\infty$
\end{minipage}};
\node[state] (IME) [below=of LME] {$L^1$-ME 
};
\node[state] (SBE) [below left=of LME] {SBE
};
\node[state] (GR) [below=of IME] {$\operatorname{gr}(G)= \operatorname{gr}(H)$};
\node[state] (CE) [below=of SBE] {\begin{minipage}{5cm}
bilipschitz asymptotic cones \end{minipage}};

\path (QI) edge [->,bend left =0] node[above]{Shalom
} (UME);
\path (QI) edge [->,bend left =0] (E);
\path (E) edge [->,bend left =0] (SBE);
\path (UME) edge [->,bend left =0] (LME);
\path (LME) edge [->,bend left =0] (IME);
\path (IME) edge [<->,bend left =0] node[left]{\begin{minipage}{3cm} 
Austin/\Cref{mainthm:Carnot} \end{minipage}} (GR);
\path (CE) edge [<->,bend left =0] node[left]{Cornulier
} (SBE);
\path (GR) edge [<->,bend left =0] node[below]{Pansu
} node[below]{} (CE);
\end{tikzpicture}
\end{center}
    \caption{Relations between simply connected nilpotent $G$ and $H$.}
    \label{fig:releq}
\end{figure}

Austin's result suggests an alternative, ergodic theoretic way of interpolating between being quasi-isometric and having bilipschitz asymptotic cones for nilpotent groups (see \Cref{fig:releq} for an illustration of the relations between the two approaches). Define $p_{ME}(G,H)$ and $p_{OE}(G,H)$ to be the supremum over all $p\geq 0$ for which two given nilpotent groups $G$ and $H$ are $L^p$ ME, resp.\  $L^p$ OE. Note that by Theorem \ref{mainthm:growth}, $p_{OE}(G,H)>0$ and by Theorem \ref{mainthm:Carnot}, $p_{OE}(G,H)>1$ if $G$ and $H$ have isomorphic Carnot. Obviously, $p_{ME}(G,H)\geq p_{OE}(G,H)$, and we suspect these quantities should be equal. 
This bring us to the following problem.

\begin{problem}
    Compute or estimate $p_{ME}(G,H)$ for all pairs of non-isomorphic simply connected nilpotent Lie groups $G$ and $H$.
    Do we always have $p_{ME}(G,H)<\infty$? 
\end{problem}
\Cref{cor:converse-Austin} provides a full answer when $\gr(G)$ and $\gr(H)$ are not isomorphic.
Clearly, the real challenge concerns the case when $\gr(G)$ and $\gr(H)$ are isomorphic. Indeed, a positive answer to the question above would not only prove the quasi-isometry conjecture, but actually refine it in a quite interesting way. 
Our first contribution to this problem is the following quantitative version of Theorem \ref{mainthm:Carnot}, of which \Cref{mainthm:Carnot} for simply connected nilpotent groups is a direct consequence.

\begin{theorem}\label{thmintro:quantitativeMain}
 Let $G$ be a simply connected nilpotent Lie group of step $s_G$. Then $p_{OE}(G,\gr(G))\geq \frac{s_G}{s_G-(1-e_G)}$.
 If, moreover, the difference between $\mathfrak{g}$ and $\gr(\mathfrak{g})$ lies in the $s_G$-th term of the lower central series, then this bound can be improved to $p_{OE}(G,\gr(G))\geq \frac{1}{e_G}$.
\end{theorem}
The bound  $\frac{1}{e_G}$ is best possible considering the techniques employed in this paper, and we suspect that it should hold in general. Conjecturing the equality $p_{OE}(G,\gr(G))=\frac{1}{e_G}$ to hold in complete generality is tempting, but probably too optimistic given our current level of understanding. Just proving it in certain situations already seems challenging. 

We illustrate the bound $p_{OE}(G,\gr(g))\geq \frac{1}{e_G}$ with a concrete family of examples, which will also allow us to prove its asymptotic optimality in a special case (see \Cref{secintro:obstruction}). Recall that the model filiform group $L_m$ of dimension $m$ (and nilpotency class $m-1$) is the simply connected nilpotent Lie group whose Lie algebra is defined by the presentation
\begin{equation}\label{eqn:filiform-Lie-algebra}
    \mathfrak{l}_m=\left\langle x_1,\dots, x_m \left\mid 
    \begin{array}{ll} 
    \left[x_1,x_i\right]= x_{i+1},~ &2\leq i \leq m-1,\\ \left[x_i,x_j\right]=1,~ &\mbox {else} \end{array} \right. \right\rangle.
\end{equation}
Observe that $L_3$ is the real 3-Heisenberg group.

The centre of $L_m$, for $m\geq 3$, is isomorphic to $\mathbb{R}$, allowing us to take the central product $G_{m,n}:=L_m \times_Z L_n$. Here we define the \emph{central product} $K\times_{\theta} L$ of arbitrary groups $K$, $L$ equipped with an isomorphism $\theta: Z(K)\to Z(L)$ of their centres as the quotient $(K\times L)/\left\{(z,\theta(z))\mid z\in Z(K)\right\}$. We write $K\times _Z L$ if the central product does not depend on $\theta$ up to isomorphism. Observe that $L_3\times _Z L_3$ is the real 5-Heisenberg group. 

If $m>n\geq 3$, then a simple Lie algebra calculation shows that $\gr(G_{m,n})\cong L_{m}\times L_{n-1}\not \cong G_{m,n}$. Moreover, by \cite[Proposition 6.13]{Cor-19}, we have $e_{G_{m,n}}=\frac{n-1}{m-1}$. This implies:
\begin{proposition}\label{prop:G-m-n-OE}
    For all $m>n\geq 3$ and all $p<\frac{m-1}{n-1}$, the groups $G_{m,n}$ and $L_{m}\times L_{n-1}$ are $L^p$ OE. In other words $p_{OE}(G_{m,n},L_{m-1}\times L_n)\geq \frac{m-1}{n-1}=\frac{1}{e_{G_{m,n}}}$.
\end{proposition}

\subsection{Constructions of explicit OE-couplings}\label{secintro:Folnertilings}

To prove the existence of $L^p$ OE couplings between nilpotent groups (\Cref{mainthm:Carnot,mainthm:growth,thmintro:quantitativeMain}), we build explicit orbit equivalences from F\o lner tiling sequences. F\o lner tiling sequences were introduced under different names by Danilenko \cite{Dan-16} and by Cecchi and Cortez \cite{CecCor-19} and then, independently, under the name F\o lner tiling sequences by Delabie, Koivisto, Le Ma\^itre and Tessera in \cite{DKLMT-22}. The latter authors used them to construct explicit orbit equivalences between finitely generated amenable groups which satisfy good integrability conditions, including some very specific examples of nilpotent groups, as we mentioned above. In \Cref{sec:orbit-equivalence} we will introduce them in the more general context of unimodular locally compact second countable amenable groups and check that the required results from \cite{DKLMT-22} generalise to this situation.

Given an amenable group $G$, a F\o lner tiling sequence provides a sequence of subsets $(F_{k})$ of $G$ such that the tiles $T_k:= F_0 \cdot \dots \cdot F_{k}$ form a F\o lner sequence for $G$ and such that elements of $T_{k}$ decompose uniquely as $x_0\cdot \cdots \cdot x_k$ with $x_i\in F_i$. Left multiplication then induces a left action on $X_G:=\prod_{k\geq 0} F_k$, which preserves a standard Borel probability measure induced by equipping the $F_i$ with the normalised counting measure (if they are finite) or the normalised Haar measure (if they are precompact). The action induces the cofinite equivalence relation on $X_G$. This means that if we have a second amenable group $G'$ equipped with a F\o lner tiling sequence $(F_i')$ and there are bijections $F_i\to F_i'$, then the induced map $\phi: X_G\to X_{G'}$ defines an orbit equivalence between $G$ and $G'$. Determining the integrability of this coupling amounts to determining the integrability of the function $x\mapsto d_{S_{G'}}(\phi(s\mult_G x), \phi(x))$, where $d_{S_{G'}}$ is the word metric on $G'$ with respect to a compact generating set $S_{G'}$ for $G'$, $s\in S_G$ takes values in a compact generating set $S_G$ for $G$, and $x\in X_G$.

To apply this to nilpotent groups (simply connected Lie or finitely generated), we construct an explicit F\o lner tiling sequence for every such group in \Cref{sec:explicit-tilings-for-nilpotent-groups}. For simply connected nilpotent Lie groups, we obtain it by first decomposing elements in normal form with respect to a suitable basis of the Lie algebra that respects the lower central series, and then defining rectangular tiles in the coordinates of this normal form, scaled by the distorsion of the corresponding terms of the lower central series; in the finitely generated case we pass through the Malcev completion. Given a pair of nilpotent groups (simply connected Lie or finitely generated), we can then adapt the sizes of tiles so that the above recipe induces an orbit equivalence between them. The integrability conditions asserted in \Cref{mainthm:growth} then follow from a straight-forward argument, similar to the one employed for special cases in \cite{DKLMT-22}.

Showing that for the case of two nilpotent groups with the same Carnot there is an $L^p$ OE with $p>1$ is more subtle and requires additional work. To see this, we start by observing that given a simply connected nilpotent group $G$, a basis for its Lie algebra $\mathfrak{g}$ as above induces a vector space isomorphism $\mathfrak{g}\to \gr (\mathfrak{g})$. The difference in the group laws between $G$ and $\gr(G)$ with respect to this identification can be determined in terms of the Baker Campbell Hausdorff (BCH) formula. Cornulier \cite{Cor-19} proved that, up to a multiplicative constant that depends on the number of factors, this difference is bounded above by the word length of the factors composed with a sublinear function $r\mapsto r^{t}$, where for a suitable basis we can choose $t=e_G$. We prove \Cref{mainthm:Carnot,thmintro:quantitativeMain} for the simply connected Lie group case in \Cref{sec:length-of-commutators,sec:OE-same-Carnot,sec:OE-same-Carnot-central} by choosing compatible normal forms (and thus F\o lner tilings) for $G$ and $\gr(G)$ coming from the identification of their Lie algebras and then carefully exploiting Cornulier's sublinear bound. Finally, in \Cref{sec:finitely-generated-same-Carnot} we treat the finitely generated case of \Cref{mainthm:Carnot,thmintro:quantitativeMain}: after passing to a suitable torsion-free finite index subgroup and a compatible basis of the Lie algebra $\mathfrak{g}$ of the corresponding Malcev completion the arguments are similar as in the simply connected case.

\subsection{Obstruction to $L^p$ ME: inducing distorted central extensions}\label{secintro:obstruction}
So far none of the tools (based on induction of cohomology) developed by Shalom and Sauer extend to $L^p$ ME. This may seem surprising, because induction techniques are a priori available in this setting. However, the difficulty is to prove that induction induces an injective map in cohomology. This crucial part of the argument in Shalom or Sauer's works is missing under the weaker assumption of $L^p$ ME, and the whole strategy falls apart.
Here, we introduce a new method which allows to induce $2$-cocycles that define central extensions of maximal nilpotency class. 

\begin{theorem}\label{thm:extension}
Let $s\in\N$, let $G$ and $H$ be simply connected, $(s-1)$-step nilpotent Lie groups that are $L^{p}$-measure equivalent for some $p>s$ and let $\tilde{H}$ be a central extension of $H$ by $\R$ such that $\tilde{H}$ has nilpotency class $s$.
Then $G$ has a central extension $\tilde{G}$ by $\R$ of nilpotency class $s$.
\end{theorem}

To our knowledge, \Cref{thm:extension} provides the first obstruction to the existence of a $L^p$ ME between two simply connected nilpotent groups with isomorphic associated Carnot. As a consequence, we obtain the first known {\it finite} upper bounds on $p_{ME}(G,\gr(G))$ for an infinite family of (non-Carnot) simply connected nilpotent Lie groups.

\begin{corollary}\label{cor:NotLp}
	For $m>n\geq 3$ the groups $G_{m,n}$ and $L_m\times L_{n-1}$ have isomorphic associated Carnot groups, but are not $L^p$ ME for any $p>m$. In particular, $p_{ME}(G_{m,n},\gr(G_{m,n}))\leq m$.
\end{corollary}
\begin{proof}
    The group $L_{m+1}\times L_{n-1}$ is a central extension of $L_m\times L_{n-1}=\gr(G_{m,n})$ by $\R$ of nilpotency class $m$, while it follows from \cite[Lemma 7.10]{LIPT-23} that $G_{m,n}$ does not have any central extension of class $m$. Thus, the assertion follows from \Cref{thm:extension}.
\end{proof}

Finally, by combining \Cref{prop:G-m-n-OE} and \Cref{cor:NotLp} we can show that for the family $H_m:=G_{m+1,3}$, with $m\geq 3$, the lower bounds on $p_{ME}(H_m,\gr(H_m))$ provided by \Cref{thmintro:quantitativeMain} are asymptotically optimal as the nilpotency class $m$ goes to infinity.
\begin{corollary}\label{cor:asymptotic-optimality}
    For all $m\geq 3$ the $m$-nilpotent group $H_m:= G_{m+1,3}$ satisfies
    \[
        e_{H_m}=\frac{m}{2}\leq p_{ME}(H_m,\gr(H_m))\leq m+1.
    \]
    In particular, $p_{ME}(H_m,\gr(H_m))\simeq m$.
\end{corollary}

This result should be compared to \cite[Theorem V]{GMLIP-23}, which proves the asymptotic optimality of Cornulier's constant with regards to the existence of an $r^t$-SBE for the same family of groups.

Let us end this paragraph with a word on the techniques used in the proof of Theorem \ref{thm:extension}. 
The first observation is that a cohomology class $c$ representing a central extension $\tilde{H}$ of an $(s-1)$-step nilpotent Lie group $H$ can be represented by a $2$-cocycle of at most polynomial growth of degree $s$ in the word metric. Hence $L^s$-integrability is enough to induce such a $2$-cocycle $c'$ on $G$. Assuming that the central extension $\tilde{H}$ has degree $s$, we need to prove that the same holds for the central extension $\tilde{G}$ of $G$ associated to $c'$. The strategy consists of proving that $\gr({\tilde{G}})$ and $\gr({\tilde{H}})$ are isomorphic. Indeed, this implies that $\gr({\tilde{G}})$, and therefore $\tilde{G}$ have nilpotency class $s$. In order to do so, we use a result of Cantrell \cite{Cantrell}, which nicely gathers Pansu's and Austin's theorems in one statement. Roughly speaking, Cantrell's theorem says that if $\alpha$ is the cocycle associated to a $L^1$-ME coupling from $G$ to $H$, then $\alpha$ converges (in probability) to an isomorphism $\varphi:\gr({G})\to\gr({H})$ after rescaling (see Proposition \ref{prop:cantrell}). 

Assuming that the coupling is $L^p$ with $p>s$, we can then exploit Cantrell's theorem to deduce that the induced $2$-cocycle $c'$ converges (after suitably rescaling it) to a $2$-cocycle on $\gr({H})$ corresponding to a central extension of $\gr({H})$ which is isomorphic to $\gr({\tilde{G}})$ (see Lemma \ref{lem:convcocycle} for this key step).

\subsection{Structure} In \Cref{sec:background-measure-orbit-equivalence} we summarise the necessary background material on measure and orbit equivalences, and their integrability. In \Cref{sec:orbit-equivalence,sec:explicit-tilings-for-nilpotent-groups,sec:length-of-commutators,sec:OE-same-Carnot,sec:OE-same-Carnot-central,sec:finitely-generated-same-Carnot} we introduce F\o lner tiling sequences and then use them to construct explicit orbit equivalences between nilpotent groups with good integrability properties. This proves \Cref{mainthm:Carnot,mainthm:growth,thmintro:quantitativeMain} and their consequences. We refer to \Cref{secintro:Folnertilings} for a summary of the main ideas going into the proof of these results, including details on the structure of these sections. In \Cref{sec:obstruction} we constrain the existence of an $L^p$ ME between two nilpotent groups in terms of their central extensions and, as a consequence, prove \Cref{thm:extension}. In \Cref{sec:questions} we list some questions that remain open following our work. Finally, in \Cref{sec:AustinBowenLC} we prove that results by Bowen, Austin and Cantrell generalise from discrete to locally compact groups.

\subsection{Notation and conventions}
We fix some notation and conventions that we will use throughout this work. Given two functions $f$ and $g$ that take values in the positive real numbers, we write $f\lesssim g$, if there exists a constant $C\geq 1$ such that $f(a)\leq Cg(b)$ for all $a$ and $b$ in the domain; the domain can for instance be elements of a given group and the constant $C$ could depend on this group. We write $f\simeq g$ if $f\lesssim g$ and $g\lesssim f$.

For elements $x,~y\in G$ of a group $G$ we will use the convention $[x,y]:= xyx^{-1}y^{-1}$. We will further use the notation $[x_1,\cdots,x_k]:=[x_1,[\dots, [x_{k-1},x_k]\dots]]$ for the $k$-fold commutator of elements $x_1,\dots, x_k\in G$.

We will use the convention that $\frac{1}{0}=+\infty$.

For a group $G$ we denote its group law and group action by $\cdot$ or $\ast$. Since in some situations we will have different groups acting on the same space, we will sometimes also write $\cdot_G$ or $\ast_G$ to avoid any confusion.

\subsection*{Acknowledgements} We thank Corentin Correia and Gabriel Pallier for helpful comments and suggestions on an earlier version of this work. The second author gratefully acknowledges support by the DFG projects 281869850 and 541703614.

\section{Measure and orbit equivalence}\label{sec:background-measure-orbit-equivalence}
In this section we recall basic notions around measure and orbit equivalences of groups. Both are natural notions of equivalences between groups in the context of their actions on measure spaces. Measure equivalence plays a similar role in measured group theory as quasi-isometries do in geometric group theory, where groups act by isometries on metric spaces. Orbit equivalence is a stronger notion than measure equivalence, where we preserve additional information about the orbits of the action. Here we will provide the basic definitions required in our work. We refer to \cite{Gab-10, Fur-11} and the references therein for further details, with a particular focus on the countable case and to \cite{KoiKyeRau-21} for further details on the locally compact case.

Often measure and orbit equivalence are considered in the context of countable groups, but they can be defined for locally compact second countable (short: lcsc) groups, which we shall do here. We recall that an lcsc group is called \emph{unimodular} if its Haar measure is biinvariant. Examples of unimodular lcsc groups include all countable groups, nilpotent Lie groups and semisimple Lie groups. We call an action of a unimodular lcsc group $G$ on a measure space $(\Omega,\mu)$ by isomorphisms of measure spaces \emph{essentially free} if for all $g\in G$ and almost every $x\in X$ the stabiliser $\left\{g\in G\mid  g\cdot x = x\right\}$ is trivial. We further call a subset $Y\subset \Omega$ a \emph{fundamental domain} for the $G$-action on $\Omega$ if, up to sets of measure zero, $\Omega$ decomposes as disjoint union $\bigsqcup_{g\in G}g\cdot Y$. 

\begin{definition}
    Let $G$, $H$ be unimodular lcsc groups equipped with Haar measures $\lambda_G$ and $\lambda_H$. A \emph{measure equivalence coupling between $G$ and $H$} is a measure space $(\Omega,\mu)$ and finite measure spaces $(X_{G},\mu_G)$, $(X_H,\mu_H)$ together with isomorphisms $\iota \colon (G,\lambda_G)\times (X_G,\mu_G)\to (\Omega,\mu)$ and $j \colon (H,\lambda_H)\times (X_H,\mu_H)\to (\Omega,\mu)$ of measure spaces, such that the actions $G\curvearrowright (\Omega,\mu)$ and $H\curvearrowright (\Omega,\mu)$ defined by $g\cdot \iota(g',x_G)=\iota(g\cdot g',x_G)$, $h\cdot j(h',x_H)=j(h\cdot h, x_H)$ commute. We call $G$ and $H$ \emph{measure equivalent} and write $G\sim_{ME} H$, if there is a measure equivalence coupling between them.
\end{definition}

\begin{remark}
    The existence of a measure equivalence coupling between $G$ and $H$ implies the existence of commuting actions of $G$ and $H$ on an infinite measure space $(\Omega,\mu)$ so that both actions admit measurable fundamental domains of finite measure. In particular, the actions of $G$ and $H$ on $\Omega$ are essentially free. For countable groups this condition is equivalent to the existence of a measure equivalence coupling between them, see \cite{Fur-11}. In the uncountable case it is usually not possible to choose a measurable fundamental domain for an action on an arbitrary measure space $(\Omega,\mu)$. This motivates the above somewhat more technical definition of measure equivalence, which essentially amounts to making the existence of such a fundamental domain part of the assumptions.
\end{remark}

Some classes of groups are very rigid under measure equivalence. An important result in this direction due to Furman is that every countable group which is measure equivalent to a lattice in a semi-simple Lie group $G$ of real rank at least $2$ is itself commensurable to a lattice in $G$, see \cite{Fur-99} and also \cite[Example 2.2]{Fur-11}. However, in other classes measure equivalence is a much less rigid concept. For instance, a result of Ornstein and Weiss \cite{OrnWei-80} says that all infinite countable amenable groups form a single measure equivalence class. This has been extended to the class of non-compact unimodular amenable lcsc groups by Koivisto, Kyed and Raum \cite{KoiKyeRau-21}.

In this context it is natural to ask for a finer quantitative notion of measure equivalence that can for instance distinguish between different amenable groups. This can be achieved by considering an additional integrability condition on the associated cocycles. More precisely, given a measure equivalence from $G$ to $H$, we can define a map $\alpha 
 \colon G\times X_H \to  H$ almost everywhere by $\alpha(g,x)\mult_H \left(g\mult_G x\right)=: T^g(x)\in X_H$, that is, $\alpha(g,x)\in H$ is the unique element which maps $g\mult_G x\in \Omega$ to $X_H$, where we identify $X_G$ and $X_H$ with $\left\{1\right\}\times X_G\subset G\times X_G$ and $\left\{1\right\}\times X_H\subset H\times X_H$. Analogously we define $\beta \colon H\times X_G \to G$ by $\beta(h,x)\cdot_G(h\cdot_H x)=:S^h(x)\in X_G$ for $h\in H$ and $x\in X_G$. One can check that the maps $\alpha$ and $\beta$ define measurable cocycles, meaning that they are measurable and for (almost) all $g_1,g_2\in G$ and $x\in X_G$ (resp. $h_1, h_2 \in H$ and $x\in X_H$) the identities
 \begin{equation}\label{eqn:cocycle-identities}
    \alpha(g_1\mult_G g_2,x) = \alpha(g_1, T^{g_2} (x))\mult_H \alpha(g_2,x) \mbox{ (resp. } \beta(h_1\mult_H h_2,x) = \beta(h_1, S^{h_2} (x)) \mult_G \beta(h_2,x) \mbox{)}
 \end{equation}
 hold. We will call $\alpha$ and $\beta$ the \emph{cocycles associated} with the measure equivalence coupling between $G$ and $H$. Moreover, the maps $(g,x)\mapsto T^{g}(x)$ and $(h,x)\mapsto S^h(x)$ define group actions of $G$ on $X_H$ and of $H$ on $X_G$.

 \begin{definition}
     Let $\phi, \psi \colon \mathbb{R}_{\geq 0}\to \mathbb{R}_{\geq 0}$ be non-decreasing functions and let $S_G\subset G$, $S_H\subset H$ be compact generating sets for $G$ and $H$. We will say that a measure equivalence coupling from $G$ to $H$ is \emph{$(\phi, L^0)$-integrable} if for every $g\in G$ there is a $c_g>0$ such that
     \[
        \int_{X_G} \phi\left(\frac{d_{S_H}(\alpha(g,x),1)}{c_g}\right) d\mu_G(x) < +\infty.
     \]
     We will say that the measure equivalence coupling from $G$ to $H$ is \emph{$(\phi,\psi)$-integrable} if, in addition, for every $h\in H$ there is a $c_h>0$ such that
     \[
        \int_{X_H} \psi\left(\frac{d_{S_G}(\alpha(h,x),1)}{c_h}\right) d\mu_H(x) < +\infty.        
     \]

     We say that a measure equivalence coupling from $G$ to $H$ is \emph{$(L^p,L^q)$-integrable} for $p,q \in \mathbb{R}_{>0}$, if we can choose $\phi(x)=x^p$ and $\psi(x)=x^q$, and $L^p$ integrable if it is $(L^p,L^p)$-integrable.
 \end{definition}

The above definitions raise the problem of actually constructing measure equivalences between groups that satisfy non-trivial integrability conditions. In this work we will consider this problem in the context of nilpotent groups, which form a specific class of amenable groups, using an approach introduced by Delabie, Koivisto, Le Ma\^itre and Tessera  \cite{DKLMT-22}. It relies on producing explicit orbit equivalences satisfying a non-trivial integrability condition.

\begin{definition}
    Let $G, H$ be unimodular lcsc groups that admit measure preserving actions $G\curvearrowright (X_G, \mu_G)$, $H\curvearrowright (X_H,\mu_H)$ on standard probability measure spaces $(X_G,\mu_G)$ and $(X_H,\mu_H)$. We call the actions \emph{orbit equivalent} if there is an isomorphism $f: (X_G,\mu_G)\to (X_H,\mu_H)$ of measure spaces such that for almost all $x\in X_G$ we have $f(G\cdot x)=H\cdot f(x)$.

    Similar as for measure equivalence couplings, we define measurable cocycles $\alpha \colon G\times X_G \to H$ and $\beta \colon H\times X_H \to G$ associated with an orbit equivalence coupling by $\alpha(g,x)\mult_H f(x) = f(g\mult_G x)$, resp. $\beta(h,x)\mult_G f^{-1}(x) = f^{-1}(h\mult _H x)$. They satisfy cocycle identities anologous to the ones in \Cref{eqn:cocycle-identities}, and we can define the $(\phi,L^0)$-, $(\phi,\psi)$-, $(L^p,L^q)$-, and $L^p$ integrability of an orbit equivalence.
\end{definition}

Recall that a measure preserving action of a group $G$ on a probability measure space $(X,\mu)$ is called \emph{ergodic} if for every $G$-invariant set $A\subseteq X_G$ we have $\mu(A)\in \left\{0,1\right\}$. Our interest in orbit equivalences stems from the following result, see \cite[Lemma 3.11]{KoiKyeRau-21}.
\begin{lemma}
    Let $G$ and $H$ be unimodular lcsc groups that admit orbit equivalent, ergodic, essentially free, measure preserving actions on standard Borel probability spaces. Then $G$ and $H$ are measure equivalent.
\end{lemma}
There is a converse of this result under the weaker assumption that the actions are stably orbit equivalent, see \cite[Theorem A]{KoiKyeRau-21} for a precise statement in the lcsc case.

While a priori we do not seem to have gained much by passing to orbit equivalences, in \Cref{sec:orbit-equivalence} we will discuss how a method introduced by Delabie, Koivisto, Le Ma\^itre and Tessera \cite{DKLMT-22} for producing explicit orbit equivalences that satisfy good integrability conditions between countable amenable groups from so-called F\o lner tiling sequences can be generalised to obtain explicit integrable orbit equivalences between unimodular amenable lcsc groups.

\section{Explicit orbit equivalences from F\o lner tilings}\label{sec:orbit-equivalence}
In this section we will explain a method for constructing explicit orbit equivalences between groups from so-called F\o lner tiling sequences. For countable discrete groups they were studied by Danilenko \cite{Dan-16} and by Cecchi and Cortez \cite{CecCor-19} under other names. The name F\o lner tiling sequences, which we will use here, was coined by Delabie, Koivisto, Le Ma\^itre and Tessera \cite{DKLMT-22}, who used them to determine explicit integrability conditions for orbit equivalences between finitely generated groups. Since we will build orbit equivalences between simply connected nilpotent Lie groups that may not contain any finitely generated lattices, we will require them in the more general context of compactly presented groups. 

We will therefore explain how the results of \cite[Section 6.1]{DKLMT-22} extend to lcsc groups, closely following their exposition. The main change is that we replace finite sets by precompact sets of positive measure and cardinality by a Haar measure in their definitions. We then check that everything works accordingly. This provides a natural generalisation of the definitions and results in \cite[Section 6.1]{DKLMT-22} and in particular covers the case of countable discrete groups discussed there.

\begin{definition}\label{def:folner-tiling}
    Let $G$ be an amenable lcsc group equipped with a Haar measure $\mu$ and let $\left(F_k\right)$ be a sequence of subsets of $G$ such that $F_0$ is precompact, open almost everywhere and of positive measure, and $F_k$ is finite for $k\geq 1$. The associated sequence of \emph{tiles} $(T_k)$ is defined inductively by $T_0=F_0$ and $T_{k+1}=  T_k \cdot F_{k+1}$. We call $\left(F_k\right)$ a \emph{F\o lner tiling sequence} for $G$ if the following conditions hold:
    \begin{enumerate}
        \item for all $n\in \mathbb{N}$ the tile $T_{k+1}$ decomposes as a disjoint union $T_{k+1}=\bigsqcup_{g\in F_{k+1}} T_k \cdot g$;
        \item $\left(T_k\right)$ is a F\o lner sequence for $G$, that is, for all $\gamma\in G$ we have $\lim_{k\to \infty} \frac{\mu(\gamma \cdot T_k\setminus T_k)}{\mu(T_k)} =0$.
    \end{enumerate}
\end{definition}

For every F\o lner tiling sequence we can define a Borel probability space $(X=\prod_{k\in \mathbb{N}} F_k, \mu=\prod_{k\in \mathbb{N}} \mu_k)$, where $F_0$ is equipped with the normalisation $\mu_0$ of the restriction $\mu|_{F_0}$ of the Haar measure and the $F_k$ for $k\geq 1$ are equipped with the normalised counting measure $\mu_k$. Note that if $G$ is finitely generated, then $\mu|_{F_0}$ is just the normalised counting measure on $F_0$, while if $G$ is a Lie group, then $(F_0,\mu|_{F_0})$ is isomorphic to the standard Borel probability space $([0,1],\mu_{Borel})$.

To every element $x=(x_k)\in X$ we associate a sequence $(g_k(x))\in G$ of elements of $G$ defined by $g_k(x)=x_0\cdot x_1 \dots x_k\in T_k$. Condition (1) in \Cref{def:folner-tiling} guarantees that every element of $T_k$ can be uniquely written as a product $x_0\cdot x_1 \cdots x_k$ with $x_i\in F_i$. Thus, there is a bijection $\prod_{0\leq k\leq n} F_k \to T_n$, $(x_0,\cdots, x_n)\mapsto x_0 \cdots x_n$. Given an element $x=(x_k)\in X$ we will sometimes also write $f_i(x):=x_i\in F_i$ in situations where only writing $x_i$ may lead to confusion.

Since by Condition (2) the sequence $(T_k)$ is a F\o lner sequence, for every $\gamma \in G$ and almost every $x\in X$ there is an $n\in \mathbb{N}$ such that $\gamma \cdot g_n(x)\in T_n$. There is thus a unique decomposition $\gamma \cdot g_n(x) = x_0'\cdots x_n'$ with $x_i'\in F_i$. We define an action of $G$ on $X$ by $\gamma \cdot x = (x_0', \dots, x_n', x_{n+1}, x_{n+2},\dots)$; since $\gamma\cdot g_{n+1}(x)=\gamma \cdot g_n(x)\cdot x_{n+1}\in T_{n+1}$, this is independent of the choice of $n$. This means that for every $\gamma\in G$ and almost every $x\in X$ we have that for all sufficiently large $n\in \mathbb{N}$ the identity $g_n(\gamma\cdot x) = \gamma \cdot g_n(x)$ holds.

The \emph{cofinite equivalence relation} $E_{cof}$ on $X$ is defined by $x\sim y$ for $x=(x_k),y=(y_k)\in X$ if $x$ and $y$ differ only in finitely many terms.
\begin{proposition}
    Let $G$ be an amenable lcsc group and let $(F_k)$ be a F\o lner tiling sequence for $G$. Then $G$ admits an ergodic, measure preserving action on the standard Borel probability space $(X=\prod_{k\in \mathbb{N}} F_k, \mu)$, which, up to a set of measure zero, induces the cofinite equivalence relation $E_{cof}$ on $X$.
\end{proposition}
\begin{proof}
    We first prove that the action induces the cofinite equivalence relation on $X$. It follows from the definition of the $G$-action on $X$ that for every $\gamma \in G$ we have that for almost every $x\in X$ the elements $\gamma \cdot x$ and $x$ differ in only finitely many entries.  Conversely, given elements $x,y\in X$ with $x_k=y_k$ for all $k\geq n+1\in \mathbb{N}$, we define $\gamma:= g_n(x)\cdot g_n(y)^{-1}$. Then $\gamma \cdot g_n(y) = g_n(x)$ and thus $\gamma \cdot y = x$. 

    To see that the action of $G$ is measure preserving, observe that the action of $G$ on itself by left translation is continuous. Thus, given $\gamma\in G$ and $x\in {\rm int}(F_0)\times \prod_{i\geq 1} F_i$ with $\gamma \cdot g_n(x)=x_0'\cdots x_n'\in T_n$, there is an open ball $B\subset F_0\times \prod_{i\geq 1} \left\{x_i\right\}$ around $x$ such that for all $y\in B$ we have $\gamma \cdot y = y_0'x_1'\cdots x_n'x_{n+1}\cdots $ for some $y_0'\in F_0$. Since the measure on $F_0$ is the normalized restriction of the Haar measure on $G$ (thus $G$-invariant), this implies that $\mu(\gamma \cdot B)= \mu(B)$. Since every measurable set in $X$ can be approximated by covers by countably many balls of this form, this shows that the $G$-action on $X$ is measure preserving. 

    Finally, the ergodicity of the action follows from a standard argument, combining that the action induces the cofinite equivalence relation on $X$ and the Kolmogorov zero-one law. 
\end{proof}

\begin{corollary}\label{cor:existence-of-orbit-equivalence}
    Let $G$ and $G'$ be amenable lcsc groups with F\o lner tiling sequences $(F_k)$ and $(F_k')$ such that $F_k$ and $F_k'$ have the same cardinality for all $k\geq 0$.\footnote{ Note that here $F_0$ may have finite or infinite cardinality, while the cardinalities of the $F_k$ for $k\geq 1$ are finite.} Then there exists an isomorphism $f=\prod_k f_k: (X=\prod_k F_k,\mu) \to (X'=\prod_k F_k',\mu')$ of standard probability spaces and any such isomorphism induces an orbit equivalence of the actions $G\curvearrowright (X,\mu)$ and $G'\curvearrowright (X',\mu')$.
\end{corollary}
\begin{proof}
    Since standard Borel probability spaces of the same cardinality are isomorphic, our assumptions guarantee that there are isomorphisms $f_k\colon (F_k,\mu_k)\to (F_k',\mu_k')$ for all $k\geq 0$. Clearly every such isomorphism $f=\prod_k f_k$ maps equivalence classes of the cofinite equivalence relation on $X$ to equivalence classes of the cofinite equivalence relation on $X'$. Thus, $f=\prod_k f_k$ induces an orbit equivalence of the actions $G\curvearrowright (X,\mu)$ and $G'\curvearrowright (X',\mu')$.
\end{proof}

Our goal is to quantify orbit equivalences between nilpotent groups. For this we introduce the following quantified version of F\o lner tiling sequences.

\begin{definition}\label{def:quant-Folner}
    Let $G$ be an amenable lcsc group equipped with a Haar measure $\mu$ and let $(F_k)$ be a F\o lner tiling sequence for $G$. Assume that $S_G$ is a compact generating set for $G$. Let further $(\epsilon_k)$ be a non-increasing sequence in $\mathbb{R}_{>0}$ that converges to $0$, and let $(R_k)$ be a sequence in $\mathbb{R}_{\geq 0}$. 
    
    We say that $(F_k)$ is a \emph{$(\epsilon_k,R_k)$-F\o lner tiling sequence} for $G$ with respect to $S_G$ if the following hold:
    \begin{enumerate}
        \item each tile $T_k$ has diameter at most $R_k$ with respect to $d_{S_G}$; and
        \item every $s\in S_G$ satisfies $\mu(s\cdot T_k\setminus T_k)\leq \epsilon_k \mu(T_k)$.
    \end{enumerate}
\end{definition}

To determine sequences $(\epsilon_k)$ and $(R_k)$ for a given F\o lner tiling sequence $(F_k)$ it will be useful to consider the (possibly infinite) measurable distance function 
\[
    \rho : X\times X\to \mathbb{N}
\]
defined by
\[
    \rho(x,y):= \inf\left\{n\in \mathbb{N}\mid x_k=y_k \forall ~ k\geq n\right\}.
\]

Then for every $\gamma \in G$ and almost every $x\in X$ we have $\rho(\gamma\cdot x,x)\leq k$ if and only if $\gamma \cdot g_k(x)\in T_k$. This implies that

\begin{equation}\label{eqn:Folner-condition}
    \mu\left(\left\{x\mid \rho(\gamma\cdot x,x)>k\right\}\right) = \frac{\mu(\gamma \cdot T_k\setminus T_k)}{\mu(T_k)}=\frac{\mu(\gamma\cdot T_k \triangle T_k)}{2\mu(T_k)}.
\end{equation}
We also observe here that to see that \Cref{def:folner-tiling} (2) holds, it suffices to check that it holds in a uniform way on a compact generating set $S_G$ for $G$. In particular, if we check that a sequence $(F_k)$ satisfies \Cref{def:folner-tiling} (1) and \Cref{def:quant-Folner} (2), then this shows that $(F_k)$ is a F\o lner tiling sequence. 

\begin{lemma}\label{lem:quant-orbit-equivalence}
    Let $G$ be an amenable lcsc group, let $S_{G}$ be a compact generating set and let $(F_k)$ be a $(\epsilon_k,R_k)$-F\o lner tiling sequence. Then the following hold:
    \begin{enumerate}
        \item for all $s\in S_G$ and all $k\geq 0$ we have $\mu\left(\left\{x\in X\mid \rho(s\cdot x,x)>k\right\}\right)\leq \epsilon_k$; and
        \item for every $\gamma\in G$ with $d_{S_G}(\gamma,1)>2R_k$ and almost every $x\in X$ we have $\rho(\gamma\cdot x,x)>k$.
    \end{enumerate}
\end{lemma}
\begin{proof}
    Assertion (1) is an immediate consequence of \Cref{eqn:Folner-condition}. For (2) observe that the fact that the diameter of $T_k$ is at most $R_k$ together with $d_{S_G}(\gamma,1)>2R_k$ implies that $\gamma\cdot T_k \cap T_k=\emptyset$.
\end{proof}

Our main application of F\o lner tiling sequences will be to pairs $G, G'$ of simply connected nilpotent Lie groups and to finitely generated nilpotent Lie groups. We will show that for suitable real numbers $p,q\in \mathbb{R}_{>0}$ and all $\epsilon>0$ there is an $(\phi_{\epsilon},\psi_{\epsilon})$-orbit equivalence coupling for
\[
    \phi_{\epsilon}(x)=\frac{x^{p}}{\log(x)^{1+\epsilon}}, ~~ \psi_{\epsilon}(x)=\frac{x^q}{\log(x)^{1+\epsilon}}.
\]
This will imply the existence of an $(L^r,L^s)$-orbit equivalence coupling between $G$ and $G'$ for every $r<p$ and $s<q$. We end this section with an lcsc version of \cite[Proposition 6.9]{DKLMT-22}, which will be useful in the case of nilpotent groups of the same growth type, where we will show that we can choose $p=q=1$.
\begin{proposition}\label{prop:quant-orbit-equivalence}
    Let $G$ and $G'$ be amenable lcsc groups with compact generating sets $S_G$ and $S_{G'}$, let $(F_k)$ be a $(\epsilon_k,R_k)$-F\o lner tiling sequence for $G$ and $(F_k')$ a $(\epsilon'_k,R'_k)$-F\o lner tiling sequence for $G'$ such that $F_k$ and $F_k'$ have the same cardinality for all $k\geq 0$. Assume that there is a non-decreasing function $\phi \colon \mathbb{R}_{\geq 0} \to \mathbb{R}_{\geq 0}$ such that the sequence $\epsilon_{k-1}\cdot \phi(2R_k')$ is summable. 

    Then any orbit equivalence coupling $f:(X,\mu)\to (X',\mu')$ as in \Cref{cor:existence-of-orbit-equivalence} is \emph{$(\phi,L^0)$-integrable}, that is, for all $s\in S_G$ we have
    \[
        \int_X \phi\left(d_{S_{G'}}(f(s\cdot x), f(x))\right) d\mu(x) < \infty.
    \]

    If, moreover, there is a non-decreasing function $\phi'\colon \mathbb{R}_{\geq 0} \to \mathbb{R}_{\geq 0}$ such that $\epsilon_{k-1}'\cdot \phi'(2R_k)$ is summable, then $f$ is $(\phi,\phi')$-integrable.
\end{proposition}
\begin{proof}
    For all $s\in S_G$ and all $k\in \mathbb{N}$ we have
    \[
        \mu\left(\left\{x\in X \mid d_{S_{G'}}(f(s\cdot x),f(x))>2R_k'\right\}\right)\leq \mu\left(\left\{x\in X\mid \rho(s\cdot x, x)>k\right\}\right)\leq \epsilon_k,
    \]
    where for the first inequality we use \Cref{lem:quant-orbit-equivalence} (2) and the fact that $\rho(\gamma \cdot x,x)>k$ if and only if $\rho'(f(\gamma\cdot x),f(x))>k$, and for the second inequality we use \Cref{lem:quant-orbit-equivalence}.

    Since $\phi$ is non-decreasing, we deduce
    \begin{align*}
        &\int_X \phi\left(d_{S_{G'}}(f(s\cdot x), f(x))\right) d\mu(x)\\
        &\leq \phi(2R_0')\cdot \mu(X)+\sum_{k=1}^{\infty} \phi(2R_k')\cdot \mu\left(\left\{
x\in X\mid 2R_{k-1}'< d_{S_{G'}}(f(s\cdot x), f(x))\leq 2R_k' \right\}\right)\\
        &\leq \phi(2R_0') + \sum_{k=1}^{\infty} \epsilon_{k-1}\cdot \phi(2R_k') < \infty,
    \end{align*}
    where the last inequality follows from our assumption. This proves the first part. The ``moreover'' part follows analogously.
\end{proof}

\section{F\o lner tilings in simply connected nilpotent Lie groups}
\label{sec:explicit-tilings-for-nilpotent-groups}
Our main application of the results of \Cref{sec:orbit-equivalence} will be to simply connected nilpotent Lie groups. We will now construct a F\o lner tiling sequence for every such group. This will rely on the existence of suitable normal forms. Some of our arguments generalise arguments developed for finitely generated abelian groups and the integral 3-Heisenberg group in \cite[Section 6.2]{DKLMT-22}.

Let $G$ be a simply connected nilpotent Lie group of nilpotency class $n$ and let $\mathfrak{g}$ be its Lie algebra. Following \cite{Osi-01}, we will call a basis $\left\{X_{ij}\mid 1\leq i \leq s_G,~ 1\leq j \leq m_i\right\}$ canonical if $\left\{X_{i1}, \cdots, X_{im_i}\right\}\subset \gamma_i(\mathfrak{g})\setminus \gamma_{i+1}(\mathfrak{g})$ defines a basis for the $i$-th quotient $\gamma_i(\mathfrak{g})/\gamma_{i+1}(\mathfrak{g})$ of the lower central series $(\gamma_i(\mathfrak{g}))$ of $\mathfrak{g}$. We denote by $x_{ij}^{\lambda_{ij}}:= exp(\lambda_{ij}X_{ij})$ the image of $\lambda_{ij}X_{ij}$ under the exponential map $exp \colon \mathfrak{g}\to G$ for $\lambda_{ij}\in \mathbb{R}$. The fact that the latter is an isomorphism together with the definition of a canonical basis implies the following result, see \cite[Theorem 4.3]{Osi-01}. 

\begin{proposition}\label{prop:normal-form}
    Let $G$ be a simply connected nilpotent Lie group of nilpotency class $s_G$ and let $\left\{X_{ij}\mid 1\leq i \leq s_G,~ 1\leq j \leq m_i\right\}$ be a canonical basis. Then every element $g\in G$ admits a decomposition of the form
    \[
        g= \prod_{i=1}^{s_G}\prod_{j=1}^{m_i} x_{ij}^{\lambda_{ij}(g)},
    \]
    where the $\lambda_{ij}(g)\in \mathbb{R}$ are uniquely determined by $g$ and $|g|_G\simeq \sum_{i,j} |\lambda_{ij}(g)|^{\frac{1}{i}}$
\end{proposition}

Let $G$ be a simply connected nilpotent Lie group of nilpotency class $n$, and let $\Gamma$ be a lattice in $G$. We define $\overline{\gamma}_k(\Gamma)=\Gamma\cap \gamma_k(G).$ This defines a nested sequence of normal subgroups of $\Gamma$, such that $\gamma_k(\Gamma)<\overline{\gamma}_k(\Gamma)$. The interest of considering $\overline{\gamma}_k(\Gamma)$ instead of $\gamma_k(\Gamma)$ is to get rid of possible torsion for the sequence of relative quotients. Indeed, one checks that $\overline{\gamma}_k(\Gamma)/\overline{\gamma}_{k+1}(\Gamma)$ is a lattice in $\gamma_k(G)/\gamma_{k+1}(G)\simeq \R^{m_k}$. By a straight-forward induction on the nilpotency class of $G$, it suffices to check that $\overline{\gamma}_{s_G}(\Gamma)$ embeds as a lattice in $\gamma_{s_G}(G)$. Note that since it is a discrete subgroup, it is enough to see that it is cocompact, and therefore that $\gamma_{s_G}(\Gamma)$ is cocompact in $\gamma_{s_G}(G)$, which is well-known (see e.g. \cite{Rag-72}).

We now choose a canonical basis $\left\{X_{ij}\mid 1\leq i \leq {s_G},~ 1\leq i \leq m_i\right\}$ such that for a given $1\leq i \leq s_G$,  $(exp(X_{ij}))_{1\leq j \leq m_i}$ is a basis of the lattice  $\overline{\gamma}_k(\Gamma)/\overline{\gamma}_{k+1}(\Gamma)$. We will call such a choice of canonical basis \emph{compatible} with the lattice $\Gamma$.
This allows us to state a discrete version of Proposition \ref{prop:normal-form}. Its proof is identical and we also refer to \cite{Osi-01} for details.

\begin{proposition}\label{prop:normal-formdiscrete}
 Let $\Gamma$ be a torsion-free finitely generated nilpotent Lie group of nilpotency class $s_G$ and let $\left\{X_{ij}\mid 1\leq i \leq s_G,~ 1\leq j \leq m_i\right\}$ be a compatible canonical basis for the associated simply connected Lie group $G$.
Then every element $g\in \Gamma$ admits a decomposition of the form
    \[
        g= \prod_{i=1}^{s_G}\prod_{j=1}^{m_i} x_{ij}^{\lambda_{ij}(g)},
    \]
    where the $\lambda_{ij}(g)\in \Z$ are uniquely determined by $g$ and $|g|_\Gamma\simeq \sum_{i,j} |\lambda_{ij}(g)|^{\frac{1}{i}}$
\end{proposition}

As a consequence we can prove the following normal form result, which will provide us with the ingredients for a F\o lner tiling sequence.

\begin{lemma}\label{lem:normal-form}
    Let $G$ be a simply connected nilpotent Lie group and let $x_{ij}$ be as in \Cref{prop:normal-form}. Then every $g\in G$ has a unique decomposition of the form
    \[
        g= \prod_{\ell=0}^k \left(\prod_{i=1}^{s_G}\prod_{j=1}^{m_i} x_{i,j}^{\delta_{i,j,\ell}2^{\ell\cdot i}}\right),
    \]
    where $\delta_{i,j,0}\in \left(-2^{i-1},2^{i-1}\right]$, $\delta_{i,j,\ell}\in \left\{-2^i,\dots,-1\right\}$ if $\ell\geq 1$ is odd and $\delta_{i,j,\ell}\in \left\{0,\cdots,2^i-1\right\}$ if $\ell\geq 1$ is even, and $k$ is a minimal even integer for which there is such a decomposition. Moreover, for every integer $k'\geq k$ there is also a unique decomposition of the above form and the $\delta_{i,j,\ell}$ for $\ell\leq k'$ do not depend on $k'$.
    
If, moreover, $G$ admits a lattice $\Gamma$, and $x_{ij}$ is as in \Cref{prop:normal-formdiscrete}, then such a unique decomposition exists in $\Gamma$, with $\delta_{i,j,0}\in \left\{-2^{i-1}+1,\dots,2^{i-1}\right\}$.    
\end{lemma}
\begin{proof}
We will only prove the statement for $G$ as the proof for a lattice $\Gamma$ is identical (using \Cref{prop:normal-formdiscrete} instead of \Cref{prop:normal-form}). The proof is by induction on the nilpotency class $s_G$ of $G$.  It is straight-forward to see that the result holds for $s_G=1$, as in this case $G\cong \mathbb{R}^{m_1}$ is abelian.

    Assume that by induction the result holds for all simply connected nilpotent groups of nilpotency class $s_G-1$ and let $G$ be $s_G$-nilpotent. Then there is an even integer $k\geq 0$ for which there are unique $\delta_{i,j,\ell}$ satisfying the above conditions such that 
    \[
        h:= g^{-1}\cdot \prod_{\ell=0}^k\left( \prod_{i=1}^{s_G-1}\prod_{j=1}^{m_i}x_{ij}^{\delta_{i,j,\ell}2^{\ell\cdot i}}\right)\in \gamma_{s_G}(G)=\oplus_{j=1} ^{m_{s_G}} \mathbb{R}x_{s_G j}.
    \]
    In particular, there is a unique decomposition $h=\prod_{j=1}^{m_{s_G}}x_{ij}^{\lambda_{s_G j}(h)}$. After possibly making $k$ larger, we may further assume that $\lambda_{s_G j}(h)\in \left(a_{s_G k},b_{s_G k}\right]$, where
    \[
        a_{ik}=1-2^{i-1}-\frac{2^{(k+1)i}-1}{4^i-1},~~ b_{ik}=2^{(k+1)i} +1 -2^{i-1}- \frac{2^{(k+1)i}-1}{4^i-1}
    \]
    for $1\leq j \leq m_{s_G}$. Hence, there is a unique decomposition of $\lambda_{s_G j}(h)$ of the form
    \[
        \lambda_{s_G j}(h)= \sum_{\ell=0}^k \delta_{s_G,j,\ell}2^{\ell \cdot s_G},
    \]
    where $\delta_{s_G,j,\ell}$ satisfies the above conditions. We deduce that $h$ uniquely decomposes as 
    \[
        h= \prod_{\ell=0}^k\prod_{j=1}^{m_{s_G}} x_{s_Gj}^{\delta_{s_G,j,\ell}2^{\ell\cdot s_G}}.
    \]
    Hence, the fact that the element $h$ is central in $G$, combined with the uniqueness of the decomposition of $g\cdot h$ as $\prod_{\ell=0}^k\left( \prod_{i=1}^{s_G-1}\prod_{j=1}^{m_i}x_{ij}^{\delta_{i,j,\ell}2^{\ell\cdot i}}\right)\in \gamma_{s_G}(G)=\oplus_{j=1} ^{m_{s_G}} \mathbb{R}x_{s_G j}$, implies the existence of a unique decomposition of $g$ as in the assertion for our chosen value of $k$. This completes the induction step.    
\end{proof}

The normal form defined in \Cref{lem:normal-form} provides us with the following candidate for a F\o lner tiling sequence for a simply connected nilpotent group $G$:
\begin{itemize}
    \item $F_0= \left\{\prod_{i=1}^{s_G}\prod_{j=1}^{m_i} x_{ij}^{\delta_{i,j,0}}\mid \delta_{i,j,0}\in \left(-2^{i-1},2^{i-1}\right],~ 1\leq i \leq s,~ 1\leq j\leq m_i\right\}$,
    \item $F_{\ell}=\left\{\prod_{i=1}^{s_G}\prod_{j=1}^{m_i} x_{ij}^{\delta_{i,j,\ell}2^{\ell\cdot i}}\mid \delta_{i,j,\ell}\in \left\{-2^i,\dots, -1\right\},~ 1\leq i \leq s,~ 1\leq j\leq m_i\right\}$ if $\ell\geq 1$ is odd,
    \item $F_{\ell}=\left\{\prod_{i=1}^{s_G}\prod_{j=1}^{m_i} x_{ij}^{\delta_{i,j,\ell}2^{\ell\cdot i}}\mid \delta_{i,j,\ell}\in \left\{0,\dots, 2^i-1\right\},~ 1\leq i \leq s,~ 1\leq j\leq m_i\right\}$ if $\ell \geq 1$ is even.
\end{itemize}

\Cref{lem:normal-form} guarantees that the sequence $(F_k)$ satisfies \Cref{def:folner-tiling} (1) for the corresponding tiles $T_0=F_0$ and $T_{k+1}=\bigsqcup_{g\in F_{k+1}} T_k\cdot g$ and it remains to be checked that the $T_k$ define a F\o lner sequence for $G$. 

Before doing so we observe that if $r_G$ denotes the polynomial growth degree of $G$, then $|F_k|=2^{r_G}$ for $k\geq 1$ and thus $\mu(T_k)=2^{r_G\cdot k}\cdot \mu(F_0)$. Moreover, for every $\ell$ the elements in $F_{\ell}$ have word length $\leq C\sum_{i=1}^{s_G}\sum_{j=1}^{m_i} 2 ^{\ell}$ with respect to the word metric defined by a compact generating set $S_G\subseteq F_0$, where $C$ only depends on $G$ and the choice of $S_G$. Thus, the tile $T_k$ has diameter $\leq \sum_{\ell=0}^k \sum_{i=1}^{s_G}\sum_{j=1}^{m_i} 2 ^{\ell}\leq C \cdot h(G) \cdot 2^{k+1}$, where $h(G)$ is the Hirsch length of $G$.

\begin{lemma}\label{lem:nilpotent-explicit-Folner-tiling}
    Let $G$ be a simply connected nilpotent Lie group of nilpotency class $s_G$ (resp.\ a lattice $\Gamma$) and let $\left\{x_{ij}\mid 1\leq i\leq s_G,~ 1\leq j \leq m_i\right\}$ be as in \Cref{prop:normal-form} (resp.\ as in \Cref{prop:normal-formdiscrete}). Then the sequence $(F_k)$ defined above is an $(\epsilon_k,R_k)$-F\o lner tiling sequence for $G$ (resp.\ $\Gamma$) with $\epsilon_k=O(2^{-k})$, $R_k=O(2^k)$ and, for $k\geq 1$, $|F_k|= 2 ^{r_G}$, where $G$ has polynomial growth of degree $r_G$. 
\end{lemma}
\begin{proof}
 Once again, we only write the proof for $G$. We fix a compact generating set $S_G\subseteq F_0$ for $G$.
    The sequence $(F_k)$ satisfies \Cref{def:folner-tiling} (1) by construction. As mentioned before, if we show that there is a non-increasing sequence $(\epsilon_k)$ which converges to $0$ such that for every $s\in S_G$ we have $\frac{\mu(s \cdot T_k\setminus T_k)}{\mu(T_k)}\leq \epsilon_k$, then this also implies that $(T_k)$ defines a sequence of F\o lner sets for $G$. Moreover, the above discussion shows that we can choose $R_k=2 \cdot C\cdot h(G)\cdot 2 ^k$ to obtain an upper bound on the diameters of the $T_k$. 
    
    So the only thing we need to check is \Cref{def:quant-Folner} (2), that is, we need to prove that there is a non-increasing sequence $(\epsilon_k)$ with $\epsilon_k=O(2^{-k})$ such that for all $s\in S_G$ we have
    \[
        \mu(s\cdot T_k\setminus T_k)\leq \epsilon_k \mu (T_k).
    \]
    
    We will prove this by induction on the nilpotency class $n$ after observing that the F\o lner tiling sequence $(F_k)$ and the sequence of tiles $(T_k)$ naturally define a F\o lner tiling sequence $(\overline{F}_k)$ and a sequence of tiles $(\overline{T}_k)$ for the quotient $G/\gamma_{s_G}(G)$ of $G$ by the $s_G$-th term of the lower central series, where in $\overline{F}_k$ and $\overline{T}_k$ we omit the generators $x_{s_G j}$. 

    We start with the base case $s_G=1$. In this case the group is abelian. Since $S_G\subseteq F_0$ it suffices to bound above $\mu(F_0\cdot T_k\setminus T_k)$. Moreover, by construction $T_k=(a_{1k},b_{1k}]^{m_1}\subset \mathbb{R}^{m_1}$ and for every $s\in S_G$ and $x=f_0(x)\cdots f_k(x)\in T_k$ we have
    \[
        s\cdot f_0(x)\cdots f_k(x) \in \left(a_{1k}-2,b_{1k}+2\right]^{m_1}.
    \]
    In particular, since the Haar measure $\mu$ on $G$ is the standard Lebesgue measure on $\mathbb{R}^{m_1}$ and $b_{1k}-a_{1k}=2^{k+1}$, a standard calculation shows that there is a constant $C$ that only depends on $m_1$ such that 
    \[
        \mu (S_G\cdot T_k\setminus T_k)\leq \mu\left(\left(a_{1k}-2,b_{1k}+2\right]^{m_1} \setminus (a_{1k},b_{1k}]^{m_1}\right)\leq C\cdot 2 ^{(m_1-1)k}=\frac{C}{2^{m_1}}\cdot 2^{-k}\cdot \mu(T_k).
    \]

    We will now assume that the result holds for groups of nilpotency class at most $s_G-1$. We denote by $\overline{\mu}$ any Haar measure on $\overline{G}=G/\gamma_{s_G}(G)$ and by $S_{\overline{G}}$ the generating set induced by projecting $S_G$. Let $\overline{C}$ be the constant such that $\overline{\mu}(S_{\overline{G}} \cdot \overline{T}_k\setminus \overline{T}_k)\leq \overline{C}2^{-k}\overline{\mu}(\overline{T}_k)$ for all $k\geq 0$. 
    
    We will distinguish two possible cases for elements $s\cdot x$ with $s\in S_G$, $x\in T_k$ and $s\cdot x\notin T_k$. 
    
    The first are elements such that their projection $\overline{s}\cdot \overline{x}$ to $\overline{G}$ is not in $\overline{T}_k$. There are at most 
    \[
                \frac{\mu(T_k)}{\overline{\mu}(\overline{T_k})}\cdot \overline{\mu}\left( S_{\overline{G}} \cdot \overline{T}_k\setminus \overline{T}_k\right)=        \frac{\mu(T_k)}{\overline{\mu}(\overline{T_k})}\cdot \overline{\mu}\left( \overline{T}_k\setminus S_{\overline{G}}^{-1} \cdot\overline{T}_k\right)
    \]
    such elements. 
    Indeed, observe that the measure defined by first integrating along  
    $\gamma_{s_G}(G)$ (which yields a $\gamma_{s_G}(G)$-invariant function on $G$, or equivalently a function on $\overline{G}$) and then along $\overline{G}$ is invariant, hence defines a Haar measure of $G$. 
    On the other hand, by construction, the fibers of $T_k\to \overline{T}_k$ are all translates of a single subset of $\gamma_{s_G}(G)$, whose measure is $\frac{\mu(T_k)}{\overline{\mu}(\overline{T_k})}$.
    Now, consider a subset $A\subset T_k$. First integrating along $\gamma_{s_G}(G)$ gives rise to a positive function on $\overline{G}$ supported on the projection $\overline{A}$, and with values at most the size of a fiber of $T_k\to \overline{T}_k$, implying the asserted upper bound. Altogether, we deduce that
    \[
        \frac{\mu(T_k)}{\overline{\mu}(\overline{T_k})}\cdot \overline{\mu}\left(S_{\overline{G}} \cdot \overline{T}_k\setminus \overline{T}_k\right) \leq \overline{C}\cdot 2^{-k}\mu(T_k).
    \]

    The second are elements such that $\overline{s}\cdot \overline{x}\in \overline{T}_k$, but $s\cdot x\notin T_k$. Then we can write $x=z\cdot \overline{f}_0(\overline{x}) \cdots \overline{f}_0(\overline{x})$ and $s\cdot x =z'\cdot \overline{f}_0(\overline{s}\cdot \overline{x}) \cdots \overline{f}_k(\overline{s}\cdot \overline{x})$, where by slight abuse of notation we denote by $\overline{f}_i(\overline{x})$, respectively $\overline{f}_i(\overline{s}\cdot \overline{x})$, the canonical lifts of the $i$-th term of the tiled form for $\overline{x}$, respectively $\overline{s}\cdot \overline{x}$, to $G$. If $\overline{f}_{\ell}(\overline{x})=\overline{f}_{\ell}(\overline{s}\cdot \overline{x})$ for all $\ell> k_0$ for some fixed $k_0\in \left\{0,\cdots, k\right\}$, then
    \begin{align*}
        1&= d_{S_G}\left(z\cdot \overline{f}_0(\overline{x}) \cdots \overline{f}_{k_0}(\overline{x}),z'\cdot \overline{f}_0(\overline{s}\cdot \overline{x}) \cdots \overline{f}_{k_0}(\overline{s}\cdot \overline{x})\right)\\
        &=d_{S_G}\left(\left(\overline{f}_0(\overline{x}) \cdots \overline{f}_{k_0}(\overline{x})\right)\cdot \left(\overline{f}_0(\overline{s}\cdot \overline{x}) \cdots \overline{f}_{k_0}(\overline{s}\cdot \overline{x})\right)^{-1},z'\cdot z^{-1}\right)\\
        &\geq d_{S_G}(z'z^{-1},1)-d_{S_G}\left(1,\left(\overline{f}_0(\overline{x}) \cdots \overline{f}_{k_0}(\overline{x})\right)\cdot \left(\overline{f}_0(\overline{s}\cdot \overline{x}) \cdots \overline{f}_{k_0}(\overline{s}\cdot \overline{x})\right)^{-1}\right),
    \end{align*}
    where for the second equality we use that $z,~z'\in \gamma_{s_G}(G)\leq Z(G)$ are central. This implies that   
    \begin{align*}
        d_{S_G}(z,z')&\leq 1+d_{S_G}\left(\overline{f}_0(\overline{x}) \cdots \overline{f}_{k_0}(\overline{x}), \overline{f}_0(\overline{s}\cdot \overline{x}) \cdots \overline{f}_{k_0}(\overline{s_G}\cdot \overline{x})\right)\\
        & \leq 1+R_{k_0}\\
        & \leq  C' 2^{k_0},
    \end{align*}
for some constant $C'$ that only depends on $\overline{G}$. Since $\gamma_{s_G}(G)\leq G$ has distortion of polynomial degree at most $s_G$, we deduce that there is a constant $C''$ that only depends on $G$ such that
    \[
        d_{S_{\gamma_{s_G}(G)}}(z,z')\leq C'' 2^{s_G\cdot k_0},
    \] where $S_{\gamma_{s_G}(G)}$ is the generating set 
    \[
        S_{\gamma_{s_G}(G)}=\left\{x_{s_G 1}^{\lambda_{s_G 1}},\cdots , x_{{s_G}m_{s_G}}^{\lambda_{s_G m_{s_G}}}\mid \lambda_{s_G j}\in\left(-2^{s_G-1},2^{s_G-1}\right], 1\leq j\leq m_{s_G}\right\}
    \]
    for $\gamma_{s_G}(G)$. This means that $z'$ has to take values within a $C''2^{s_G\cdot k_0}$ neighbourhood of the set $\left(a_{s_G k},b_{s_G k}\right]^{m_{s_G}}\subset \gamma_{s_G}(G)\cong \mathbb{R}^{m_{s_G}}$, or said differently in a set of measure
    \[
        \leq \mu_{\mathbb{R}^{m_{s_G}}}\left( \left(a_{s_G k}-C''\cdot 2^{s_G \cdot k_0},b_{s_G k}+ C''\cdot 2^{s_G \cdot k_0}\right]^{m_{s_G}}\setminus \left(a_{s_G k},b_{s_G k}\right]^{m_{s_G}}\right)\leq C_3\cdot 2^{s_G \cdot k_0}\cdot 2^{(m_{s_G}-1)\cdot k\cdot s_G}
    \]
    for a constant $C_3$ that only depends on $m_{s_G}$, where we use that $b_{s_G k}-a_{s_G k}=2^{s_G \cdot k}$. We deduce that the total Haar measure of elements of the second kind is bounded above by
    \begin{align*}
        &C_4 \sum_{\ell=1}^k \frac{\overline{\mu}(\overline{T}_k)}{\overline{\mu}(\overline{T}_{\ell})}\cdot \overline{\mu}(S_{\overline{G}}\cdot \overline{T}_{\ell-1} \setminus \overline{T}_{\ell-1})\cdot C_3 \cdot 2^{s_G \cdot \ell}\cdot 2^{(m_{s_G}-1)\cdot k\cdot s_G}\\
        &\leq C_4 2^{-r_{\overline{G}}}\sum_{\ell=1}^{k} \frac{\overline{\mu}(\overline{T}_k)}{\overline{\mu}(\overline{T}_{\ell-1})}\cdot \overline{\mu}(S_{\overline{G}}\cdot \overline{T}_{\ell-1} \setminus \overline{T}_{\ell-1})\cdot C_3 \cdot 2^{s_G\cdot \ell}\cdot 2^{(m_{s_G}-1)\cdot k\cdot s_G}\\
        &\leq C_4 2^{-r_{\overline{G}}} \sum_{\ell=1}^{k} \mu(T_k) 2^{-s_G m_{s_G} k} \cdot \overline{C} \cdot 2^{-\ell+1}\cdot C_3 \cdot 2^{s_G\cdot \ell}\cdot 2^{(m_{s_G}-1)\cdot k\cdot s_G}\\
        &\leq C_5 2^{-k}\mu(T_k)\sum_{\ell=1}^k 2^{(\ell-k)(s_G-1)} \\
        &\leq C_6 2^{-k}\mu(T_k),
    \end{align*}
    where to obtain the first line we use that if $\overline{f}_{\ell}(\overline{x})\neq \overline{f}_{\ell}(\overline{s}\cdot \overline{x})$, then $\overline{s}\cdot \overline{f}_0(\overline{x})\cdots \overline{f}_{\ell}(\overline{x})\notin \overline{T}_{\ell-1}$, and to obtain the last line we use that $s_G\geq 2$ and thus the sum is a convergent geometric series; the constants $C_4$, $C_5$ and $C_6$ only depend on $G$ ($C_5$ and $C_6$ absorb previous constants). Note that we have again used in the above inequalities that the Haar measure of the extension $1\to \gamma_{s_G}(G)\to G \to \overline{G}\to 1$ can be described by first integrating along $\gamma_{s_G}(G)$ and then along $\overline{G}$.

    Combining the two cases, we obtain that there is a constant $C_7$ that only depends on $G$ such that
    \[
        \mu(S_G\cdot T_k\setminus T_k)\leq C_7 2^{-k}\cdot \mu(T_k).
    \]
    Thus, choosing $\epsilon_k=C_7 2^{-k}$ completes the proof.
\end{proof}

As a consequence we obtain:

\begin{proposition}\label{prop:coupling-same-growth-Lie}
    Let $G$, $H$ be simply connected nilpotent Lie groups of the same polynomial growth. Then they are $L^p$-orbit equivalent for every $p<1$. If, moreover, $G$ and $H$ admit lattices $\Gamma$ and $\Lambda$, then these are $L^p$-orbit equivalent for every $p<1$.
\end{proposition}

\begin{proof} 
    As before, we shall only treat the case of $G$ and $H$.
    By \Cref{lem:nilpotent-explicit-Folner-tiling} there are $(\epsilon_k,R_k)$-F\o lner tiling sequences $(F_{k,G})$, resp. $(F_{k,H})$ for $G$, resp. $H$, and a constant $C>1$ such that $\epsilon_k\leq C\cdot 2^{-k}$, $R_k\leq C\cdot 2^k$, and $F_{k,G}$ and $F_{k,H}$ have the same cardinality for all $k\geq 0$; here for $k=0$ we use that $G$ and $H$ are simply connected nilpotent Lie groups and thus $F_{0,G}$ and $F_{0,H}$ are uncountable. Thus, \Cref{prop:quant-orbit-equivalence} implies that the orbit equivalence coupling from \Cref{cor:existence-of-orbit-equivalence} is $(\phi,\phi)$-integrable for every non-decreasing function $\phi$ such that the sequence $\epsilon_{k-1}\cdot \phi(2R_k)$ is summable.

    For $\psi_{p,\epsilon}(x)= \frac{x^p}{\log(x)^{1+\epsilon}}$ we have
    \begin{align*}
        \sum_{k=1}^{\infty} \epsilon_{k-1}\cdot \psi_{p,\epsilon}(2R_k) 
        &\leq 2^{p+1}\cdot C^{p}\cdot \sum_{k=1}^{\infty} 2^{-k}\cdot \frac{(2^k)^p}{\left(\log(C)+ \log(2^k)\right)^{1+\epsilon}} \\
        & = \sum_{k=1}^{\infty} 2^{p+1}\cdot C^{p}\frac{2^{k\cdot(p-1)}}{(\log(C)+k)^{1+\epsilon}},
    \end{align*}
    which converges for all $p\leq 1$ and $\epsilon>0$. Since $x^p=o(\psi_{1,\epsilon})$ for every $p<1$, we deduce that the orbit equivalence coupling is $L^p$-integrable for every $p<1$. Since the $G$- and $H$-actions defining the coupling are essentially free, the assertion follows.
\end{proof}

By choosing adapted (coarser) F\o lner tiling sequences, similar as in the proof of \cite[Corollary 6.11]{DKLMT-22}, we can prove \Cref{mainthm:growth} for simply connected nilpotent Lie groups and for torsion-free finitely generated nilpotent groups.

\begin{lemma}\label{lem:adapted-Folner-tiling}
    Let $G$ be an infinite simply connected nilpotent Lie group, or a torsion-free finitely generated nilpotent group, of polynomial growth degree $r_G$ and let $m\in \mathbb{N}_{>0}$. Then $G$ admits a $(\epsilon_k,R_k)$-F\o lner tiling sequence $(F_k')$ with $|F_k'|=2^{m\cdot r_G}$ for $k\geq 1$, $\epsilon_k=O(2^{-m\cdot k})$, and $R_k=O(2^{m\cdot k})$.  
\end{lemma}
\begin{proof}
    We explain the case of a simply connected nilpotent Lie group, the other case is again analogous. Let $(F_k)$ be the F\o lner tiling for $G$ considered in \Cref{lem:nilpotent-explicit-Folner-tiling}. We define $F_0'=F_0$ and for $k\geq 1$ we define $F_k'=F_{m(k-1)+1}\cdot F_{m(k-1)+2} \cdots F_{mk}$. Then $|F_k'|=2^{m\cdot r_G}$. Moreover, $T_{k}'= T_{mk}$, and thus $\epsilon_k=O(2^{-m\cdot k})$ and $R_k=O(2^{m\cdot k})$.
\end{proof}

\begin{proposition}\label{prop:coupling-arbitrary-nilpotent-groups}
    Let $G$ and $H$ be infinite simply connected nilpotent Lie groups, or torsion-free finitely generated nilpotent groups, of polynomial growth degrees $r_G$ and $r_H$. Then there is a orbit equivalence coupling from $G$ to $H$, which is $(L^p,L^q)$-integrable for all $p<\frac{r_H}{r_G}$ and $q<\frac{r_G}{r_H}$. Moreover if $G$ and $H$ admit lattices $\Gamma$ and $\Lambda$, then there is a orbit equivalence coupling from $\Gamma$ to $\Lambda$, which is $(L^p,L^q)$-integrable for all $p<\frac{r_H}{r_G}$ and $q<\frac{r_G}{r_H}$.
\end{proposition}
\begin{proof} 
    Once again, we shall only treat the case when $G$ and $H$ are simply connected nilpotent, the discrete case is analogous.
    We apply \Cref{lem:adapted-Folner-tiling} to $G$ with $m=r_H$ and to $H$ with $m=r_G$ to obtain $(\epsilon_{k,G},R_{k,G})$- and $(\epsilon_{k,H},R_{k,H})$-F\o lner tilings sequences $(F_{k,G})$ and $(F_{k,H})$ for $G$ and $H$ with $|F_{k,G}|=|F_{k,H}|=2^{r_G\cdot r_H}$, $\epsilon_{k,G}=O(2^{-r_H\cdot k})$, $\epsilon_{k,H}=O(2^{-r_G\cdot k})$, $R_{k,G}=O(2^{r_H \cdot k})$ and $R_{k,H}=O(2^{r_G\cdot k})$. As above, let $\psi_{p,\epsilon}(x)=\frac{x^p}{\log(x)^{1+\epsilon}}$. Then the same arguments as in the proof of \Cref{prop:coupling-same-growth-Lie} show that the sequence $\epsilon_{k-1,G}\cdot \psi_{p,\epsilon}(2\cdot R_{k,H})$ is summable for all $p\leq \frac{r_H}{r_G}$ and $\epsilon>0$, while $\epsilon_{k-1,H}\cdot \psi_{q,\epsilon}(2\cdot R_{k,G})$ is summable for all $q\leq\frac{r_G}{r_H}$ and $\epsilon>0$. Thus, we can again argue using \Cref{prop:quant-orbit-equivalence} as in the proof of \Cref{prop:coupling-same-growth-Lie} to derive the assertion.
\end{proof}

We end this section with a proof of \Cref{mainthm:growth}. It will rely on the following easy lemma. 

\begin{lemma}\label{lem:AbfiniteIndex}
    Let $F$ be a finite group and $n\in \N$. Then $\Z^n$ is bilipschitz equivalent (hence $L^{\infty}$ OE) to $\Z^n\times F$.
\end{lemma}
\begin{proof}
Let $\Lambda$ be a subgroup of $\Z^n$ of index $|F|$. Note that $\Lambda$ is still isomorphic to $\Z^n$. On the other hand, by \cite{Whyte}, $\Z^n$ is bilipschitz equivalent to $\Lambda\times F$, so the lemma is proved.
\end{proof}

\begin{proof}[Proof of \Cref{mainthm:growth}]
The Lie group case and the torsion-free finitely generated case is a direct consequence of the more precise \Cref{prop:coupling-arbitrary-nilpotent-groups}.  
Now assume that $\Gamma$ and $\Lambda$ are infinite finitely generated virtually nilpotent groups and let $\Gamma'$ and $\Lambda'$ be torsion-free finite index subgroups, of indices respectively $i$ and $j$. By \cite{Whyte}, $\Gamma$ and $\Lambda$ are respectively bilipschitz equivalent (hence $L^{\infty}$ OE) to $\Gamma'\times \Z/i\Z$ and $\Lambda'\times \Z/j\Z$. By the torsion-free case, we know that $\Gamma'$ and $\Lambda'$ are both $L^{<1}$ OE to respectively $\Z^{r_\Gamma}$ and $\Z^{r_\Lambda}$. In turn, this implies that $\Gamma$ and $\Lambda$
are both $L^{<1}$ OE to respectively $\Z^{r_\Gamma}\times  \Z/i\Z$ and $\Z^{r_\Lambda}\times  \Z/j\Z$. But by \Cref{lem:AbfiniteIndex} those are $L^{\infty}$ OE to $\Z^{r_\Gamma}$ and $\Z^{r_\Lambda}$. To conclude, we simply compose these OE-couplings and apply \cite[Proposition 2.26]{DKLMT-22}, which ensures that the resulting coupling from $\Gamma$ to $\Lambda$ is $(L^p,L^q)$-integrable  for all $p<\frac{r_\Lambda}{r_\Gamma}$ and $q<\frac{r_\Gamma}{r_\Lambda}$. This completes the proof.
\end{proof}

\section{Word length estimates in simply connected Lie groups}\label{sec:length-of-commutators}

To prove Theorem \ref{mainthm:Carnot} we will require some more precise estimates of the word lengths of coycles than in the previous sections. In this section we will prove some auxiliary results on the word length of commutators in simply connected nilpotent groups, which we will require to obtain these estimates.

We let $G$ be a simply connected nilpotent Lie group of step $s_G$, $\mathfrak{g}$ be its Lie algebra and as before we denote by $\gamma_i(\mathfrak{g})$ the $i$-th term of its lower central series. Let $A_i$ be a vector complement of $\gamma_{i+1}(\mathfrak{g})$ in $\gamma_i(\mathfrak{g})$.
We shall assume that $\mathfrak{g}$ is $s_G$-step nilpotent.
We let $\|\cdot\|$ be a norm on $\mathfrak{g}$. As usual we identify $G$ with its Lie algebra via the exponential map.
With this convention, we shall write $x\ast y$ for the group product.

We fix a compact symmetric generating set $S_G$ of $G$ and denote by $|x|$ the word length of $x\in G=\mathfrak{g}$ with respect to this generating set.
We recall \cite[Proof of Th\'eor\`eme II.1]{Guivarch73} that there exists $C\geq 1$ such that for all $x=\sum_{i=1}^{s_G} x_i\neq 0$, where each $x_i\in A_i$, we have
\begin{equation}\label{eq:length}
|x|\simeq 1+\sum_{i=1}^{s_G} \|x_i\|^{1/i}.
\end{equation}
Since $\|\cdot\|^{1/i}$ is subadditive for every $i\geq 1$
we deduce that
\[|x+y|\lesssim |x|+|y|.\]
We also  deduce from (\ref{eq:length}) that  for each $x\in \gamma_i(\mathfrak{g})$,  $|x|\lesssim (1+\|x\|)^{1/i}$.

On the other hand, bilinearity of the bracket ensures that for all $x,y\in \mathfrak{g}$,
\[\|[x,y]\|\lesssim \|y\|\|x\|.\]

\begin{lemma}\label{lem:commij}
Let $i,j\geq 1$. For every $u\in A_i$ and $v\in A_j$, 
we have $|[u,v]|^{i+j}\lesssim|u|^i|v|^{j}$
\end{lemma}
\begin{proof} 
Note that for every $u\in A_i$ and $v\in A_j$, $[u,v]\in \gamma_{i+j}(\mathfrak{g})$. Hence we obtain
\begin{align*}
      |[u,v]|^{i+j} & \lesssim 1+\|[u,v]\|\\
      & \lesssim 1+\|u\|\|v\|\\
      & \lesssim 1+|u|^{i}|v|^{j}
    \end{align*}
    Since $u$ and $v$ can be assumed to be non-zero and thus have non-trivial word length $|u|,|v|\in \mathbb{N}$, we can ignore the additive term $1$, and the lemma is proved. 
\end{proof}
An immediate induction yields the following generalization. 
\begin{lemma}\label{lem:multicomm}
Let $i_1,i_2,\ldots, i_q\geq 1$. For every $u_j\in A_{i_j},$ 
we have $|[u_1,u_2,\ldots, u_{q}]|^{i_1+\ldots+i_q}\lesssim|u_1|^{i_1}\ldots|u_q|^{i_q}$
\end{lemma}

 \begin{lemma}\label{lem:commLie}
Let $x,y\in \mathfrak{g}$, 
we have \[|[x,y]|\lesssim \max_{i,j\geq 1; i+j\leq s_G}|x|^{i/(i+j)}|y|^{j/(i+j)}.\]
More precisely, if $x\in \gamma_p(\mathfrak{g})$ and $y\in\gamma_q(\mathfrak{g})$,
then \[|[x,y]|\lesssim \max_{i\geq p,j\geq q; i+j\leq s_G}|x|^{i/(i+j)}|y|^{j/(i+j)}.\]
\end{lemma}
\begin{proof}
Applying Lemma \ref{lem:commij}, we obtain
\begin{align*}
      |[x,y]| & \lesssim \sum_{i\geq p,j\geq q}|[x_i,y_j]|\\
      & \lesssim \sum_{i\geq p,j\geq q; i+j\leq s_G}|x_i|^{i/(i+j)}|y_j|^{j/(i+j)} \\
      & \lesssim \max_{i\geq p,j\geq q; i+j\leq s_G}|x|^{i/(i+j)}|y|^{j/(i+j)}
    \end{align*}
    Where for the last inequality, we have used that $|x_i|\lesssim |x|$, which follows once again from (\ref{eq:length}).
\end{proof}
Once again we have a generalization to iterated commutators.
 \begin{lemma}\label{lem:multicommutatorsLie}
Let $x\in \gamma_p(\mathfrak{g})$ and $y\in \gamma_q(\mathfrak{g})$. Then for every iterated commutator $z$ of $x$ and $y$
we have \[|z|\lesssim \max_{i\geq p,j\geq q; i+j\leq s_G}|x|^{i/(i+j)}|y|^{j/(i+j)}.\]
\end{lemma}
\begin{proof}
 The proof is very similar to the baby case treated above. The only difference lies in the heavier notation. Let $q$ and $p$ be respectively the numbers of occurrences of $x$ and $y$ in the iterated commutator $z$. We let $\Sigma$ be the set of pairs  $(I,J)$  of sequences $I=(i_1,\ldots,i_q)$ of indices $\geq p$ , and  $J=(j_1,\ldots,j_p)$ of indices $\geq q$ such that the total sum $N:=i_1+\ldots i_q+j_1+\ldots j_p\leq s_G$. 
 Reasoning as in the proof of Lemma \ref{lem:commLie}, but using Lemma \ref{lem:multicomm} instead of Lemma \ref{lem:commij}  we indeed obtain,  
 \[|z|\lesssim \max_{(I,J)\in \Sigma}|x|^{(i_1+\ldots i_q)/N}|y|^{(j_1+\ldots j_p)/N},\]
 which immediately implies the lemma.
\end{proof}

We convert this inequality to a similar one, but for commutators in the group. 
 \begin{proposition}\label{prop:commutatorsGroup}
Let $x\in \gamma_p(\mathfrak{g})$ and $y\in \gamma_q(\mathfrak{g})$. 
We have \[|x\ast y\ast x^{-1}\ast y^{-1}|\lesssim \max_{i\geq p,j\geq q; i+j\leq s_G}|x|^{i/(i+j)}|y|^{j/(i+j)}.\]
\end{proposition}
\begin{proof}
We apply the Baker-Campbell-Hausdorff formula, which enables us to write $x\ast y\ast x^{-1}\ast y^{-1}$ as a finite linear combination of iterated commutators of $x$ and $y$. Hence the lemma follows from Lemma \ref{lem:multicommutatorsLie}.
\end{proof}
\begin{corollary}\label{cor:commutatorEstimate}
  Let $x,y\in \mathfrak{g}$ such that $y\in \gamma_q(\mathfrak{g})$, and let $\varepsilon>0$. If $|y|\lesssim |x|^{1-\varepsilon}$, then  
    \[|x\ast y\ast x^{-1}\ast y^{-1}|\lesssim  |x|^{1-\frac{q\cdot \varepsilon}{s_G}}.\] 
\end{corollary}
\begin{proof}
 By Proposition \ref{prop:commutatorsGroup}, we have
\begin{align*}
 |x\ast y\ast x^{-1}\ast y^{-1}| &\lesssim \max_{i\geq 1, j\geq q; i+j\leq s_G}|x|^{i/(i+j)}|y|^{j/(i+j)}\\
 & \lesssim \max_{i\geq 1, j\geq q; i+j\leq s_G}|x|^{1-j\varepsilon/(i+j)}.
 \end{align*}
 We conclude since the max is attained for $i+j=s_G,$ and $j=q$.
\end{proof}

\section{Integrable orbit equivalence between a group and its Carnot}\label{sec:OE-same-Carnot}

In this section we construct explicit $L^p$-integrable orbit equivalences between a simply connected Lie group $G$ and its associated Carnot group $\gr(G)$ for some $p>1$ that depends on $G$. As before, denote by $\mathfrak{g}$ the Lie algebra of $G$. Recall that the associated Carnot graded group $\gr(G)$ is the simply connected nilpotent Lie group whose associated Lie algebra is isomorphic to $\gr(\mathfrak{g}) :=\bigoplus_{i=1}^{s_G} \gamma_i(\mathfrak{g})/\gamma_{i+1}(\mathfrak{g})$ equipped with the graded Lie algebra structure $[,]_{\gr}$, obtained by reducing the Lie bracket $[,] : \gamma_i(\mathfrak{g})\times \gamma_j(\mathfrak{g})\to \gamma_{i+j}(\mathfrak{g})$ modulo $\gamma_{i+j+1}(\mathfrak{g})$. In particular, a canonical basis $\left\{X_{ij}\mid 1\leq i\leq s_G,~1\leq j\leq m_i\right\}_{i,j}$ for $\mathfrak{g}$ defines a choice of canonical basis for $\gr(\mathfrak{g}) $ and thus an isomorphism of $\mathbb{R}$-vector spaces $\gr(\mathfrak{g}) \to \mathfrak{g}$ which is compatible with the filtration defined by the lower central series. By combining this isomorphism with the F\o lner tiling sequences with respect to our chosen canonical basis from Section \ref{sec:explicit-tilings-for-nilpotent-groups} we will produce an explicit orbit equivalence between $G$ and $\gr(G)$, which satisfies the conditions in Corollary \ref{cor:existence-of-orbit-equivalence} and which is $L^p$-integrable for some $p>1$. As before we identify the Lie groups with their Lie algebras via the exponential map and then define the maps on tiles via an $\mathbb{R}$-vector space isomorphism

As in Section \ref{sec:length-of-commutators} fix a norm $||\cdot ||$ on $\mathfrak{g}$ and equip $\gr(\mathfrak{g}) $ with the pull-back of this norm under our chosen identification.  We recall that the Guivarc'h norm is $\simeq$-equivalent to the word metric on $G$ and $\gr(G)$. Thus, by slight abuse of notation, we denote $|x|:= 1+ \sum_{i=1}^{s_G} ||x_i||^{\frac{1}{i}}$, where as before $x=\sum_{i=1}^{s_G} x_i$ with $x_i\in A_i:={\mathrm{span}}_{\mathbb{R}}\left\{X_{ij}\mid 1\leq j\leq m_i\right\}$.

We will now identify $\gr(\mathfrak{g}) $ with $\gr(G)$ and $\mathfrak{g}$ with $G$ via their exponential functions.  Then the multiplication in $G$ induces a multiplication $\ast$ in $\mathfrak{g}$ and the multiplication in $\gr(G)$ induces a multiplication $\ast_{\gr}$ in $\gr(\mathfrak{g})$. The difference between the multiplication on the associated groups $G$ and $\gr(G)$ is described by the Baker--Campbell--Hausdorff formula. It depends on the choice of the isomorphism $\gr(\mathfrak{g}) \to \mathfrak{g}$ of $\mathbb{R}$-vector spaces and thus on the choice of canonical basis. More precisely, it is bounded in terms of a rational constant $e_D\in \left[0,1\right)$ associated with a so-called grading operator $D$, which is in turn determined by our choice of canonical basis. We refer to \cite[Section 6]{Cor-19} for the precise definition and properties of a grading operator $D$ and the constant $e_D$. Here we will only need that for all grading operators $D$ we have $e_D\leq 1-\frac{1}{s_G}$ and that there is a choice of canonical basis whose associated grading operator realises the infimum $e_{G}:=\inf\left\{e_D\mid D \mbox{ grading operator}\right\}$.

For the remainder of this section we fix a canonical basis and denote $D$ its associated grading operator. We can now state a slight generalisation of a result of Cornulier that bounds above the difference in word length between products of $n$ elements in $G$ and $\gr(G)$ in terms of $e_D$. 
\begin{lemma}\label{lem:cornulier-multiple-factors}
    Let $x_1,\cdots,x_n\in \mathfrak{g}$. Then
    \begin{equation}\label{eq:mult-error}
        |x_1\ast x_2 \ast \ldots \ast x_n - x_1\ast_{\gr} x_2 \ast_{\gr} \ldots \ast_{\gr} x_n| \lesssim  \max (1,|x_1|^{e_D}, |x_2|^{e_D}, \ldots, |x_n|^{e_D}).
    \end{equation}
\end{lemma}
\begin{proof}
    The proof is completely analogous to the proof for $n=2$ given by Cornulier \cite[Lemma 6.17]{Cor-19}. Indeed, the key steps in Cornulier's proof are: 
    \begin{itemize}
        \item[(i)]\cite[page 35, Equation (1)]{Cor-19}, which gives an upper bound on the Guivarc'h norm of the difference between Lie brackets in $\mathfrak{g}$ and $\gr(\mathfrak{g}) $ (and is independent of $n$); and
        \item[(ii)] an application of the Baker--Campbell--Hausdorff-formulas using the difference between Lie brackets from (i) to bound above the Guivarc'h norm of the difference $x_1\ast x_2 - x_1\ast_{\gr} x_2$.
    \end{itemize}
    It is easy to see that an analogous application of the Baker--Campbell--Hausdorff formula for $n>2$ yields the desired upper bound, using the precisely same arguments as Cornulier used in step (ii) of his proof. We thus omit the details of this estimate here. 
    
    Alternatively, Gabriel Pallier pointed out to us that one can also reduce to \cite[Lemma 6.17]{Cor-19} by induction on $n$. Indeed, by the triangle inequality we have
    \begin{align*}
        &|x_1\ast x_2 \ast \ldots \ast x_n - x_1\ast_{\gr} x_2 \ast_{\gr} \ldots \ast_{\gr} x_n|\\
        &\leq \sum_{i=1}^{n-1} |(x_1\ast \cdots \ast x_i) \ast (x_{i+1}\ast_{\gr} \cdots \ast_{\gr} x_n)- (x_1\ast \cdots \ast x_i) \ast_{\gr} (x_{i+1}\ast_{\gr} \cdots \ast_{\gr} x_n)|.
    \end{align*}
    We conclude by applying \cite[Lemma 6.17]{Cor-19} to every summand and then using the following estimate (and the analogous estimate for $\ast_{\gr}$)d
    \[
         |x_1 \ast \cdots \ast x_{k}|^{e_D}
\leq (|x_1| + \cdots + |x_{k}|)^{e_D}
\lesssim \max (1, | x_1 |^{e_D}, \cdots, |x_k|^{e_D}).
    \]
\end{proof}

We will apply \Cref{lem:cornulier-multiple-factors} to prove the existence of $L^p$-orbit equivalence couplings for some $p>1$. We start by observing that our identification of the Lie algebras of $G$ and $\gr(G)$ provides us with a canonical identification $\phi=\prod_k \phi_k \colon X_{\gr(G)}=\prod_k F_{k,gr(G)}\to X_G=\prod_k F_{k,G}$ of their associated $(\epsilon_k,R_k)$-F\o lner tiling sequences from \Cref{lem:nilpotent-explicit-Folner-tiling}. More precisely, using the notation of \Cref{sec:explicit-tilings-for-nilpotent-groups}, $\phi_{k}$ maps the element $\prod_{i=1}^{s_G}\prod_{j=1}^{m_i} x_{i,j}^{\delta_{i,j,k}2^{k\cdot i}}$ of $\gr(G)$, where the product is taken with respect to $\ast_{\gr}$, to the same element of $G$, where now the product is taken with respect to $\ast$.

By slight abuse of notation we will also denote by $\phi$ the restriction $\phi|_{T_{k,gr(G)}} \colon T_{k,gr(G)} \to T_{k,G}$. We further denote by $\hatphi=\prod_k \hatphi_k \colon (X_{\gr(G)},\mu_{\gr(G)})\to (X_G,\mu_G)$ an isomorphism of measure spaces; it exists by \Cref{cor:existence-of-orbit-equivalence} and since the sets $F_k$ for $k\geq 1$ are finite we may assume that $\hatphi_k=\phi_k$ for $k\geq 1$, explaining the notation. However, we may not a priori assume that $\hatphi_0=\phi_0$. 

Note that $\phi_0$ also provides us with an identification of compact generating sets $S_{\gr(G)}\subseteq F_{0,\gr(G)}$ for $\gr(G)$ and $S_{G}\subseteq F_{0,G}$ for $G$. 

Where this will not lead to confusion we will subsequently sometimes omit the isomorphism $\phi$ in our notation and indicate which group we are in just by the group operation $\ast_{\gr}$, resp. $\ast$, to simplify notation. 

Before proceeding, we observe that the maps $\hatphi$ and $\phi$ have finite distance.
\begin{lemma}\label{lem:identity-at-finite-distance}
    For almost every $x\in X$ we have $d_{S_G}(\hatphi(x),\phi(x))\leq R_0.$
\end{lemma}
\begin{proof}
    By definition we have $\hatphi_k=\phi_k$ for $k\geq 1$. Thus,
    \[
        d_{S_G}(\hatphi(x),\phi(x))\leq d_{S_G}(\hatphi_0(x),\phi_0(x)) \leq R_0.
    \]
\end{proof}

\begin{lemma}\label{lem:tiled-error-terms}
    Let $s\in S_{\gr(G)}$ and $x\in X_{\gr(G)}$ with $\rho(s\ast_{\gr} x, x)=k$ and let $g:=f_0(x)\ast_{\gr} \cdots \ast_{\gr} f_k(x)\in T_{k,\gr(G)}$. Then $s\ast_{\gr} g= f_0(s\ast_{\gr} x)\ast_{\gr} \cdots \ast_{\gr} f_k(s\ast_{\gr} x)\in T_{k,\gr(G)}$
    and there are elements $v_0,\cdots, v_k\in G$ of word length $|v_\ell|\lesssim 2^{\ell\cdot e_D}$ such that $s\ast \phi(g)= v_0\ast \phi_0(f_0(s\ast_{\gr} x))\ast \cdots \ast v_k\ast \phi_k(f_k(s\ast_{\gr} x))$
\end{lemma}
\begin{proof}
    The first assertion is an immediate consequence of the definition of $\rho$. So we only need to prove the second assertion. To simplify notation we will write $f_i:=f_i(x)$ and $f_i^{(s)}:=f_i(s\ast_{\gr} x)$.
    
    Denote $r_{-1}:=s$ and for $0\leq \ell \leq k$ 
    \[
        r_\ell:= \left(f_0^{(s)} \ast_{\gr} \cdots \ast_{\gr} f_\ell^{(s)}\right)^{-1} \ast_{\gr}\left(s \ast_{\gr} f_0 \ast_{\gr} \cdots \ast_{\gr} f_\ell\right).
    \]
    We view $r_\ell$ as the $\ell$-th residue term in the construction of the tiled form for $s\ast_{\gr}g$ in $\gr(G)$. In particular $r_k=1$, and $r_\ell$ has word length $\lesssim 2^\ell$, since $f_{\ell}$ and $f_{\ell}^{(s)}$ have word length $\lesssim 2^{\ell}$ for $0\leq \ell \leq k$.
    
    Then
    \[
    r_{\ell-1}\ast_{\gr} f_\ell \ast_{\gr} r_\ell^{-1} \ast_{\gr} (f_\ell^{(s)})^{-1} = 1 \mbox{ in } \gr(G).
    \]
    For $0\leq \ell \leq k$, define
    \begin{equation} \label{eq:def-vi}
        v_\ell:= r_{\ell-1}\ast \phi_\ell(f_\ell) \ast r_\ell^{-1} \ast (\phi_\ell(f_\ell^{(s)}))^{-1}\in G.
    \end{equation}
    Recall that $\phi_\ell(f_\ell)= \prod_{i=1}^{s_G}\prod_{j=1}^{m_i} x_{i,j}^{\delta_{i,j,\ell}2^{\ell\cdot i}}$, where comparing to $f_\ell$ the only change is that we are taking products with respect to $\ast$ instead of $\ast_{\gr}$. 
    
    By Lemma \ref{lem:cornulier-multiple-factors},
    \begin{align*}
        |v_\ell|&=|v_\ell - 0|\\ 
        & = |r_{\ell-1}\ast \phi_{\ell}(f_\ell) \ast r_\ell^{-1} \ast (\phi_{\ell}(f_\ell^{(s)}))^{-1} -  r_{\ell-1}\ast_{\gr} f_\ell \ast_{\gr} r_\ell^{-1} \ast_{\gr} (f_\ell^{(s)})^{-1}|\\
        &\lesssim \max \left(1, |r_{\ell-1}|^{e_D}, |f_\ell|^{e_D}, |r_\ell|^{e_D}, |f_\ell^{(s)}|^{e_D}\right)\\
        &\lesssim \max (1, 2^{\ell\cdot e_D}),
    \end{align*}
    where for the second equality we recall that we are identifying the neutral element $1$ in $\gr(G)$ and $G$ with the neutral element $0$ in their Lie algebras, and for the first inequality we apply \Cref{lem:cornulier-multiple-factors} to a product of $2+2\cdot \sum_{i=1}^{s_G} m_i$ terms and use that $\left| x_{i,j}^{\delta_{i,j,\ell} 2^{\ell\cdot i}}\right|\lesssim |f_\ell|^{e_D}$ for all $i,j$.

    Finally, by Equation \eqref{eq:def-vi}, we have
    \begin{align*}
        s\ast \phi(g)&= s\ast \phi_1(f_0)\ast \cdots \ast \phi_k(f_k)\\
        &= v_1\ast \phi_1(f_0^{(s)})\ast \cdots \ast v_k\ast \phi_k(f_k^{(s)}).
    \end{align*}
\end{proof}

As a consequence we can estimate the difference in word length between $s\ast \phi(g)$ and $\phi(s\ast_{\gr} g)$ for all elements $s\in S_{\gr(G)}$ and $x\in X_{\gr(G)}$.

\begin{proposition}\label{prop:bound-on-word-length}
    Let $s\in S_{\gr(G)}$ and let $x\in X_{\gr(G)}$ with $\rho(s\ast_{\gr} x,x)=k$. Let $g:=f_0(x)\ast_{\gr} \cdots \ast_{\gr} f_k(x)\in T_{k,\gr(G)}$. Then $s\ast_{\gr} g = f_0(s\ast_{\gr} x) \ast_{\gr} \cdots \ast_{\gr} f_k(s\ast_{\gr} x)\in T_{k,\gr(G)}$ and there is a constant $C>0$ that only depends on $\gr(G)$, $G$ and the identification of their Lie algebras such that 
    \[
        d_{S_G}\left(\phi(s\ast_{\gr} x), s\ast \phi(x)\right) \leq C \cdot 2^{\left(1-\frac{1-e_D}{s_G}\right)k},
    \]
    where $s_G$ is the nilpotency class of $G$ and $e_D\in \left[0,1-\frac{1}{s_G}\right]$ is Cornulier's constant.
\end{proposition}
\begin{proof}
    By Lemma \ref{lem:tiled-error-terms}, $s\ast \phi(g)= v_0\ast \phi_0(f_0(s\ast_{\gr} x))\ast \cdots \ast v_k\ast \phi_k(f_k(s\ast_{\gr} x))$ with $|v_\ell|\lesssim 2^{\ell \cdot e}$, while $\phi(s\ast_{\gr} g)= \phi_0(f_0(s\ast_{\gr} x))\ast \cdots \ast \phi_k(f_k(s\ast_{\gr} x))$. 
    
    Choosing $x_\ell:= (\phi_0(f_0(s\ast_{\gr} x))\ast \cdots \ast \phi_{\ell}(f_{\ell}(s\ast_{\gr} x)))^{-1}$ and $y_\ell:=v_\ell$ for $0\leq \ell \leq k$ (where $x_0=1$), we obtain
    \[
        s\ast \phi(g)= \left(\prod_{\ell=0}^k y_\ell\ast\left[y_\ell^{-1},x_\ell^{-1}\right]\right) \ast \phi_0(f_0(s\ast_{\gr} x))\ast \cdots \ast \phi_k(f_k(s\ast_{\gr} x)). 
    \]

    Since $|x_\ell|\lesssim 2^\ell$ and $|y_\ell|\lesssim 2^{\ell \cdot e_D}$, we argue as in the proof of Corollary \ref{cor:commutatorEstimate} to see that Proposition \ref{prop:commutatorsGroup} implies
    \[
        |\left[y_\ell^{-1},x_\ell^{-1}\right]|\lesssim 2^{\left(1-\frac{1-e_D}{s_G}\right)\ell}.
    \]
    Using again that $|y_\ell|\lesssim 2^{\ell \cdot e_D}\leq 2^{\left(1-\frac{1-e_D}{s_G}\right)\ell}$, as $s_G\geq 1$, and recalling that in contrast to usual convention we defined the word metric via left multiplication by generators, the assertion is an immediate consequence.
\end{proof}

\begin{corollary}\label{cor:bound-on-word-length}
    There exists a constant $C>0$ that only depends on $\gr(G)$, $G$ and the identification of their Lie algebras such that for all $s\in S_{\gr(G)}$ and almost every $x\in X_{\gr(G)}$ with $\rho(s\ast_\gr x,x)\leq k$ we have
    \[
        d_{S_G}(\hatphi(s\ast_{\gr} x),\hatphi(x))\leq C \cdot 2^{\left(1-\frac{1-e_D}{s_G}\right)\cdot k},
    \]
    where $s_G$ denotes the nilpotency class of $G$ and $e_D\in \left[0,1-\frac{1}{s_G}\right]$ is Cornulier's constant.
\end{corollary}
\begin{proof}
    This is an immediate consequence of \Cref{prop:bound-on-word-length} and the following estimates for all $s\in S_{\gr(G)}$ and almost every $x\in X$:
    \begin{align*}
        &d_{S_G}\left(\hatphi(s\ast_{\gr} x),\hatphi(x)\right)\\
        \leq & d_{S_G}\left(\hatphi(s\ast_{\gr} x),\phi(s\ast_{\gr} x)\right) + d_{S_G}\left(\phi(s\ast_{ \gr} x\right), s\ast \phi(x))
        \\ & + d_{S_G}\left(s\ast \phi(x),\phi(x)\right)+ d_{S_G}\left(\phi(x),\hatphi(x)\right)\\
        \leq & 2\cdot R_0 + 1 + d_{S_G}\left(\phi(s\ast_{\gr} x),s\ast \phi(x)\right).
    \end{align*}
    where the second inequality follows from \Cref{lem:identity-at-finite-distance} and the right invariance of $d_{S_G}$.
\end{proof}

If $G$ is not a Carnot group, then Cornulier's constant $e_D$ takes values in $\left(0,1-\frac{1}{s_G}\right]$. \Cref{cor:bound-on-word-length} shows that in this case  $d_{S_G}\left(\hatphi(s\ast_{\gr} x), \hatphi(x)\right)$ is bounded above by a function that is $O(2^{\alpha\cdot k})$ for some $\alpha\in (0,1)$.

\begin{proof}[Proof of the first statement of \Cref{thmintro:quantitativeMain}]
    Let $G$ be a simply connected nilpotent Lie group. By choosing a suitable canonical basis we may assume that $e_G=e_D$.  We will show that there is an $(L^p,L^0)$-integrable orbit equivalence coupling between $\gr(G)$ and $G$. To see that this coupling is $(L^p,L^p)$-integrable it suffices to observe that we did not use at any point in this section that our coupling goes from $\gr(G)$ to $G$ and thus all results also hold for the inverse coupling from $G$ to $\gr(G)$.
    
    The actions $\gr(G)\curvearrowright (X_{\gr(G)},\mu_{\gr(G)})$ and $G\curvearrowright (X_{G},\mu_G)$ are essentially free. Thus, it suffices to prove that they define an $L^p$-integrable orbit equivalence coupling for some $p>1$.
    
    By \Cref{lem:nilpotent-explicit-Folner-tiling} and \Cref{cor:bound-on-word-length} there is a constant $C>0$ such that $(F_{k,\gr(G)})$ and $(F_{k,G})$ are $(\epsilon_k,R_k)$-F\o lner tiling sequences for $\epsilon_k\leq C\cdot 2^{-k}$ and $d_{S_G}\left(\hatphi(s\ast_{\gr} x),\hatphi(x)\right)\leq C\cdot 2^{\left(1-\frac{1-e_G}{s_G}\right)\cdot k}$ for $\rho(s\ast_{\gr} x,x)\leq k$. 

    Thus, for all $s\in S_{\gr(G)}$ and all $k\in \mathbb{N}$ we have
    \[
        \mu\left(\left\{x\in X_{\gr(G)}\mid d_{S_G}\left(\hatphi(s\ast_{\gr} x), \hatphi(x)\right)> C\cdot 2^{\left(1-\frac{1-e_G}{s_G}\right)\cdot k}\right\}\right)\leq \mu\left(\left\{x\in X\mid \rho\left(s\ast_{\gr} x,x\right)>k\right)\right\}\leq \epsilon_k.
    \]
    We deduce that for $\psi_{p,\epsilon}(x)=\frac{x^p}{\log(x)^{1+\epsilon}}$ and all $s\in S_{\gr(G)}$ we have
    \begin{align*}
        &\int_{X}\psi_{p,\epsilon}\left(d_{S_G}(\hatphi(s\ast_{\gr} x),\hatphi(x))\right) d\mu(x)\\
        &\leq \psi_{p,\epsilon}(C) + \sum_{k=1} ^{\infty}\left[ \psi_{p,\epsilon}\left(C\cdot 2^{\left(1-\frac{1-e_G}{s_G}\right)\cdot k}\right)\right.\\
        &\hspace{2.5cm}\left.\cdot \mu\left(\left\{x\in X_{\gr(G)}\mid C\cdot 2^{\left(1-\frac{1-e_G}{s_G}\right)\cdot (k-1)}<d_{S_G}\left(\hatphi(s\ast_{\gr} x), \hatphi(x)\right)\leq C\cdot 2^{\left(1-\frac{1-e_G}{s_G}\right)\cdot k}\right\}\right)\right]\\
        &\leq \psi_{p,\epsilon}(C)+ \sum_{k=1}^{\infty} (\epsilon_{k-1}-\epsilon_k)\cdot \psi_{p,\epsilon}\left(C\cdot 2^{\left(1-\frac{1-e_G}{s_G}\right)\cdot k}\right)\\
        &\leq \psi_{p,\epsilon}(C) + C^{1+p}\cdot \sum_{k=1}^{\infty}2^{-k}\cdot \frac{2^{\left(1-\frac{1-e_G}{s_G}\right)\cdot k\cdot p}}{\left(\log(C)+\left(1-\frac{1-e_G}{s_G}\right)\cdot k\right)^{1+\epsilon}}\\
        &= \psi_{p,\epsilon}(C) + C^{1+p}\cdot \sum_{k=1}^{\infty}\cdot \frac{2^{k\cdot \left(\left(1-\frac{1-e_G}{s_G}\right)\cdot p-1\right)}}{\left(\log(C)+\left(1-\frac{1-e_G}{s_G}\right)\cdot k\right)^{1+\epsilon}},
    \end{align*}
    which converges for every $\left(1-\frac{1-e_G}{s_G}\right)\cdot p-1\leq 0$ and $\epsilon>0$. The latter is the case if and only if $\epsilon>0$ and $p\leq\frac{s_G}{s_G-(1-e_G)}$. Since for $p<\frac{s_G}{s_G-(1-e_G)}$ we have that $x^p=o(\psi_{\frac{s_G}{s_G-(1-e_G)},\epsilon})$ and the argument is symmetric in $\gr(G)$ and $G$, it follows that $G$ and $\gr(G)$ are $L^p$-orbit equivalent for all $p<\frac{s_G}{s_G-(1-e_G)}$. 
\end{proof}

\begin{remark}\label{rmk:p-bounds-for-general-D}
    If in the first statement of the proof of \Cref{thmintro:quantitativeMain} we choose another canonical basis, then the same proof shows that the induced orbit equivalence coupling is $L^p$-integrable for all $p< \frac{s_G}{s_G-(1-e_D)}\leq \frac{s_G^2}{s_G^2-s_G+1}$, where the second inequality follows from the fact that $e_D\leq 1-\frac{1}{s_G}$ for all grading operators $D$. We will require this to give explicit bounds on the $p$ for which two lattices with isomorphic associated Carnot graded groups are $L^p$ OE, as in this case we can not choose the canonical bases freely.
\end{remark}

\section{Better integrability when the difference is in the centre}\label{sec:OE-same-Carnot-central}
In this section we show that the upper bound of $\frac{s_G}{s_G-(1-e_G)}$ on the $p$ for which there is an $L^p$-orbit equivalence between a simply connected Lie group $G$ and its associated Carnot graded group $\gr(G)$ is not optimal in general. To see this we consider the case when the difference between the group structures on $G$ and $\gr(G)$ is in the last term of their lower central series. More precisely, we will now assume that there is an isomorphism $\psi \colon  \gr(G)/\gamma_{s_G}(\gr(G))\to G/\gamma_{s_G}(G)$. 

\begin{remark}
To obtain the bounds in this section we choose a canonical basis whose associated grading operator $D$ satisfies $e_D=e_G$. All bounds in this section also hold for a general canonical basis for $\mathfrak{g}$ and its associated grading operator $D$ if we replace $e_G$ by $e_D$ everywhere.
\end{remark}

We will use the same notation as in \Cref{sec:OE-same-Carnot}. In particular, we denote by $(F_{k,\gr(G)})$ and $(F_{k,G})$ the $(\epsilon_k,R_k)$-F\o lner tiling sequences for $\gr(G)$ and $G$, and by $\phi=(\phi_k)$ the map between these tiling sequences. We further denote by $(\overline{F_{k,\gr(G)}})$, $(\overline{F_{k,G}})$ and $\overline{\phi}=(\overline{\phi}_k)$ the induced tiling sequences and map on $\overline{\gr(G)}:=\gr(G)/\gamma_{s_G}(\gr(G))$ and $\overline{G}:=G/\gamma_{s_G}(G)$. We denote by $\pi_{\gr(G)} \colon \gr(G)\to \overline{\gr(G)}$ and $\pi_G \colon G\to \overline{G}$ the projections. 

Note that we can choose the isomorphism $\gr(\mathfrak{g})\to \mathfrak{g}$ used to construct the map $\phi$ so that it is compatible with $\psi$, meaning that for all $k\geq 0$ we have $ \overline{\phi}_k=\pi_G \circ \phi_k = \psi \circ \pi_{\gr(G)}|_{F_{k,\gr(G)}}$. This implies in particular that $\pi_G(v_{\ell})=1\in G$ for the projection of the error terms $v_{\ell}$ in \Cref{eq:def-vi}, yielding the following strengthened version of \Cref{lem:tiled-error-terms}.

\begin{lemma}\label{lem:tiled-error-terms-central}
    Let $s\in S_{\gr(G)}$ and $x\in X_{\gr(G)}$ with $\rho(s\ast_{\gr} x,x)=k$ and let $g:= f_0(x)\ast_{\gr} \cdots \ast_{\gr} f_k(x) \in T_{k,\gr(G)}$. Then there are elements $v_0,\cdots,v_k\in \gamma_{s_G}(G)$ of word length $|v_\ell|\lesssim 2^{\ell\cdot e_G}$ such that $s\ast \phi(g)= v_0\ast \phi_0(f_0(s\ast_{\gr} x))\ast \cdots \ast v_k\ast \phi_k(f_k(s\ast_{\gr} x))$.
\end{lemma}

Since $\gamma_{s_G}(G)\leq Z(G)$, we deduce the following strengthened version of \Cref{prop:bound-on-word-length}.

\begin{proposition}\label{prop:bound-on-word-length-central}
    Let $s\in S_{\gr(G)}$ and let $x\in X_{\gr(G)}$ with $\rho(s\ast_{\gr} x,x)=k$. Let $g:=f_0(x)\ast_{\gr} \cdots \ast_{\gr} f_k(x)\in T_{k,\gr(G)}$. Then $s\ast_{\gr} g = f_0(s\ast_{\gr} x) \ast_{\gr} \cdots \ast_{\gr} f_k(s\ast_{\gr} x)\in T_{k,\gr(G)}$ and there is a constant $C>0$ that only depends on $\gr(G)$, $G$ and the chosen identification of their Lie algebras such that 
    \[
        d_{S_G}\left(\phi(s\ast_{\gr} g), s\ast \phi(g)\right) \leq C \cdot 2^{k \cdot e_G}.
    \]
\end{proposition}
\begin{proof}
    By the same line of argument as in the proof of \Cref{prop:bound-on-word-length} we deduce that
    \[
        s\ast \phi(g)= \left(\prod_{\ell=0}^k y_\ell\ast\left[y_\ell^{-1},x_\ell^{-1}\right]\right) \ast \phi_0(f_0(s\ast_{\gr} x))\ast \cdots \ast \phi_k(f_k(s\ast_{\gr} x)). 
    \]
    However, by definition, $y_{\ell}=v_{\ell}\in \gamma_m(G)\leq Z(G)$ and thus the $[y_{\ell},x_{\ell}]$ vanish in $G$. Hence,
    \[
        d_{S_G}(\phi(s\ast_{\gr} g),s\ast \phi(g))= d_{S_G}\left(\prod_{\ell=0}^k y_{\ell}, 1 \right)\lesssim \sum_{\ell=0}^k 2^{\ell \cdot e_G }\lesssim 2^{k \cdot e_G}.
    \]
    This completes the proof.
\end{proof}

\begin{proof}[Proof of the second statement of \Cref{thmintro:quantitativeMain}]
    Using the strengthened bounds from \Cref{prop:bound-on-word-length-central}, the same arguments as in the proof of \Cref{thmintro:quantitativeMain} at the end of \Cref{sec:OE-same-Carnot} imply $p_{OE}(G,\gr(g))\geq \frac{1}{e_G}$ as desired.
\end{proof}

\begin{remark}
    Our proof of the second statement of \Cref{thmintro:quantitativeMain} only uses that the difference betweeen the Lie algebra structures of $\gr(\mathfrak{g})$ and $\mathfrak{g}$ is in the center of their Lie algebras. It is possible that some elements of the center of the two Lie algebras are not contained in the last term of the lower central series. Our proof yields the same results in this more general setting.
\end{remark}

\section{\texorpdfstring{$L^p$}{Lp} OE between finitely generated nilpotent groups with isomorphic Carnot}\label{sec:finitely-generated-same-Carnot}
The goal of this section is to prove \Cref{mainthm:Carnot} for virtually nilpotent finitely generated groups. In Section \ref{sec:OE-same-Carnot}, we proved that a simply connected nilpotent Lie group $G$ is $L^p$ OE to its Carnot $\gr(G) $ for some $p>1$. By transitivity of $L^p$ OE for $p\geq 1$, it follows that any two simply connected nilpotent Lie groups with isomorphic Carnots are $L^p$ OE for some $p>1$ (depending on the groups), implying \Cref{mainthm:Carnot} in the simply connected nilpotent case. We want to establish this for virtually finitely generated nilpotent groups with isomorphic Carnot. Since any virtually finitely generated nilpotent group is virtually torsion-free, it is bilipschitz equivalent to a product $\Lambda\times \Z/n\Z$ where $\Lambda$ is torsion-free. Hence we are reduced to treating the case of groups which are direct products of a torsion free nilpotent (and therefore a lattice in a simply connected nilpotent Lie group) with a finite cyclic group.

Our first step will be to consider a case where we can directly exploit the construction from Section \ref{sec:OE-same-Carnot}. In order to make it precise, we use the fact that if $G$ admits a lattice, then one can find a quite specific one, defined as follows.
Let $\Gamma$ be a lattice in $G$. By \cite[Theorem 4.2]{EasHut-23}, there exists a bracket closed lattice $A$ of $\mathfrak{g}$ such that $A\subset \log \Lambda$.
A well-known induction on the dimension of $\gamma_{s_G}(\mathfrak{g})$ shows that $A\cap \gamma_{s_G}(\mathfrak{g})$ is a lattice in $\gamma_{s_G}(\mathfrak{g})$. By induction on $s_G$, one can thus choose a basis $\left\{X_{ij}\right\}$ of $A$, such that for each $i$, the $X_{ij}$'s freely generate a complement $\mathfrak{g}_i$ of $\gamma_{i+1}(\mathfrak{g})$ in $\gamma_{i}(\mathfrak{g})$. Let us denote $A_i$ the additive subgroup spanned by the $X_{ij}$'s. By construction, we have $[A_i,A_j]\subset \bigoplus_{k\geq i+j}A_k$. By definition of the modified bracket in $\gr(\mathfrak{g})$, we have $[A_i,A_j]_\gr\subset A_{i+j}$. In particular, $A$ is bracket-closed both for the Lie bracket coming from $\mathfrak{g}$ and the Lie bracket coming from $\gr(\mathfrak{g})$. By \cite[Lemma 4.3]{TointonTessera}, after multiplying $A$ by a positive integer $N$ (which only depends on $s_G$) we can further assume that $\Lambda:=\exp(A)$ and $\Lambda_\gr:=\exp_\gr(A)$ are respectively subgroups of $G$ and $\gr(G)$. Note that multiplication by $N$ does not affect the aforementioned properties of the basis $\left\{X_{ij}\right\}$, which therefore yields a canonical basis for $\mathfrak{g}$ (and $\gr(\mathfrak{g})$), which is compatible with $\Lambda$ and $\Lambda_{\gr}$. As before we denote by $D$ the grading operator associated with this choice of canonical basis. We emphasize that here our choice of canonical basis is constraint by the lattice $\Lambda$ and we cannot thus choose it such that $e_G=e_D$. Recall that nevertheless we have that $e_D\in \left[0,1-\frac{1}{s_G}\right]$ is bounded above in terms of the nilpotency class of $G$. 

Restricting the normal forms defined with respect to the basis $\left\{X_{ij}\right\}$ for elements of $G$ and $\gr(G)$ in Section \ref{sec:OE-same-Carnot} to elements of $\Lambda$ and $\Lambda_\gr$ yields generators and normal forms for $\Lambda$ and $\Lambda_{\gr}$ so that the maps $\phi,~ \hat{\phi}:X_{\gr(G)}\to X_G$ restrict to induce OE-couplings between the standard Borel probability measure spaces $X_{\Lambda_{\gr}}$ and $X_\Lambda$ associated to the F\o lner tilings for $\Lambda_{\gr}$ and $\Lambda$ given by \Cref{lem:normal-form}. Observe that \Cref{lem:tiled-error-terms} applies to this coupling, where by construction the $v_\ell$ are now in $\Lambda$. We have thus maneuvered ourselves into a situation, where the arguments in \Cref{sec:OE-same-Carnot} apply verbatim to the OE-coupling $X_{\Lambda_{\gr}}\to X_{\Lambda}$. As a consequence, we deduce the following result.
\begin{proposition}
    Let $G$ be a simply connected nilpotent Lie group of nilpotency class $s_G$ and let $\Gamma < G$ be a lattice. Then there is a lattice $\Lambda < G$ which has finite index in $\Gamma$ and a lattice $\Lambda_{\gr}<\gr(G)$, such that $\Lambda$ and $\Lambda_{\gr}$ are $L^p$-orbit equivalent for every $p<\frac{s_G}{s_G-(1-e_D)}\leq \frac{s_G^2}{s_G^2-s_G+1}$.
\end{proposition}
 
To deduce the general result from this specific case, we start with a general observation.
\begin{proposition}\label{prop:BE-lattices-in-Carnot}
Let $\Gamma$ and $\Lambda$ be lattices in an almost connected Lie group $G$ whose connected component of the identity $G^0$ admits a one-parameter group of automorphisms not preserving the Haar measure (e.g.\ $G^0$ is Carnot). Then, $\Gamma$ and $\Lambda$ are bilipschitz equivalent (hence $L^{\infty}$ OE).
\end{proposition}
\begin{proof}
Let $\Lambda^0=\Lambda\cap G^0$ and $\Gamma^0=\Gamma\cap G^0$, and let $m=[\Lambda : \Lambda^0]$ and $n=[\Gamma : \Gamma^0]$. Note that $\Lambda$ and $\Gamma$ are respectively bilipschitz equivalent to $\Lambda^0\times \Z/m\Z$ and $\Gamma^0\times \Z/n\Z$, which are both lattices in $G^0\times (\Z/m\Z)\times (\Z/n\Z) $. Hence we can assume without loss of generality that $G$ itself admits a one-parameter group $T$ of automorphisms not preserving the Haar measure.
On replacing $\Lambda$ by $t(\Lambda)$ for some $t\in T$, we can assume that $\Lambda$ and $\Gamma$ have the same covolume in $G$. Then the statement follows from \cite[Proposition 1.5]{GenTe}.
\end{proof}
\begin{corollary}
    Let $G$ and $G'$ be two simply connected nilpotent Lie groups with isomorphic Carnot, and let respectively $\Gamma$ and $\Gamma'$ be lattices in $G$ and $G'.$
    Then for all $j,k\geq 1$,  $\Gamma\times \Z/j\Z$ and $\Gamma'\times \Z/k\Z$ are $L^p$ OE for all $p<\frac{s_G^2}{s_G^{2}-s_G +1}$.
\end{corollary} 
\begin{proof}
Let $\Lambda,\Lambda_\gr, \Lambda',\Lambda'_\gr$ be lattices in respectively $G$, $\gr(G)$, $G'$ and $\gr(G')$ such that $\Lambda$ has index $n\geq 1$ in $\Gamma$, $\Lambda'$ has index $m\geq 1$ in $\Gamma'$ and such that $\Lambda$ and $\Lambda_\gr$, and respectively $\Lambda'$ and $\Lambda'_\gr$ are $L^p$ OE for some $p>1$. Note that $\Gamma\times \Z/j\Z$ is $L^{\infty}$ OE to $\Lambda\times \Z/jn\Z$ and similarly that $\Gamma'\times \Z/k\Z$ is $L^{\infty}$ OE to $\Lambda'\times \Z/km\Z$.
By transitivity of $L^p$ OE for $p\geq 1$, $\Gamma\times \Z/j\Z$ is $L^p$ OE to $\Lambda_\gr\times \Z/jn\Z$ and that $\Gamma'\times \Z/k\Z$ is $L^p$ OE to $\Lambda'_\gr\times \Z/km\Z$. Note that $\Lambda_\gr\times \Z/jn\Z$ and $\Lambda'_\gr\times \Z/km\Z$ are both lattices in $\gr(G)\times \Z/jknm\Z$, and therefore are $L^{\infty}$ OE by Proposition \ref{prop:BE-lattices-in-Carnot}. Composing these couplings yields the desired $L^p$ OE between $\Gamma$ and $\Gamma'$, where the upper bound on $p$ is an immediate consequence of the fact that for every grading operator $D$ on a Lie algebra of nilpotency class $s_G$ we have $e_D\leq 1-\frac{1}{s_G}$.
\end{proof}

\section{Obstructions from central extensions}\label{sec:obstruction}
In this section we prove constraints on the existence of an $L^p$ ME between a pair of simply connected nilpotent Lie groups with isomorphic associated Carnot group in terms of the existence of maximally distorted central extensions. As a consequence we will be able to prove that $G_{m,3}$ and $L_m\times \mathbb{R}^2$ are not $L^p$ ME for any $p>m$.

\subsection{Second cohomology group and central extensions}

Before stating and proving our main results, we recall how the cohomology group $H^2(G,A)$ parametrizes the central extensions of a group $G$ by the commutative ring $A$. The group $H^2(G,A)$ is defined as the quotient of $Z^2(G,A)$, the set of functions $\varphi \colon G^2 \to A$ such that $\varphi(1,1)=0$ and $\varphi(g_1,g_2g_3)+\varphi(g_2,g_3) = \varphi(g_1g_2,g_3)+\varphi(g_1,g_2)$, viewed as an abelian group with pointwise operation, by its subgroup $B^2(G,A)$ of coboundaries, i.e. maps of the form $(g_1,g_2)\mapsto \psi(g_1)+\psi(g_2)-\psi(g_1 g_2)$ for some function $\psi\colon G \to A$ such that $\psi(1)=0$. Every $\varphi \in Z^2(G,A)$ gives rise to a central extension 
\[ 1 \to A \to E \to G \to 1\]
together with a (set-theoretical) section $s\colon G \to E$ by setting $E = A \times G$ with the group operation $(a,g_1)(b,g_2) = (a+b+\varphi(g_1,g_2),g_1 g_2)$, and $s(h)=(0,h)$. Conversely, every central extension $E$ of $G$ by $A$ and section $s$ of $G$ in $E$ give rise to an element of $Z^2(G,A)$, by setting $\varphi(g_1,g_2) = s(g_1) s(g_2) s(g_1 g_2)^{-1}$. Two elements in $Z^2(G,A)$ define isomorphic extensions if and only if they differ by an element in $B^2(G,A)$.

We end this section with an auxiliary result that we will require later.

\begin{lemma}\label{lem:commute}
	Let $G$ be a group, let $A$ be an abelian group, and let $\widetilde{G}$ be a central extension with associated cocycle $\psi\colon G^2\to A$. If $g, h \in G$ are commuting elements and $\tilde{g},\tilde{h}\in \widetilde{G}$ are lifts of these elements into $\widetilde{G}$, then $[\tilde{g},\tilde{h}]=\psi(g,h)-\psi(h,g)$.
\end{lemma}
\begin{proof}
	We can consider the group extension as $A\times G$ with $(a,g)(b,h)=(a+b+\psi(g,h),gh)$. Then we have that
	\begin{align*}
		[\tilde{g},\tilde{h}] & = (\psi(g,h),gh) \; (-\psi(g,g^{-1}),g^{-1}) \; (-\psi(h,h^{-1}),h^{-1})\\
		& = (\psi(g,h)+ \psi(gh,g^{-1}) - \psi(g,g^{-1}),h) \; (-\psi(h,h^{-1}),h^{-1})\\
		& = (\psi(g,h)+ \psi(hg,g^{-1}) - \psi(g,g^{-1}),e_G)\\
		& = (\psi(g,h)- \psi(h,g),e_G),
	\end{align*}
	which concludes this lemma.
\end{proof}

\subsection{Obstructions to the existence of an $L^p$ ME}

In what follows we will use the notation of Section \ref{sec:OE-same-Carnot}: we identify $G$ with its Lie algebra $\mathfrak{g}$, denoting the group law by $\ast$; we fix vector complements $\mathfrak{m}_i$ of $\gamma_{i+1}(\mathfrak{g})$ in $\gamma_{i}(\mathfrak{g})$, and use them to identify $\mathfrak{g}$ with $\gr(\mathfrak{g})$, denoting the group law in $\gr(G)$ by $\ast_\gr$. We also fix a word metric $d_{\gr(G)}$ on $\gr(G)$.
Given $u=u_1+\ldots u_{s_G}\in \gr(\mathfrak{g})$, decomposed according to the decomposition $\gr(\mathfrak{g})=\mathfrak{m}_1\oplus \ldots \oplus \mathfrak{m}_s$, and $t>0$, we denote $t\bullet u := tu_1+t^2u_2\ldots +t^{s_G}u_{s_G}$. This defines a one-parameter family of Lie algebra automorphisms of  $\gr(\mathfrak{g})$ that preserves the grading. As before we will also denote for $u\in \mathfrak{g}$ and $t>0$ by $t\bullet u$ the element obtained from the $t$-action on the corresponding element in $\gr(\mathfrak{g})$. This induces a one-parameter family of vector space automorphisms of $\mathfrak{g}$; they are, however, not Lie algebra automorphisms. Finally, we equip $\gr(G)$ with a Carnot-Caratheodory metric $d_{\gr(G)}$ compatible with this grading preserving isomorphism: i.e.\ such that $d_{\gr(G)}(t\bullet g,t\bullet g')=td_{\gr(G)}(g,g')$.

\begin{proposition}[{Cantrell \cite[Proposition 5.5]{Cantrell}}]\label{prop:cantrell}
    Let $G$ and $H$ be simply connected nilpotent Lie groups which are $L^1$-measure equivalent and let $\alpha\colon G\times X_H\to H$, $\beta\colon H\times X_G \to G$ be the associated cocycles. Then there exists a bilipschitz group isomorphism $\phi\colon \gr(G) \to \gr(H)$ such that the following holds:
Fix $\bar{g}\in \gr(G)$. For all $\eps>0$ there exists $\delta>0$ and $N\in \N$ such that whenever $g\in G$, $n\geq N$ are such that \[d_{\gr(G)} \left(\frac{1}{n}\bullet g,\bar{g}\right)\leq \delta,\] then with probability at least $1-\eps$, we have 
\[d_{\gr(H)} \left(\frac{1}{n}\bullet \alpha(g,x),\phi(\bar{g})\right)\leq \eps.\]
In other words, if $g_n\in G$ is such that $\frac{1}{n}\bullet g_n$ converges to $\bar{g}$
then $\frac{1}{n}\bullet \alpha(g_n,x)$ converges in probability to $\phi(\bar{g})$.
\end{proposition}
\begin{proof}
    The proof is a straight-forward adaptation of Cantrell's proof in the finitely generated case. We refer to \Cref{app:Cantrell} for further details.
\end{proof}

We have $d_{\gr(G)} \left(\frac{1}{n}\bullet g,\bar{g}\right)=\frac{1}{n}d_{\gr(G)} \left(g,n\bullet \bar{g}\right).$
Moreover, since the identification map $G\to \gr(G)$ is a SBE \cite{Cor-19}, the condition that 
$d_{\gr(G)} \left(g_n,n\bullet \bar{g}\right)=o(n)$ is equivalent to 
$d_{G} \left(g_n,n\bullet \bar{g}\right)=o(n)$. Therefore one has the following equivalent formulation of Cantrell's theorem, where $d_G$ and $d_H$ are word metrics on $G$ and $H$ for arbitrary compact generating sets.

\begin{proposition}\label{prop:cantrellBis}
    Let $G$ and $H$ be simply connected nilpotent Lie groups which are $L^1$-measure equivalent and let $\alpha\colon G\times X_H\to H$, $\beta\colon H\times X_G \to G$ be the associated cocycles. Then there exists a bilipschitz group isomorphism $\phi\colon \gr(G) \to \gr(H)$ such that the following holds:
Fix $\bar{g}\in \gr(G)$. For all $\eps>0$ there exists $\delta>0$ and $N\in \N$ such that whenever $g\in G$, $n\geq N$ are such that \[d_{G} \left(g,n\bullet \bar{g}\right)\leq \delta n,\] then with probability at least $1-\eps$, we have 
\[d_{H} \left(\alpha(g,x),\phi(g)\right)\leq \eps n.\]
In other words, if $g_n\in G$ is such that $\frac{1}{n}d_G( g_n,n\bullet\bar{g})$ tends to zero,
then $\frac{1}{n}d_H(\alpha(g_n,x),\phi(g_n))$ tends to zero in probability.
\end{proposition}

The following lemmas give us some tools to compare the $2$-cocycles of $G$ and $H$ and how they can be extended to the asymptotic cones.
First we show how a cocycle of $H$ induces a cocycle of $G$.

\begin{proposition}\label{prop:cocyles}
	Let $G$ and $H$ be simply connected nilpotent Lie groups that are $L^p$-measure equivalent for some $p>1$ and let $\alpha\colon G\times X_H\to H$ be an $L^p$-integrable, measure equivalence cocycle. Let  $r\le p$ be a natural number and let $C>1$. If $c\colon H^2\to \R$ is a $2$-cocycle of $H$ such that for   all $h_1,h_2\in H$ we have $|c(h_1,h_2)|\le C(|h_1|^r+|h_2|^r)$, then the map $c'\colon G^2\to \R$, defined by
	\[c'(g_1,g_2)=\int_{X_H}c(\alpha(g_1,x),\alpha(g_2,g_2^{-1}\cdot x))d\mu(x), \mbox{ for all } g_1,g_2\in G,\]
     is a $2$-cocycle of $G$.
\end{proposition}
\begin{proof}
	The map $c'$ is well defined, since
	\begin{align*}
		|c'(g_1,g_2)|&\le \int_{X_H}|c(\alpha(g_1,x),\alpha(g_2,g_2^{-1}\cdot x))|d\mu(x)\\ &\le \int_{X_H} C\left( |\alpha(g_1,x)|^r+|\alpha(g_2,g_2^{-1}\cdot x)|^r\right)d\mu(x)<\infty.
	\end{align*}
	To show that $c'$ is a cocycle we use that $\alpha$ is a measure equivalence cocycle (that is $\alpha(g_1g_2,x)=\alpha(g_1,g_2\cdot x)\alpha(g_2,x)$ for every $g_1,g_2\in G$ and $x\in X_H$) and $c$ is a $2$-cocycle (that is $c(h_2,h_3) - c(h_1h_2,h_3) + c(h_1,h_2h_3) - c(h_1,h_2)=0$ for every $h_1,h_2,h_3\in G$).
	This implies that
	\begin{align*}
	0 & = \int_{X_H} c(\alpha(g_2,x),\alpha(g_3,g_3^{-1}\cdot x)) - c(\alpha(g_1,g_2\cdot x)\alpha(g_2,x),\alpha(g_3,g_3^{-1}\cdot x))\\
	& \AlignRight{+ c(\alpha(g_1,g_2\cdot x),\alpha(g_2,x)\alpha(g_3,g_3^{-1}\cdot x)) - c(\alpha(g_1,g_2\cdot x),\alpha(g_2,x)) d\mu(x)}\\
	& = c'(g_2,g_3) - \int_{X_H} c(\alpha(g_1g_2,x),\alpha(g_3,g_3^{-1}\cdot x)) d\mu(x)\\
	& \AlignRight{+ \int_{X_H} c(\alpha(g_1,g_2\cdot x),\alpha(g_2g_3,g_3^{-1}\cdot x))d\mu(x) - \int_{X_H} c(\alpha(g_1,x),\alpha(g_2,g_2^{-1}\cdot x)) d\mu(x)}\\
	& = c'(g_2,g_3) - c'(g_1g_2,g_3)+ \int_{X_H} c(\alpha(g_1,x),\alpha(g_2g_3,g_3^{-1}g_2^{-1}\cdot x))d\mu(x) - c'(g_1,g_2)\\
	& = c'(g_2,g_3) - c'(g_1g_2,g_3)+ c'(g_1,g_2g_3)- c'(g_1,g_2)\\
	\end{align*}
	for any $g_1,g_2,g_3\in G$, which concludes this proof.
\end{proof}

\begin{proposition}\label{prop:cocycleBound}
	Let $G$ be a simply connected $(s-1)$-nilpotent Lie group for some $s$ and let $\tilde{G}$ be a central extension of $G$ by $\R$, which we assume to be $s$-nilpotent. Then $\tilde{G}$ has an associated $2$-cocycle $\psi\colon G^2\to \R$ with $|\psi(g_1,g_2)|_N\le C(|g_1|^{s-1}|g_2|+|g_1||g_2|^{s-1})$ for some $C>0$ independent of $g_1,g_2\in G$.
\end{proposition}
\begin{proof} 
Identify the central extension $\tilde{G}$ with the Lie algebra $\tilde{\mathfrak{g}}=\R\oplus \mathfrak{g}$ and fix a norm on $\tilde{\mathfrak{g}}$. We express the product $(0,g_1)(0,g_2)=(\psi(g_1,g_2),g_1 g_2)$ in $\tilde{\mathfrak{g}}$ using the BCH formula and similarly write the product $g_1 g_2$ as an element of $\mathfrak{g}$ using the BCH formula. Since $\widetilde{G}$ and $G$ are both nilpotent of class $\leq s$, we deduce from the precise form of the BCH formula that 
\[
    (\psi(g_1,g_2),0)=(0,g_1)(0,g_2) - (0,g_1g_2)
\]
is a sum of commutators of length at most $s$ in $g_1$ and $g_2$, where in each commutator each of $g_1$ and $g_2$ appears at least once. Exploiting multilinearity of the Lie bracket and the fact that for all integers $a,~b>0$ with $a+b\leq s$ we have $|g_1|^a |g_2|^b\leq O(|g_1|^{s-1}|g_2|+|g_1||g_2|^{s-1})$, we deduce the assertion.
\end{proof}

\begin{lemma}\label{lem:cocyleToCone}
	Let $G$ be a simply connected nilpotent Lie group and let $c\colon G^2\to \R$ be the $2$-cocycle from \Cref{prop:cocycleBound} such that the corresponding extension is $s$-nilpotent. Then there is a constant $C>0$ such that the following holds. Let $R>0$, $\varepsilon>0$, and let $u$, $v$, $\bar{u}$, $\bar{v}$ be elements of $G$ with $d_G(u,\bar{u})<\eps R$, $d_G(v,\bar{v})<\eps R$ and $|u|,|v|,|\bar{u}|,|\bar{v}|<R$. Then $|c(u,v^{-1})-c(\bar{u},\bar{v}^{-1})|<\eps CR^s$.
\end{lemma}
\begin{proof}
Using the triangle inequality and the cocycle relation, we obtain:	
	\begin{align*}
	\left|c(u,v^{-1})\!-\!c(\bar{u},\bar{v}^{-1})\right|\! & \le \left|c(u,v^{-1})-c(u,\bar{v}^{-1})\right| + \left|c(u,\bar{v}^{-1})-c(\bar{u},\bar{v}^{-1})\right|\\
	& \le \left|c(u\bar{v}^{-1},\bar{v}v^{-1})-c(\bar{v}^{-1},\bar{v}v^{-1})\right| + \left|c(u\bar{u}^{-1},\bar{u}\bar{v}^{-1})-c(u\bar{u}^{-1},\bar{u})\right|\\
	& \le \left|c(u\bar{v}^{-1},\bar{v}v^{-1})\right| \!+\! \left|c(\bar{v}^{-1},\bar{v}v^{-1})\right| \!+\! \left|c(u\bar{u}^{-1},\bar{u}\bar{v}^{-1})\right| \!+\! \left|c(u\bar{u}^{-1},\bar{u})\right|.
	\end{align*}
	We then conclude thanks to  \Cref{prop:cocycleBound}. 
\end{proof}
\Cref{lem:cocyleToCone} implies that the $2$-cocycle $c$ extends to the asymptotic cone. We will make this precise and use it in the following result, in which we compare the cocycles $c$ and $c'$. Note that in its statement, we implicitly identify $G$ and $H$ with $\gr(G)$ and $\gr(H)$ respectively.

\begin{lemma}\label{lem:convcocycle}
	Let $s\in\N$, and let $G$ and $H$ simply connected, $(s-1)$-nilpotent Lie groups. Assume that there exists a $(L^p,L^1)$-ME from $G$ to $H$ such that $p>s$, and let $\alpha:G\times X_H\to H$ be the cocycle from $G$ to $H$. 
Let $\tilde{H}$ be a central extension of $H$ by $\R$ and let $c\colon H^2\to \R$ be the associated $2$-cocycle.
Let $\phi\colon \gr(G)\to \gr(H)$ be the isomorphism given by Proposition \ref{prop:cantrell}.
Then, every pair of sequences $u_n,v_n\in G$  such that $|u_n|_G$ and $|v_n|_G$ are in $O(n)$ satisfies
\[\lim_{n\to \infty}\frac{1}{n^s}\left(c'(u_n,v_n))-c( \phi(u_n),\phi(v_n))\right)=0.\]
\end{lemma}
\begin{proof}
By compactness of the sequences $\frac{1}{n}\bullet u_n$ and $\frac{1}{n}\bullet v_n$ we can assume without loss of generality that they converge in $\gr(G)$.
Fix $\eps>0$. In the sequel, we denote $\bar{u}_n=\phi(u_n)$ and $\bar{v}_n=\phi(v_n)$.
    We apply \Cref{prop:cantrellBis} to elements $(u_n)_n$, $(v_n^{-1})_n$, $\phi((u_n)_n)=(\bar{u}_n)_n$ and $\phi((v_n^{-1})_n)=(\bar{v}_n^{-1})_n$,
    to obtain that, for $n$ large enough, there exists a subset $A\subset X_H$ such that $\mu(X_H\setminus A)\leq \eps$ and such that for all $x\in A$, the quantities  
   $\frac{1}{n}d_{H}(\alpha(u_n,x),\bar{u}_n)$ and $\frac{1}{n}d_{H}(\alpha(v_n^{-1},x),\bar{v}_n^{-1})$ 
   are at most $\eps$. By \Cref{prop:cocycleBound} there exists a constant $C>0$ such that $|c(h_1,h_2)|\le C(|h_1|_H^s+|h_2|_H^s)$ for every $h_1,h_2\in H$.
	Then for all $g_1,g_2\in G$ and every measurable subset $Z\subset X_H$ we have that
	\begin{align*}
        &\left|\int_{Z}c(\alpha(g_1,x),\alpha(g_2,g_2^{-1}\cdot x))d\mu(x)\right|\\
		& \leq  \int_{Z}C(|\alpha(g_1,x)|_H^s+|\alpha(g_2,g_2^{-1}\cdot x)|_H^s)d\mu(x)\\
		& \leq  C\mu\left(Z\right)^{(p-s)/p}\left(\left(\int_{Z}|\alpha(g_1,x)|_H^{p}d\mu(x)\right)^{s/p}+\left(\int_{g_2^{-1}\cdot Z}|\alpha(g_2,x)|_H^{p}d\mu(x)\right)^{s/p}\right)\\
		& \leq C'\mu\left(Z\right)^{(p-s)/p} (|g_1|_G^s+|g_2|_G^s),
	\end{align*}
	where $C'=C\sup_{s\in S_G}\|\alpha(s,\cdot)\|_{p}<\infty$.\footnote{To see that this supremum is finite, we apply \cite[A.2]{BadFurSau-13} to the compact generating subset $S_G$ of $G$.}
	Since the identities $G\to \gr(G)$ and $H\to \gr(H)$ are SBEs, the isomorphism $\phi$ induces a SBE $G\to H$. Hence, there exists a $\tilde{C}>0$ such that $|u_n|_G,|v_n|_G,|\bar{u}_n|_H,|\bar{v}_n|_H \le \tilde{C}n$.
	The above discussion and \Cref{lem:cocyleToCone} thus imply
	\begin{align*} 
	& \left|\frac{1}{n^s}\left(c'(u_n,v_n)-c(\bar{u}_n,\bar{v}_n)\right)\right|\\
	 \leq& \frac{1}{n^s}\int_{X_H}\left|c(\alpha(u_n,x),\alpha(v_n,v_n^{-1}\cdot x))-c(\bar{u}_n,\bar{v}_n)\right|d\mu(x)\\
	 \leq&  \frac{1}{n^s}\int_{A}\left|c(\alpha(u_n, x),\alpha(v_n^{-1}, x)^{-1})-c(\bar{u}_n,\bar{v}_n)\right|d\mu(x)\\
	& +\frac{1}{n^s}\int_{X_H\setminus A}\left|c(\alpha(u_n,x),\alpha(v_n,v_n^{-1}\cdot x))\right|+\left|c(\bar{u}_n,\bar{v}_n)\right|d\mu(x)\\
	 \leq & \frac{1}{n^s}\int_{A}\eps C(\tilde{C}n)^sd\mu(x)\\
	  &+ \frac{1}{n^s}C'\mu\left(X_H\setminus A\right)^{(p-s)/p}\left(|u_n|_G^s + |v_n|_G^s\right)\\
	&+\frac{1}{n^s}\int_{X_H\setminus A}C\left(|\bar{u}_n|_H^s+|\bar{v}_n|_H^s\right)d\mu(x)\\
	 \leq & \frac{1}{n^s}\eps C(\tilde{C}n)^s+ \frac{1}{n^s}C'\eps^{(p-s)/p}2\left(\tilde{C}n\right)^s+\frac{1}{n^s}\eps C2\left(\tilde{C}n\right)^s\\
	 \leq & \eps C \tilde{C}^s+ C'\eps^{(p-s)/s}2\tilde{C}^s + \eps C2\tilde{C}^s,
	\end{align*}
	for all sufficiently large $n$. Since $\eps>0$ was arbitrary and $p>s$, this proves that the limit as $n\to \infty$ is equal to $0$.
\end{proof}

By combining the previous results, we can now prove \Cref{thm:extension}.

\begin{proof}[{Proof of \Cref{thm:extension}}]
	Let $c\colon H^2\to \R$ be the $2$-cocycle of $H$ associated to $\tilde{H}$ as in \Cref{prop:cocycleBound}.
	By \Cref{prop:cocyles} there is a $2$-cocycle $c'\colon G^2\to \R$ with \[c'(g_1,g_2)=\int_{X_H}c(\alpha(g_1,x),\alpha(g_2,g_2^{-1}\cdot x))d\mu(x)\]
	for every $g_1,g_2\in G$, where $\alpha\colon G\times X_H\to H$ is an $L^p$-integrable, measure equivalence cocycle associated with the coupling.

Assume for a contradiction that the nilpotency class of the extension $\tilde{G}$ associated to $c'$ is  $s-1$. As $\tilde{H}$ has nilpotency class $s$, there are $h_1,h_2,\ldots,h_{s}\in\tilde{H}$ such that $[h_1,h_2,\ldots,h_s]\neq e_{\tilde{H}}$. Note that $$[h_1^n,h_2^n,\ldots,h_s^n] = [h_1,h_2,\ldots,h_s]^{n^{s}}$$
	for every $n\in\N$.
	For $n\in \mathbb{N}$ and $1\leq i \leq s$, let $u_{i,n}\in G$ be such that $(u_{i,n})_n=\phi^{-1}((h_i^n)_n)$ and let $\tilde{u}_{i,n}\in\tilde{G}$ be any lift of $u_{i,n}$.
	If $\tilde{G}$ is of class  $s-1$, then $[\tilde{u}_{1,n},\ldots,\tilde{u}_{s,n}]=e_G$, which, by \Cref{lem:commute}, implies that

\[ c'(u_{1,n},[u_{2,n},\ldots,u_{s,n}]) - c'([u_{2,n},\ldots,u_{s,n}],u_{1,n})=0.\]
	However,
	\begin{align*}
		|c(h_1^n,[h_2^n,\ldots, h_{s-1}^n,h_s^n]) - c([h_2^n,\ldots, h_{s-1}^n,h_s^n],h_1^n)| = |[h_1^n,h_2^n,\ldots,h_{s-1}^n,h_s^n]|=O(n^s).
	\end{align*}
	This contradicts \Cref{lem:convcocycle}, which implies that \[c(h_1^n,[h_2^n,\ldots ,h_s^n]) - c'(u_{1,n},[u_{2,n},\ldots,u_{s,n}])\]
 as well as \[c([h_2^n,\ldots,h_s^n],h_1^n) - c'([u_{2,n},\ldots,u_{s,n}],u_{1,n})\]
 are in $o(n^s)$.
\end{proof}

\section{Questions}\label{sec:questions}
We conclude by raising some open problems about compactly generated locally compact groups of polynomial growth, which in this section we will simply call groups of polynomial growth. Some of them have already appeared earlier in the text. We restate them here to have them in one place.

One main result of our work is the classification of groups of polynomial growth up to $L^p$ ME for $p\leq 1$ in \Cref{thm:classification-Lp}. For $p>1$ we only obtain partial results.
\begin{problem}
    Classify  groups of polynomial growth up to $L^p$ ME and $L^p$ OE for all $p\in [1,\infty]$. In particular, one can ask:
    \begin{itemize}
        \item Is $p_{ME}(G,H)=p_{OE}(G,H)$ for all such groups $G$ and $H$?
        \item Is $p_{ME}(G,\gr(G))=\frac{1}{e_G}$ for every simply connected nilpotent Lie group, where $e_G$ is Cornulier's constant?
    \end{itemize} 
\end{problem}

\Cref{thm:classification-Lp} shows that for groups $G$ and $H$ with non-isomorphic Carnot the interval $I$ consisting of all $p$ for which $G$ and $H$ are $L^p$ ME is half-open. This raises the question if this is always the case when the groups are not $L^{\infty}$ ME:
\begin{question}
    Let $G$ and $H$ be groups of polynomial growth with isomorphic Carnot which are not $L^{\infty}$ ME. Is the interval of $p$ such that $G$ and $H$ are $L^p$ ME of the form $[0,p_{ME}(G,H))$?
\end{question}

As we mentioned, an observation of Shalom shows that for amenable groups QI implies $L^{\infty}$ ME. This raises the question if there is a converse, at least in the case of groups of polynomial growth:
\begin{question}
    Are $L^{\infty}$ ME groups of polynomial growth necessarily quasi-isometric? 
\end{question}
Note that the authors are not aware of a single pair of groups which are $L^\infty$-ME but not quasi-isometric.

At the other end of the spectrum, the only obstructions we currently have for groups with isomorphic Carnot being $L^p$ ME is the one provided by \Cref{thm:extension}. In particular, it is given by an integer, raising the question if this is merely a relic of our proof:
\begin{question}
    Are there groups of polynomial growth $G$, $H$ with isomorphic Carnot such that $p_{ME}(G,H)<1+\epsilon$ for every $\epsilon>0$?
\end{question}

\appendix

\section{Locally compact version of results of Bowen, Austin and Cantrell}\label{sec:AustinBowenLC}

In this appendix we explain how one can extend the main results of Austin, Bowen and Cantrell in \cite{Aus-16,Cantrell} to the locally compact case. Their proofs rely on certain lower bounds on the cardinality of the intersections of balls and their (pre-)images under the cocycles. Our main task in order to extend them to locally compact groups is to be able to replace ``cardinality'' by ``Haar measure''.
Note that if $G$ is discrete and $H$ is not, then the image of a ball of $G$ is likely to be negligible with respect to the Haar measure of $H$. This indicates that we should assume that both groups are either discrete or non-discrete.

Another crucial ingredient is to choose the fundamental domains so that both measures coincide (up to some scaling factor) and are positive on their intersection. This is essential in order to express the fact that the two cocycles are mutual ``inverses''. This is straightforward when the groups are both discrete, due to the fact that both measures on the fundamental domains coincide with the restrictions of the measure $\mu$ on $\Omega$. However, this argument no longer works when the groups are non-discrete, since the fundamental domains are negligible subsets of $(\Omega,\mu)$. 
Once again, it is quite obvious that we need to assume that both groups are either discrete or non-discrete.
Hence, from now on, we will assume that {\bf both groups are non-discrete}. 

The main tool to treat both is the notion of discrete cross-section, which roughly consists of decomposing the equivalence relation associated to any measure preserving action of a locally compact Polish group into a direct product of a ``compact smooth factor" and a ``discrete factor'' (with countable orbits).

After establishing these two ingredients (gathered in Proposition \ref{prop:intersecting-fundamental-domains}), the arguments of Bowen, Austin and Cantrell adapt readily to the non-discrete setting. We do not see much added value in repeating all of their arguments almost verbatim here and instead opt for only explaining the adaptations that are required. 

 Note that Cantrell's result subsumes Austin's theorem, providing a more direct proof which does not rely on Pansu's theorem. 
Moreover, a large portion of his proof is self-contained. Actually, in the published version, Cantrell claims that his proof is entirely self-contained. If so, it would have been unnecessary to prove that Austin's arguments extend to locally compact groups.  But one specific part of Cantrell's proof did not convince us, and we suspect that it might contain a gap: this concerns the demonstration that the limit morphism is bijective (see \cite[Proposition 5.7]{Cantrell}). Fortunately,
Cantrell gave a correct argument in an earlier arXiv version of his paper based on the material of Austin's paper. Here, we give a shorter alternative proof for this step, also based on Austin's paper, but on an easier preliminary result of his.

We also noted that a lemma in Cantrell's paper implicitly relies on a  probably well-known fact about the behaviour of the iterated Baker--Campbell--Hausdorff formula which we could not find in the literature. We include the proof and along the way also slightly simplify another argument in \cite{Cantrell} (avoiding an induction). 

We expect that our approach, and especially the material of Section \ref{section:crosssections} will also make it possible to generalise other results from the finitely generated to the locally compact case, such as the results of \cite{DKLMT-22}. For instance, Correia and Paucar will use this approach in a forthcoming work \cite{CorPau-25} to prove a version of \cite[Theorem 3.2]{DKLMT-22} for locally compact groups (in the context of measured subgroups).

\subsection{Using cross sections to create intersecting fundamental domains}\label{section:crosssections}

Let $G$ and $H$ be measure equivalent Polish locally compact groups and let $G\times H \curvearrowright (\Omega, \eta)$ be an action on a measure space inducing the measure equivalence.  Denote $X=X_H\subset \Omega$ (resp. $Y=Y_G\subset \Omega$) fundamental domains for the $H$- (resp. $G$-) action on $\Omega$. We choose Haar measures $\lambda_G$ and $\lambda_H$ on $G$ and $H$, and finite measures $\mu_X$ and $\mu_Y$ on $X$ and $Y$ such that we can identify both $(G\times Y,\lambda_G\otimes \mu_Y)$ and $(H\times X,\lambda_H\otimes \mu_X)$ with $(\Omega, \eta)$ via the maps
$(g,y)\mapsto g\cdot y$ and $(h,x)\mapsto h\cdot x$. As explained above, we intentionally chose both groups to be non-discrete.

Denote $T: G\times X \to X$, $S : H\times Y\to Y$ the induced actions of $G$ on $X$ and $H$ on $Y$ and $\alpha: G\times X \to H$, $\beta: H\times Y \to G$ the associated cocycles. By definition 
\[
\alpha(g,x)\cdot g\cdot x = T^g(x) \mbox{ and } \beta(h,y)\cdot h \cdot y = S^h(y),
\]
where $\alpha(g,x)\in G$ (resp. $\beta(h,y)\in H$) are the unique elements that map $g\cdot x$ (resp. $h\cdot y$) into $X$ (resp. $Y$).

We recall that for a locally compact Polish group $\Gamma$ a \emph{cross section} of a $\Gamma$-action on $(X,\mu_X)$ is a Borel subset $X_0\subseteq X$ such that there is an open neighbourhood $U\subset \Gamma$ of the identity with the following properties:
\begin{enumerate}
    \item the restricted action map $\theta: U\times X_0 \to X$, $(u,x_0)\mapsto T^ux_0$ is injective;
    \item the subset $G\cdot X_0$ is Borel and conull in $X$.
\end{enumerate}
Note that after possibly replacing $U$ by a subset $V$ with $V^2 \subset U$, we may assume that $\theta : U^2 \times X_0 \to X$ is injective. We will now always assume that $U$ is chosen in this way. 
Such a cross section exists for every non-singular essentially free $\Gamma$-action \cite{KyePetVae-15,KoiKyeRau-21}. In particular, we can choose cross sections $X_0\subset X$ and $Y_0\subset Y$ for the $G$-action on $(X,\mu_X)$ and the $H$-action on $(Y,\mu_Y)$, which we now fix. Denote by $U\subset G$ (resp. $V\subset H$) the corresponding open neighbourhoods of the identity. 

To adapt the arguments of Bowen, Austin and Cantrell to our context, we will now explain that we can use the cross section to modify our choice of $X$ and $Y$ such that the intersection $X\cap Y$ has positive measure, $\mu_X|_{X\cap Y} = \mu_Y|_{X\cap Y}$, and the restriction of the cocycles to this intersection exchanges the Haar measures of the groups. 

\begin{proposition}\label{prop:intersecting-fundamental-domains}
    We can choose new fundamental domains $X'$ and $Y'$ so that for the intersection $X'\cap Y'$ we have $\mu_{X'}|_{X'\cap Y'}=C\cdot \mu_{Y'}|_{X'\cap Y'}$ for some constant $C>0$ and $\mu_{X'}(X'\cap Y') = C\cdot \mu_{Y'}(X'\cap Y')>0$. If, moreover, the original cocycles $\alpha$ and $\beta$ are $L^1$-integrable, then so are the cocycles $\alpha'$ and $\beta'$ associated with the fundamental domains $X'$ and $Y'$. Finally, the restrictions of $\alpha'$ and $\beta'$ to $X'\cap Y'$ are measure preserving, in the sense that for a.e.\ $x\in X'\cap Y'$, the restriction of $\alpha'(\cdot,x)$ to $\{g\in G\mid T^gx\in X'\cap Y' \}$ equipped with $\lambda_G$ is an isomorphism of measure spaces onto its image equipped with $\lambda_H$ (and similarly for $\beta'$).
\end{proposition}

\begin{proof}
We start by observing that the injectivity of $\theta$ on $U^2\times X_0$ implies that for $u\in U$ the image $X_{0,u}=T^u(X_0)$ of $X_0$ is a cross section with corresponding set $U$ and $U\cdot X_0$ decomposes as disjoint union $U\cdot X_0 = \bigsqcup_u\in X_{0,u}$. Similarly, we obtain a decomposition $V\cdot Y_0= \bigsqcup_{u\in U} Y_{0,u}$ with $Y_{0,u}=S^u(Y_0)$ into cross sections. We further fix an isomorphism of the normalised restrictions of the Haar measures to $U$ and $V$\footnote{Note that this is where we use that both $G$ and $H$ are non-discrete.}
\[
    f: \left(U,\frac{1}{\lambda_G(U)}\lambda_G|_U\right)\to \left(V,\frac{1}{\lambda_H(V)}\lambda_H|_V\right).
\]

We will now explain how one can adapt the arguments in the proof of \cite[Lemma 3.10]{KoiKyeRau-21} to replace the fundamental domains $X$ and $Y$ by fundamental domains $X'$ and $Y'$ with the desired properties. For this we will for every $u\in U$ replace the cross sections $X_{0,u}$ and $Y_{0,f(u)}$ by isomorphic cross sections $X'_{0,u}$ and $Y'_{0,f(u)}$ whose intersection $X'_{0,u}\cap Y'_{0,f(u)}$ is non-negligible.

Let 
\[
    \mathcal{S}_u= \left\{(x,y)\in X_{0,u}\times Y_{0,f(u)}\mid y\in (G\times H)\cdot X_{0,u}\right\}.
\]
One argues as in the proof of \cite[Lemma 3.10]{KoiKyeRau-21} that the projection maps $\pi_{l,u}: \mathcal{S}_u\to X_{0,u}$ and $\pi_{r,u}: \mathcal{S}_u\to Y_{0,f(u)}$ are countable-to-one and that for every $s_u=(x_u,y_u)\in \mathcal{S}_u$ there exists a unique element $(g_{s_u},h_{s_u})\in (G\times H)$ with $y_u=(g_{s_u},h_{s_u})\cdot x_u$; the latter follows directly from the freeness of the $(G\times H)$-action on $\Omega$. This uniquely defines a map $(\phi_u,\psi_u): \mathcal{S}_u \to G\times H$, $s_u \mapsto (g_{s_u},h_{s_u})$.

Arguing again as in the proof of \cite[Lemma 3.10]{KoiKyeRau-21}, by the Novikov--Lusin Theorem there are disjoint partitions $\mathcal{S}_u = \bigsqcup_{n\in \mathbb{N}}E_{n,u}$ and $\mathcal{S}_u=\bigsqcup_{m\in \mathbb{N}} F_{m,u}$ such that the restrictions $\pi_{l,u}|_{E_{n,u}}$ and $\pi_{r,u}|_{F_{m,u}}$ are injective. In particular, there is some $M_u:= E_{n,u}\cap F_{m,u}$ such that $\mu_{X_{0,u}}(\pi_{l,u}(M_u))>0$ and $\mu_{Y_{0,u}}(\pi_{r,u}(M_u))>0$. The arguments in \cite{KoiKyeRau-21} and \cite{KyePetVae-15} further show that there is a constant $C=\frac{{\rm covol(Y_{0,f(u)})}}{{\rm covol(X_{0,u})}}=\frac{{\rm covol(Y_{0})}}{{\rm covol(X_{0})}}>0$ (independent of $u\in U$), such that the map 
\begin{equation}\label{eqn:equivalence-on-cross-sections}
\pi_{r,u}\circ \pi_{l,u}^{-1}: (\pi_{l,u}(M_u),\mu_{X_{0,u}}) \to (\pi_{r,u}(M_u),C\cdot \mu_{Y_{0,f(u)}})
\end{equation}
is an isomorphism of measure spaces, which preserves the orbit equivalence relation. Denote $M_X:= \bigsqcup_{u\in U} \pi_{l,u}(M_u)\subset U\cdot X_0$ and $M_Y:=\bigsqcup_{u\in U} \pi_{r,u}(M_u)\subset V\cdot Y_0$. By Kuratowski-Ryll-Nardzewski's measurable selection theorem, we can choose the sets $M_u$ such that the induced map $\Theta: M_X\to M_Y$ which restricts to $\pi_{r,u}\circ \pi_{l,u}^{-1}$  on $\pi_{l,u}(M_u)$  is measurable. By \eqref{eqn:equivalence-on-cross-sections}, $\Theta$ induces an isomorphism between the restriction of the product measures $(\frac{1}{\lambda_G(U)}\lambda_G|_U)\otimes \mu_{X_0}$ and $(\frac{1}{\lambda_H}(V)\lambda_H|_V)\otimes (C\cdot \mu_{Y_0})$. Moreover, it follows from \cite{KyePetVae-15} (see also \cite[Theorem 3.6(2)]{KoiKyeRau-21}) that these product measures coincide with $\mu_X|_{M_X}$, respectively $\mu_Y|_{M_Y}$, up to a constant $C_1>0$. In particular, $\mu_X(M_X)=C_1\mu_Y(M_Y)>0$.

For $R>0$ we define
\[
    M_{R,u}:= \left\{s\in M_u \mid (\phi_u(s),\psi_u(s))\in B_R^G(1)\times B_R^H(1)\right\}.
\]
The restriction of $\Theta$ to $M_{X,R}:=\bigsqcup_{u\in U} \pi_{l,u}(M_{R,u})$ is still an isomorphism of measure spaces with image $M_{Y,R}:=\bigsqcup_{u\in U} \pi_{r,u}(M_{R,u})$ with respect to the restrictions of the probability measures $\mu_{X}$ and $\mu_Y$ up to a multiplicative constant, which preserves the orbit equivalence relation. Since $M_X=\cup_{R\geq 0} M_{X,R}$ and $M_Y=\cup_{R\geq 0} M_{Y,R}$, there is an $R_0>0$ such that for all $R\geq R_0$ we have $\mu_X(M_{X,R})=C_1\cdot \mu_Y(M_{Y,R})>0$.

We have now maneuvered ourself into the position where we can define new fundamental domains $X'_R$ and $Y'_R$ for every $R\in (0,\infty]$ as follows. We define measurable maps $\phi_R: Y \to G$ and $\psi_R: X\to H$ by setting $\phi_R(x):=\phi_u(x,y)$ (resp. $\psi_R(y)=\psi_u(x,y)$) if $(x,y)\in M_{R,u}$ for $u\in U$, and $\phi_R(x)=1_G$ (resp. $\psi_R(y)=1_H$) otherwise. We then define
\[
    X'_R:= \left\{\psi_R(x)\cdot x \mid x\in X\right\} \mbox{ and } Y'_R:=\left\{(\phi_R(y))^{-1}\cdot y \mid y\in Y\right\}.
\]
By construction of $X'_R$ and $Y'_R$ the associated probability measures $\mu_{X'_R}$ and $\mu_{Y'_R}$ are isomorphic to  $\mu_X$ and $\mu_Y$ under the isomorphisms $X\to X'_R$, $x\mapsto \psi_R(x)\cdot x$ and $Y\to Y'_R$, $y\mapsto (\phi_R(y))^{-1}\cdot y$ and, moreover,
\begin{equation}\label{eqn:intersection}
    X'_R\cap Y'_R \supseteq \left\{\psi_u(x,y)\cdot x\mid (x,y)\in M_{R,u}, u \in U \right\}= \left\{ \phi_u(x,y)^{-1}\cdot y \mid (x,y)\in M_{R,u}, u \in U\right\}.
\end{equation}
It thus follows from the above discussion that for $R\geq R_0$ we have $\mu_{X'_R}|_{X'_R\cap Y'_R}=C_1\cdot \mu_{Y'_R}|_{X'_R\cap Y'_R}$ and $\mu_{X'_R}(X'_R\cap Y'_R)>0$.

A priori the inclusion in \Cref{eqn:intersection} could be strict. The important point for the remainder of our argument is that the right hand side has positive measure and is contained in $X_R'\cap Y_R'$. To simplify notation we will thus now assume that equality holds in \Cref{eqn:intersection}.

The cocycles $\alpha'_R: G\times X \to H$ and $\beta'_R: H\times Y \to G$ associated with the measure equivalence coupling with respect to the fundamental domains $X'_R$ and $Y'_R$ are then defined by
\[
    \alpha'_R(g,x)=\psi_R(T^g(x))\cdot \alpha(g,x)\cdot \psi_R(x)^{-1} \mbox{ and } \beta'_R(h,y)=\phi_R(S^h(y))^{-1}\cdot \beta(h,y)\cdot \phi_R(y).
\]

Since by definition $|\psi_R(x)|_H, |\phi_R(y)|_G\leq R$ for every $x\in X$ and $y\in Y$, we have
\[
    |\alpha'_R(g,x)|_H \leq 2 R + |\alpha(g,x)|_H \mbox{ and } |\beta'_R(h,y)|_G\leq 2 R + |\beta(h,y)|_H 
\]
for all $(g,x)\in G\times X$ and $(h,y)\in H\times Y$. We deduce that $\alpha'_R(g,\cdot)$ and $\beta'_R(h,\cdot)$ are $L^p$ integrable for $p\geq 1$ if and only if $\alpha(g,\cdot)$ and $\beta(h,\cdot)$ are.

We will conclude by proving that for almost every $x\in X'_R\cap Y'_R$ the map $\alpha'_R(\cdot, x)$ is measure preserving onto its image in $H$ when restricted to $\left\{g\in G \mid T^g x \in X'_R\cap Y'_R\right\}$. A straight-forward argument shows that for such $g$ we have $\beta'_R(\alpha'_R(g,x),x)=g$ and $T^g x = S^{\alpha'_R(g,x)}x$ (see \cite[p. 121]{Aus-16} for details). Thus, $\alpha'_R(\cdot, x)$ is injective.

Consider $A=X'_R\cap Y'_R$, and let $\mathcal{R}_A$ be the common equivalence relation on $A$ induced by the actions $T$ of $G$ and $S$ of $H$. Note that elements of $\mathcal{R}_A$ can (uniquely) be written in two different ways: $(a,T^ga)$, and $(a,S^ha)$, where $a\in A$, $g\in G$ and $h=\alpha'_R(g,a)\in H$. This gives two measures $\nu_A$ and $\nu'_A$ on $\mathcal{R}_A$, defined as the restrictions of $\mu_A\otimes \lambda_G$ and $\mu_A\otimes \lambda_H$. 
We claim that these two measures coincide, which implies that the image of $\nu_A$ under the map $(a,g)\to (a,\alpha'_R(g,a))$ equals $\nu'_A$, and therefore that $\alpha'_R(\cdot, a)$ induces an isomorphism of measure spaces from $\{g\in G\mid T^ga\in A\}$ to $\{h\in H\mid S^ha\in A\}$, respectively equipped with $\lambda_G$ and $\lambda_H$.
We turn to the proof of the claim. For this, consider the two maps $(a,T^ga)\mapsto g^{-1}\cdot T^{g}a$, and $(a,S^ha)\mapsto h\cdot a$ from $\mathcal{R}_A$ to $\Omega$. Note that both maps induce isomorphisms of measure spaces respectively from $(\mathcal{R}_A,\nu_A)$ and $(\mathcal{R}_A,\nu'_A)$ to their images in $\Omega$. But these two maps actually coincide: indeed, since $T^{g}a\in X'_R$ we have 
\[
g^{-1}\cdot T^{g} a=g^{-1} \cdot \alpha_R'(g,a)\cdot g\cdot a=\alpha_R'(g,a)\cdot a = h\cdot a.
\]
So the claim is proved. 

The same statement holds by symmetry for $\beta_R'$.
Setting $X'=X'_R$ and $Y'=Y'_R$ completes the proof of the proposition.
\end{proof}

By \Cref{prop:intersecting-fundamental-domains} we may now assume that the fundamental domains $X$ and $Y$ are chosen such that the intersection $A:=X\cap Y$ has positive measure, the cocycles $\alpha$ and $\beta$ are $L^1$-integrable and their restrictions to $A$ are measure preserving in the sense stated in the proposition. We will use this as standing assumption for the remainder of this appendix.

\subsection{Bowen's result for locally compact groups}

For a locally compact Polish group $G$ equipped with the Haar measure, we denote by ${\rm gr}_G$ its growth function. With the tailored choices of fundamental domains $X$ and $Y$ provided by \Cref{prop:intersecting-fundamental-domains} at hand, there is a straight-forward adaptation of Bowen's proof of \cite[Theorem B.10]{Aus-16} to locally compact Polish groups, yielding the following result. 

Given two increasing functions $f,g:\R_+\to \R_+$, we write $f\preceq g$ if there exists a constant $C\geq 1$ such that $f(t)\leq Cg(Ct)+C$, and $f\asymp g$ if $f\preceq g\preceq f$. Recall that for a locally compact compactly generated group $G$, the growth function which associates to $r\in \R_+$ the volume of the ball of radius $r$ is invariant under a change of compact generating set up to $\asymp$-equivalence of functions. The volume growth $V_G$ of a locally compact compactly generated group $G$ is the $\asymp$-equivalence class of its growth function with respect to any compact generating set.
\begin{theorem}\label{thm:Bowen-lc}
    Let $G$ and $H$ be $L^1$-measure equivalent compactly generated locally compact Polish groups. Then $V_G(n) \asymp V_H(n)$.
\end{theorem}
By replacing each group by a direct product with $S^1$, we can reduce to the case when both groups are non-discrete.
Most of Bowen's arguments generalize immediately to the non-discrete case with the following modifications:
\begin{enumerate}
    \item we replace the counting measures on $G$ and $H$ by the Haar measures on $G$ and $H$ everywhere;
    \item we replace finite generating sets by compact generating sets $S_G$ and $S_H$ and choices of finite subsets of $G$ (resp. $H$) by choices of subsets of uniformly bounded word metric (e.g. for the set $W$ in the proof of \cite[Theorem B.9]{Aus-16});
    \item we use that by \cite[Appendix A.2]{BadFurSau-13} there are constants $a, A > 0$ such that for every $g\in G$
    \[
        \int_X |\alpha(g,x)|_H d\mu_X(x) \leq A\cdot |g|_G. 
    \]
\end{enumerate}

Indeed, one readily verifies that the only place where an additional argument is required to pass from the discrete to the non-discrete case, is the first paragraph of the proof of \cite[Theorem B.6]{Aus-16}. Here Bowen argues that there is a measurable $C$-to-one map $\psi : X\to Y$ of the form $\psi(x)=\phi(x)\cdot x$ for some measurable map $\phi: X\to G\times H$. The argument given in Bowen uses implicitly that the groups under consideration are countable. However, one can extend his argument to locally compact Polish groups by reducing to a countable dense subgroup, as we shall now explain.
    
\begin{proposition}\label{prop:Austin-B6}
    Let $G \acts (X,\mu)$ be an ergodic pmp Standard Borel action of a locally compact Polish group and let $A\subset X$ be a set of positive measure. Then there is a measurable map $\chi: X \to G$ and a positive integer $C\geq 1$ such that the map $\psi : X \to A$, $x\mapsto T^{\chi(x)}\cdot x$ is well-defined, measurable and at most $C$-to-one. 
\end{proposition}
\begin{proof}
    The proposition is well-known if $G$ is discrete and countable. Moreover, being Polish, $G$ admits a dense countable subgroup, which by the following lemma acts ergodically on $X$. This proves the proposition. 
\end{proof}
 
\begin{lemma}
    Let $G\acts (X,\mu)$ be an ergodic action of a locally compact Polish group on a finite Standard Borel space. If $\Gamma$ is a dense subgroup of $G$, then the action of $\Gamma$ is also ergodic.
\end{lemma}
\begin{proof}
It is enough to see that the action of $G$ induces a norm-continuous action $G\acts L^1(X,\mu)$, which is well-known (see for instance \cite[Corollary II.1.2]{Ornstein-Weiss}).   
\end{proof}

To construct the $C$-to-one map $\psi$ required in \cite[Theorem B.6]{Aus-16}, we choose $\phi(x)=(\chi(x),\alpha(\chi(x),x))$, where $\chi$ is as in \Cref{prop:Austin-B6}. Indeed, by the definition of the cocycle $\alpha$ we have $T^{\chi(x)}\cdot x=\phi(x)\cdot x=\psi(x)$, implying that $\psi$ is $C$-to-one by \Cref{prop:Austin-B6}. The constant $C$ obtained from our proof is bounded above by $\left\lceil \frac{\mu_X(X)}{\mu_X(A)}\right\rceil$. By choosing $\phi$ differently, we could also recover Bowen's upper bound of $\left\lceil \frac{\mu_X(X)}{\mu_Y(Y)}\right\rceil$.

As we already mentioned, the remainder of Bowen's proof now adapts readily to the locally compact case following the aforementioned modifications. We thus obtain \Cref{thm:Bowen-lc}.

\subsection{Asymmetric quantitative version of Bowen's theorem}
In \cite[Theorem 3.2]{DKLMT-22}, a slight variant of Bowen's proof is exploited to yield a very general monotonicity result regarding the volume growth. For our classification of groups of polynomial growth up to $L^p$-ME for $p\leq 1$, we need the following more specific version of that statement.

\begin{theorem}\label{thm:BowenAsym}
		Let  $0<p\leq 1$ and let $G$ and $H$ be compactly generated locally compact Polish groups. Suppose there exists a ME coupling between $G$ and $H$ such that the cocycle $\alpha:G\times X_H\to H$ is $L^p$-integrable, then 
		\[V_G(n)\preceq V_H(n^{1/p}).\]
	\end{theorem}
The result in \cite{DKLMT-22} is again only stated for the finitely generated case. Its proof generalises directly to the above version for compactly generated locally compact Polish groups with the same modifications as for the proof of Bowen's theorem. Thus, we don't give further details here. Moreover, as we already mentioned, Correia and Paucar \cite{CorPau-25} will give a detailed proof of an even more general version of \cite[Theorem 3.2]{DKLMT-22} for locally compact groups in a forthcoming paper.

\subsection{Austin's Theorem for locally compact groups}\label{app:Austin-lc}

As in the case of Bowen's proof of integrable measure invariance of growth, the proof of Austin's main theorem in \cite{Aus-16} adapts readily to the locally compact case after choosing the fundamental domains and cocycles as in \Cref{prop:intersecting-fundamental-domains}. This yields the following version of \cite[Theorem 1.1]{Aus-16}.

\begin{theorem}\label{thm:Austins-lc}
    Let $G$ and $H$ be simply connected locally compact or finitely generated nilpotent groups which are integrably measure equivalent. Then there is a bilipschitz bijection between their asymptotic cones, or, equivalently, an isomorphism $\gr(G)\cong \gr(H)$ between their associated Carnot graded groups.
\end{theorem}

The starting point of Austin's proof of \Cref{thm:Austins-lc} for the finitely generated case is that one can choose the fundamental domains $X$ and $Y$ so that they intersect in a set of positive measure. In the non-discrete case a choice of fundamental domains with this property is provided by \Cref{prop:intersecting-fundamental-domains}. In \cite[Section 2.1]{Aus-16}, Austin uses this observation to extend the definition of the sets
\[
    D_x:= \left\{g\in G \mid T^gx\in X\cap Y\right\} \mbox{ for } x\in X
\]
and
\[
    E_y:= \left\{h\in H\mid S^hy \in X\cap Y\right\} \mbox{ for } y\in Y
\]
to arbitrary elements $x\in X\cup Y$. To do so for $D_x$ he uses that by ergodicity for almost every $y\in Y$ there is some $k\in H$ with $S^ky\in X\cap Y$ and then defines $D_y:= D_{S^ky} \cdot \beta(k,y)$. Similarly for $x\in X$ and $\ell\in G$ with $T^\ell x\in X\cap Y$ he defines $E_x:= E_{T^\ell x}\cdot \alpha(\ell,x)$. He then proceeds to extend the cocycles $\alpha$ and $\beta$ to $X\cup Y$ by using the above choices of $k$ and $\ell$ to define\footnote{There seems to be a typo in the definition of the extensions $\alpha(g,y)$ and $\beta(h,x)$ on p.122 of \cite{Aus-16}, which we corrected in the formulas given here.}
\[
    \alpha(g,y):= \alpha(g\beta(k,y)^{-1},S^ky)\cdot k \mbox{ and } \beta(h,x):= \beta(h\alpha(\ell,x)^{-1},T^\ell x) \cdot \ell.
\]
Austin then asserts that for every $x\in X\cup Y$ the maps $\alpha_x: g\mapsto \alpha(g,x)$ and $\beta_x : h\mapsto \beta(h,x)$ define inverse bijections $\alpha_x : D_x \to E_x$ and $\beta_x : E_x \to D_x$. Checking the well-definedness of all these definitions and their asserted properties is straight-forward, but a bit lengthy. The properties do not rely on discreteness and thus also remain true in the locally compact case. Moreover, by \Cref{prop:intersecting-fundamental-domains} for $x\in X\cap Y$ the map $\alpha_x : D_ x \to E_x$ is an isomorphism of measure spaces with respect to the respective restrictions of the Haar measures. Since left and right multiplication by a fixed element of $G$ (resp. $H$) is also an isomorphism of measure spaces, one deduces that in our setting for all $x\in X\cup Y$ the maps $\alpha_x: D_x \to E_x$ and $\beta_x : E_x\to D_x$ are inverse isomorphisms of measure spaces with respect to the respective restrictions of Haar measures.

With this at hand and the following modifications, the remainder of Austin's arguments in his proof of \Cref{thm:Austins-lc} for finitely generated groups generalise directly to the non-discrete case:
\begin{enumerate}
    \item we replace the counting measures on $G$ and $H$ by the Haar measures on $G$ and $H$ everywhere;
    \item we replace finite generating sets by compact generating sets $S_G$ and $S_H$ and choices of finite subsets of $G$ (resp. $H$) by choices of subsets of uniformly bounded word metric (e.g. for the set $W$ in the proof of \cite[Theorem B.9]{Aus-16});
    \item we use that by \cite[Appendix A.2]{BadFurSau-13} there is a constant $A > 0$ such that for every $g\in G$
    \[
        \int_X |\alpha(g,x)|_H d\mu_X(x) \leq A\cdot |g|_G. 
    \]
    \item We use that Pansu's Theorem holds in the locally compact case by work of Breuillard \cite{Breu-14}.
    \item We use that analogous versions of all results for finitely generated nilpotent groups used in Austin's proof hold for simply connected nilpotent Lie groups.
\end{enumerate}

\begin{remark}
    A substantial part (\cite[Section 3]{Aus-16}) of Austin's proof adapts directly to simply connected nilpotent Lie groups without requiring the material of Section \ref{section:crosssections} (nor Bowen's theorem, which also uses this material). The main ingredient in Austin's proof which really relies on that material is \cite[Proposition 4.4]{Aus-16}.
\end{remark}

\subsection{Cantrell's result for locally compact groups}\label{app:Cantrell}

Similar to Bowen's and Austin's results, the proof of \cite[Theorem B]{Cantrell} adapts readily to the locally compact case, yielding the following result.
\begin{theorem}\label{thm:Cantrell-lc}
    Let $G$ and $H$ be integrably measure equivalent simply connected nilpotent Lie groups with associated cocycles $\alpha: G\times X \to H$ and $\beta : H\times Y \to G$. Then there is a bilipschitz group isomorphism $\Phi : \gr(G)\to \gr(H)$ so that for all $g\in \gr(G)$ we have that
$\frac{1}{n}\bullet g_n \to g$ implies $\frac{1}{n}\bullet \alpha(g_n,x) \to \Phi(g)$ with high probability in $x\in X$. Similarly, $\beta$ induces an isomorphism $\Psi : \gr(H)\to \gr(G)$. Moreover, $\Psi=\Phi^{-1}$.
\end{theorem}

Cantrell's proof of \Cref{thm:Cantrell-lc} in the finitely generated case has two steps. 

The first step is \cite[Proposition 5.4]{Cantrell}, which says that the rescaled cocycle  $\alpha$ converges in probability to an actual morphism $\phi$ between the Carnot-graded associated groups. As for \cite[Section 3]{Aus-16}, this part only exploits the integrability of the cocycle $\alpha$, and can be extended to simply connected nilpotent Lie groups, without the need of Section \ref{section:crosssections}. A key ingredient in this step is \cite[Proposition 3.1]{Cantrell}. Some parts of the proof of the latter result in \cite{Cantrell} can be simplified, as we will explain now. We denote $\alpha_{ab}(g,x):=\pi_{ab}(\alpha(g,x))$, where we identify $G$ with $\mathfrak{g}$ as before, fix a decomposition $\mathfrak{g}=\oplus_{i=1}^s A_i$ with $A_i\subset \gamma_{i}(\mathfrak{g})\setminus \gamma_{i+1}(\mathfrak{g})$, and define $\pi_{ab}: \mathfrak{g}\to A_1$ as the projection. We further define the projection $\pi_{com}: \mathfrak{g}\to \oplus _{i=2}^s A_i = \gamma_2(\mathfrak{g})$, the (abelian) average $\overline{\alpha_{ab}}(g):= \int_{X} \alpha_{ab}(g,x)d\mu_X(x)$ and denote by $d_1$ the metric on $\mathfrak{g}$ induced by the 1-norm $|\cdot |_1$ with respect to a fixed identification with $\mathfrak{g}=\oplus_{i=1}^s A_i\cong\R ^d$ that respects the decomposition.

\begin{proposition}[{\cite[Proposition 3.1]{Cantrell}}]\label{prop:convergence-to-abelian-average}
    Let $G$ be a simply connected nilpotent Lie group and let $g\in G$. Then $\frac{1}{n}\bullet \alpha(g^n,x)\to \overline{\alpha_{ab}} (g)$ in probability as $n\to \infty$. 
\end{proposition}

To prove \Cref{prop:convergence-to-abelian-average}, Cantrell shows that $|\pi_{com}(\alpha(g^n,x))|_H=o(n)$ in probability (see \cite[Proposition 3.4]{Cantrell}). Cantrell's proof of this is by induction and relies on several preliminary results and the cocycle identity. One can avoid the induction in his proof by replacing his \cite[Lemma 3.8]{Cantrell} by the following stronger result.
\begin{lemma}\label{lem:commutator-contributions-are-small}
    Let $G$ be a simply connected nilpotent Lie group, let $v\in A_1$ and let $M>0$. Then for all $\delta>0$ there exists some $\delta'>0$, $K\in N$ and $C>0$ so that for all $n\geq K$, $\eta \geq 1$ the following holds:

    If $g_1, \ldots, g_n \in G$ are elements so that, for all $i$,
    \begin{enumerate}
        \item $d_1(\pi_{ab}(g_i),\eta v) < \eta \delta' |v|_1$, and
        \item $|\pi_{com} (g_i)|_G < \eta M$,
    \end{enumerate}
    then $|\pi_{com}(g_1\cdot  \ldots \cdot g_n)|_G < \eta n \delta + C$.
\end{lemma}

In his proof of \cite[Lemma 3.8]{Cantrell}, Cantrell implicitly relies on a statement like the following one, which can be seen as a generalization of the Baker--Campbell--Hausdorff (short: BCH) formula for a product of a large number of group elements. Here we identify $G$ with $\mathfrak{g}$ to define a multiplicative group structure on $\mathfrak{g}$, as before.

\begin{proposition}\label{prop:iterated-BCH}
    There exists a sequence $(\alpha_k)$ of positive integers such that the following holds. For all nilpotent simply connected Lie groups $G$, and all sequences $g_1,\ldots,g_n\in \mathfrak{g}$, we have
    \[
        g_1\cdot g_2\ldots \cdot g_n=\sum_{k\geq 1} A_{k,n}(g_1,\ldots, g_n), 
    \]
where $A_{k,n}(g_1,\ldots, g_n)$ is a linear combination of $\leq \alpha_k n^k$ commutators of the form $[g_{i_1},\ldots,g_{i_k}]$, whose coefficients are rational numbers of absolute value $\leq 1$ (by convention $[g_j]=g_j$).
\end{proposition}
\begin{proof}

    Observe that for $n=2$ this is the BCH formula. We could not locate a proof of this result for general $n$ in \cite{Cantrell} or other literature and hence include it here. Roughly speaking, the proof follows by iteratively applying the BCH formula. 

    For all $k$, and $l\leq q$, we let $\beta_{q,l}$ be the number of $q$-iterated commutators appearing in the BCH formula for the product $x\cdot y$, where $y$ appears $l$ times. 
    
    For $k=1$, we simply define $A_{1,n}=g_1+\ldots g_n$, so $\alpha_1=1$. We now claim that for each $k\geq 2$,  
    \[
        \alpha_k=\sum_{q=1}^{k}\sum_{l=1}^{q-1} \beta_{q,l} \sum_{j_1+\ldots +j_l= k-q+l}\prod_{i=1}^l \alpha_{j_i}
    \] 
    satisfies the condition of the theorem. The origin of this rather complicated formula will become clear in the sequel.

    To prove the claim, we argue by induction on $n$, the case $n=1$ being obvious. We will make use of the following elementary inequality: for all $x\in \mathbb{R}$, and $k\geq 1$,
    \[
        (x-1)^k\leq x^k-k(x-1)^{k-1}\leq x^k-(x-1)^{k-1}.
    \]
    Indeed, by convexity of $f(x)=x^k$, we have $f(x)-f(x-1)\geq f'(x-1)$, which rewritten as $f(x-1)\leq f(x)-f'(x-1)$ is exactly the first inequality above.

    By our induction assumption, $g_1\ldots g_{n-1}=\sum_{j\geq 1} A_{j,n-1}(g_1,\ldots, g_{n-1})$. We now apply the BCH formula to the product $(g_1\ldots g_{n-1})\cdot g_n$. This gives us a sum of iterated commutators in the two ``letters'' $g_1\ldots g_{n-1}$ and $g_n$. After replacing $g_1\ldots g_{n-1}$ by its expression $\sum_{j\geq 1} A_{j,n-1}(g_1,\ldots, g_{n-1})$, we are left with a linear combination of a sum of iterated commutators in the $g_i$ (whose coefficients are rationals of absolute value $\leq 1$). To estimate the number of $k$-iterated commutators in the $g_i$, we observe that those come from two contributions.

    The first one is $A_{k,n-1}(g_1,\ldots, g_{n-1})$ (coming from the non-commutator term in the BCH formula), which contributes for at most $\alpha_k(n-1)^k\leq \alpha_k(n^k-(n-1)^{k-1})$ $k$-iterated commutators. 
    
    The other contribution consists of all the iterated commutators from the BCH formula (by which we mean in the letters $g_1\ldots g_{n-1}$ and $g_n$). Note that those involve the letter $g_n$ at least once. Moreover, among those, the ones which contribute to $k$-iterated commutators in the $g_i$'s must  have length $q\leq k$. If we fix the number of occurrences of $g_n$ to be $q-l\geq 1$, then in order to obtain a $k$-iterated commutator in the $g_i$, this forces the other $l$ letters $(a_1,\ldots, a_l)$ to be an element of $\prod_{i=1}^l A_{k_i}$ for a choice of $k_1,\ldots, k_l$ such that $\sum_i k_i=k-q+l$ (which is $\leq k-1$). In order to count these terms, we first fix the positions of the letter $g_n$ (giving rise to a factor $\beta_{q,l}$), then we pick a sequence $k_1,\ldots, k_l$. The number of possibilities after such a choice is at most $\prod_{i=1}^l (\alpha_{j_i}(n-1)^{j_i})=(n-1)^{k-q+l}\prod_{i=1}^l \alpha_{j_i}\leq (n-1)^{k-1}\prod_{i=1}^l \alpha_{j_i}$.
    
    Hence the number of $k$-iterated commutators in the $g_i$ arising from the second contribution is at most $\alpha_k (n-1)^{k-1}$.
    We deduce that the total number of $k$-iterated commutators is at most 
    \[
        \alpha_k(n^k-(n-1)^{k-1})+\alpha_k (n-1)^{k-1}=\alpha_k n^k.
    \]
    Hence the claim is proved and the proposition follows.
\end{proof}

\begin{proof}[Proof of \Cref{lem:commutator-contributions-are-small}]
    By \Cref{prop:iterated-BCH} there is a decomposition of the form
    \[
        g:= g_1\cdot g_2\ldots \cdot g_n=\sum_{k\geq 1} A_{k,n}(g_1,\ldots, g_n)
    \]
    with $A_{k,n}(g_1,\ldots,g_n)$ a linear combination (with coefficients of absolute value at most $1$) of at most $\alpha_k n^k$ commutators of the form 
    \[ 
        [g_{i_1},\ldots,g_{i_k}]=[\pi_{ab}(g_{i_1})+ \pi_{com}(g_{i_1}), \ldots, \pi_{ab}(g_{i_k})+\pi_{com}(g_{i_k})]
    \]
    By multilinearity, the right side further decomposes into a sum of $2^k$ commutators of the form $[v_1,\ldots, v_k]$ with $v_j\in\left\{\pi_{ab}(g_{i_j}), \pi_{com}(g_{i_j})\right\}$. We determine their contributions to $\pi_{com}(g_1\ldots g_n)$, distinguishing two cases.
    
    First, if $v_j=\pi_{ab}(g_{i_j})$ for all $j$, then we observe that $[v_1,\ldots, v_k]\in \gamma_k(\mathfrak{g})$ only contributes to $\pi_{com}(g)$ if $k\geq 2$. We can then use assumption (1) to argue as in Case 3 of the proof of \cite[Lemma 3.8]{Cantrell} that if $k\geq 2$, then the Guivarc'h norm of $[v_1,\ldots, v_k]$ is bounded above by $C \eta \delta' |v|_1$ for a suitable constant $C>0$. Since $[v_1,\ldots,v_k]\in \gamma_k(\mathfrak{g})$, the contribution of such terms coming from all summands of $A_{k,n}(g_1, \ldots, g_n)$ to the Guivarc'h norm of $\pi_{com}(g)$ is $\leq C \eta n \delta' |v|_1$. Choosing $\delta'\in (0,1)$ sufficiently small as a function of $C$, $\delta$ and $|v|_1$ completes the proof for this case.

    Second, if there are $s\geq 1$ terms $v_j$ with $v_j=\pi_{com}(g_{i_j})$, then $[v_1,\ldots, v_k]\in \gamma_{k+s}(\mathfrak{g})$ and the Guivarc'h norm of $[v_1,\ldots, v_k]$ is thus $\leq C (\delta'\eta |v|_1)^{\frac{k-s}{k+s}} (M\eta)^{\frac{s}{k+s}}\leq C |v|_1^{\frac{k-s}{k+s}} M^{\frac{s}{k+s}}\eta^{\frac{k}{k+s}}$. Since for every $s$ there are $\leq 2^k \alpha_k n^k$ such summands in $A_{k,n}(g_1, \ldots, g_n)$, their total contribution to the Guivarc'h norm of $\pi_{com}(g)$ is $\leq C \alpha_k^{\frac{1}{k+s}} 2^{\frac{k}{k+s}} |v|_1^{\frac{k-s}{k+s}} M^{\frac{s}{k+s}} \eta^{\frac{k}{k+s}} n^{\frac{k}{k+s}}$. For $K$ sufficiently large, the latter is $\leq \eta \delta n$ for all $n\geq K$.

    Using the equivalence of the Guivarc'h norm and the word metric $|\cdot |_G$ on $G$, we conclude that there is a $\delta'>0$, $K>0$ and $C>0$ such that $|\pi_{com}(g_1\ldots g_n)|_G<\eta n \delta + C$ for all $n\geq K$, $\eta \geq 1$ and all elements $g_1,\ldots, g_n\in G$ satisfying conditions (1) and (2).
\end{proof}

To prove \cite[Proposition 5.4]{Cantrell}, Cantrell extends \Cref{prop:convergence-to-abelian-average} to limits of the form $\frac{1}{n}\bullet \alpha(g_n,x)$ with $g_n=s_1^{a_{n,1}}\ldots s_k^{a_{n,k}}$, where the $a_{n,i}\to \infty$ (see \cite[Theorem 4.1]{Cantrell}) and uses these limits to define the morphism $\phi$ (see \cite[Section 5]{Cantrell}). His proofs of these results adapt readily to the case of simply connected nilpotent Lie groups.

The second step in Cantrell's argument is \cite[Proposition 5.7.]{Cantrell}. It consists of proving that the limiting morphism $\phi$ is, in fact, an isomorphism whose inverse is the rescaled limit of $\beta$. 
It seems to us that his proof of this contains a gap in the published version of his article.\footnote{The set of positive measure of $y\in X\cap Y$ such that Equation (5.3) holds in the first paragraph of Cantrell's proof of \cite[Proposition 5.7]{Cantrell} depends on $x$. In particular, there is no a priori reason why for a given choice of $x$ one can choose $y=x$. This subtlety is explained on \cite[p. 134]{Aus-16}. Cantrell's proof seems to require that such a choice is possible with positive probability in $X\cap Y$. However, he did not explain if and how this subtlety can be resolved in his situation.}
However, he gives a correct argument in an earlier version (the first arXiv version), which is based on Austin's \cite[Theorem 4.1]{Aus-16}. Roughly the idea goes as follows: once he has established that after rescaling, $\alpha$ converges to a morphism $\varphi$, Austin's \cite[Theorem 4.1]{Aus-16} implies that the limit morphism must be Lipschitz, co-Lipschitz and have dense image, hence surjective (and therefore injective as well).

It turns out though that the proof of Cantrell's \cite[Proposition 5.7.]{Cantrell} can also be corrected in a more direct way by using Austin's \cite[Proposition 4.4]{Aus-16} as follows (whose proof is a bit easier than that of the full \cite[Theorem 4.1]{Aus-16}). 
Having established \cite[Proposition 5.4]{Cantrell}, we know that there exist continuous morphisms $\phi:\gr{G}\to \gr{H}$ and $\psi:\gr{H}\to \gr{G}$ such that for all sequences $g_n\in G$ such that $1/n\bullet g_n\to \bar{g}\in \gr{G}$, the sequence $1/n\bullet \alpha(g_n,x)$ converges in probability to $\phi(\Bar{g})\in \gr{H}$.
Similarly, for all sequence $h_n\in H$ such that  $1/n\bullet h_n\to \bar{h}\in \gr{H}$, the sequence $1/n\bullet \beta(h_n,x)$ converges in probability to $\psi(\Bar{h})\in \gr{G}$.
We claim that $\psi\circ \phi=id$ (by symmetry, this will prove that they are mutual inverses).

To see this, choose $h_n$ such that $1/n\bullet h_n\to \phi(\bar{g})$ (which we denote by $\bar{h}$). In particular, we have that $d_H(h_n,\alpha(g_n,x))=o(n)$ with high probability. But now, \cite[Proposition 4.4]{Aus-16} implies that $d_G(\beta(h_n,x),g_n)=o(n)$ with high probability. Therefore, $1/n\bullet \beta(h_n,x)$ converges with high probability to the same limit as $1/n\bullet g_n$. In other words, we have $\psi ( \phi(\bar{g}))=\psi(\bar{h})=\bar{g}$, and the claim is proved.

\frenchspacing
\bibliography{References}

\bibliographystyle{alpha}

\end{document}